\documentclass{article}[11pt]
\usepackage[a4paper,top=4.0cm,bottom=2.0cm,left=2.0cm,right=2.0cm]{geometry}
\usepackage{harvard_acta}
\usepackage{amssymb}
\usepackage{amsfonts}
\usepackage{amsmath,mathrsfs,amsthm}
\usepackage{bbold,bm}
\usepackage{graphicx}
\newcommand{\ie}{{\it i.e.}}
\newcommand{\etc}{{\it  etc.}}

\newcommand{\dd}{{\rm d}}
\newcommand{\T}{{\rm T}}

\newcommand{\ignore}[1]{}

\def\R#1{$(\ref{#1})$} 	
\def\Frac#1#2{{\textstyle \frac{#1}{#2}}} 

\newtheorem{theorem}{Theorem}[section]
\newtheorem{lemma}[theorem]{Lemma}
\newtheorem{remark}[theorem]{Remark}
\newtheorem{definition}[theorem]{Definition}
\newtheorem{example}[theorem]{Example}
\newtheorem{problem}[theorem]{Problem}

\begin{document}

\title{Data Assimilation: The Schr\"odinger Perspective}

\author{Sebastian Reich\thanks{%
 Department of Mathematics, University of Potsdam \& University of Reading,
 {\tt sebastian.reich@uni-potsdam.de}
}}

\maketitle

\begin{abstract}
Data assimilation addresses the general problem of how to combine model-based predictions with partial and noisy observations 
of the process in an optimal manner. This survey focuses on sequential data assimilation techniques using probabilistic 
particle-based algorithms.  In addition to surveying recent developments for discrete- and continuous-time data assimilation, both
in terms of mathematical foundations and algorithmic implementations, we also provide a unifying framework from the perspective 
of coupling of measures, and Schr\"odinger's boundary value problem 
for stochastic processes in particular. 
\end{abstract}

%
\section{Introduction}
%

This survey focuses on sequential data assimilation techniques for state and parameter estimation
in the context of discrete- and continuous-time stochastic diffusion processes. See Figure \ref{figure:DA1}. 
The field itself is well established  \cite{sr:evensen,sr:sarkka,sr:stuart15,sr:reichcotter15,sr:ABN16}, but is also undergoing 
continuous development due to new challenges arising from emerging application areas such as medicine, traffic control, 
biology, cognitive sciences 
and geosciences.

\begin{figure}
\begin{center}
\includegraphics[width=0.8\textwidth,trim = 0 100 0 20,clip]{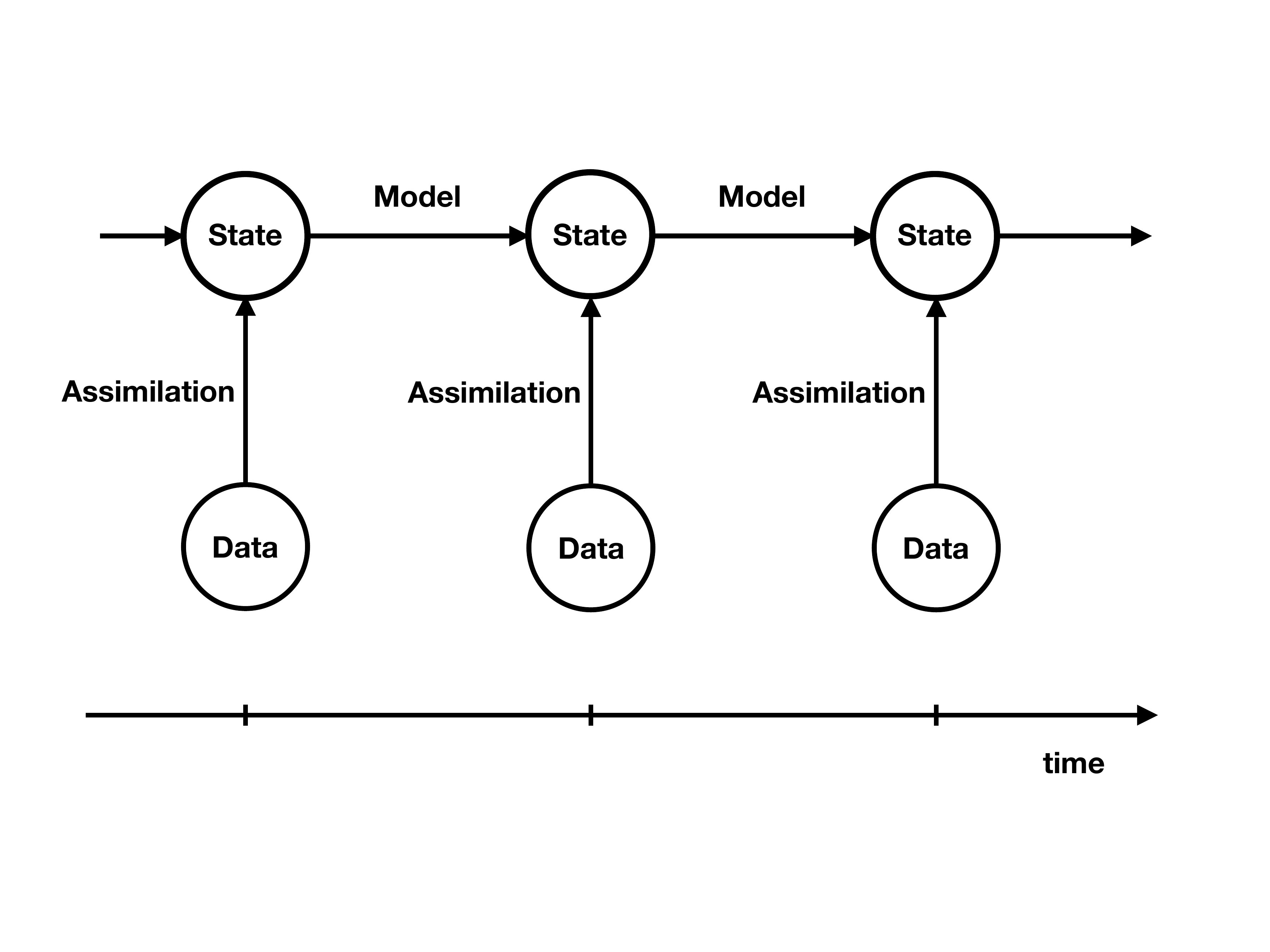}
\end{center}
\caption{Schematic illustration of sequential data assimilation, where model states are propagated forward in time under a given model
dynamics and adjusted whenever data become available at discrete instances in time. In this paper, we look at a single transition from
a given model state conditioned on all the previous and current data to the next instance in time, and its adjustment under the 
assimilation of the new data then becoming available.}
\label{figure:DA1}
\end{figure}

Data assimilation is typically formulated within a Bayesian framework in order to combine partial and noisy observations with 
model predictions and their uncertainties  with the goal of adjusting model states and model parameters in an optimal
manner.  In the case of linear systems and Gaussian distributions, this task leads to the celebrated Kalman filter 
\cite{sr:sarkka} which even today forms the basis of a number of popular data assimilation schemes and which has given 
rise to the widely used ensemble Kalman filter \cite{sr:evensen}. Contrary to standard sequential Monte Carlo methods 
\cite{sr:Doucet,sr:crisan}, the ensemble Kalman filter does not provide a consistent approximation to the sequential filtering 
problem, while being applicable to very high-dimensional problems. This and other advances have widened the scope of
sequential data assimilation and have led to an avalanche of new methods in recent years.

In this review we will focus on probabilistic methods (in contrast to data assimilation techniques based on optimisation, 
such as 3DVar and  4DVar \cite{sr:evensen,sr:stuart15}) in the form of sequential particle methods. The essential challenge of sequential
particle methods is to convert a sample of $M$ particles from a filtering distribution at time $t_k$ into
$M$ samples from the filtering distribution at time $t_{k+1}$ without having access to
the full filtering distributions. It will also often be the case in practical applications that the sample
size will be small to moderate in comparison to the number of variables we need to estimate. 

Sequential particle methods can be viewed as a special instance of interacting particle systems \cite{sr:DelMoral}.
We will view such interacting particle systems in this review from the perspective of approximating a certain boundary value 
problem in the space of probability measures, where the boundary conditions are provided by the underlying stochastic process, 
the data, and Bayes' theorem. This point of view leads naturally to optimal transportation \cite{sr:Villani,sr:reichcotter15} 
and, more importantly for this review, to Schr\"odinger's problem \cite{sr:FG97,sr:L14,sr:CGP14}, 
as formulated first by Erwin Schr\"odinger in the form of a boundary value problem for Brownian motion  \cite{sr:S31}.

This paper has been written with the intention of presenting a unifying framework for sequential data assimilation using
coupling of measure arguments provided through optimal transportation and Schr\"odinger's problem. We will also summarise novel
algorithmic developments that were inspired by this perspective. Both discrete- and continuous-time processes and data sets will
be covered. While the primary focus is on state estimation, the presented material can be extended to combined state and parameter
estimation. See Remark \ref{rem:state_estimation} below.

\begin{remark}  We will primary refer to the methods considered in the survey as 
particle or ensemble methods instead of the widely used notion of sequential Monte Carlo
methods. We will also use the notions of particles, samples and ensemble members
synonymously. Since the ensemble size, $M$, is generally assumed to be small to moderate relative
to the number of variables of interest, we will focus on robust but generally biased particle 
methods.
\end{remark}

\medskip

\subsection{Overall organisation of the paper}
This survey consists of four main parts. We start Section \ref{sec:mf_dtDA} by recalling key mathematical concepts of sequential data assimilation 
when the data become available at discrete instances in time. Here the underlying dynamic models can be either 
continuous (that is, is generated by a stochastic differential equation) or discrete-in-time. Our initial review of the problem 
will lead to the identification of three different scenarios of performing sequential data assimilation, which we denote by (A), (B) and (C). 
While the first two scenarios are linked to the classical importance resampling and optimal proposal densities for particle filtering \cite{sr:Doucet}, 
scenario (C) builds upon an intrinsic connection to a certain boundary value problem in the space of joint probability measures first considered 
by Erwin Schr\"odinger \cite{sr:S31}. 

After this initial review, the remaining parts of Section \ref{sec:mf_dtDA} provide more mathematical details on prediction in Section \ref{sec:dp}, filtering and smoothing in Section \ref{sec:filtering_and_smoothing}, and the Schr\"odinger approach to sequential data assimilation in Section \ref{sec:SP}. 
The modification of a given Markov transition kernel via a twisting function will arise as a crucial mathematical construction and will be introduced 
in Sections \ref{sec:SEN} and \ref{sec:dp}. 
The next major part of the paper, Section \ref{sec:numerics}, is devoted to numerical implementations of prediction, filtering and
smoothing, and the Schr\"odinger approach as relevant to scenarios (A)--(C) introduced earlier in Section
\ref{sec:mf_dtDA}. More specifically, this part will cover the ensemble Kalman filter and 
its extensions to the more general class of linear ensemble transform filters as well as the numerical implementation of the Schr\"odinger approach to
sequential data assimilation using the Sinkhorn algorithm \cite{sr:S67,sr:PC18}. Discrete-time stochastic systems with additive Gaussian 
model errors and stochastic differential equations with constant diffusion coefficient serve as illustrating examples throughout both Sections 
\ref{sec:mf_dtDA} and \ref{sec:numerics}.

Sections \ref{sec:mf_dtDA} and \ref{sec:numerics} are followed by two sections on the assimilation of data that arrive continuously in time.
In Section \ref{sec:mf_ct_DA} we will distinguish between data that are smooth 
as a function of time and data which have been perturbed by Brownian motion. In both cases, we will demonstrate that
the data assimilation problem can be reformulated in terms of so-called mean-field equations, 
which produce the correct conditional marginal distributions in the state variables.  In particular, in Section \ref{sec:random_data} we discuss the
feedback particle filter of \citeasnoun{sr:meyn13}  in some detail.
The final section of this review, Section \ref{sec:NM4}, covers numerical approximations to these mean-field equations in the form of 
interacting particle systems. More specifically, the continuous-time ensemble Kalman--Bucy
and numerical implementations of the feedback particle filter will be covered in detail. It will be shown in particular that the 
numerical implementation of the feedback particle filter can be achieved naturally via the approximation of an associated 
Schr\"odinger problem using the Sinkhorn algorithm.

In the appendices we provide additional background material on mesh-free approximations of the Fokker--Planck and backward Kolmogorov 
equations (Appendix A), on the regularised St\"ormer--Verlet time-stepping methods for the hybrid Monte Carlo method, 
applicable to Bayesian inference problems over path spaces (Appendix B), on the ensemble Kalman filter (Appendix C), 
and on the numerical approximation of forward--backward stochastic differential equations (SDEs) (Appendix D).


\subsection{Summary of essential notations} \label{sec:SEN}

We typically denote the probability density function (PDF) of a random variable $Z$ by $\pi$. Realisations of
$Z$ will be denoted by $z = Z(\omega)$. 

Realisations of a random variable can also be continuous functions/paths,
in which case the associated probability measure on path space is denoted by $\mathbb{Q}$. We will primarily consider
continuous functions over the unit time interval and denote the associated random variable by $Z_{[0,1]}$ and its realisations $Z_{[0,1]}(\omega)$ 
by $z_{[0,t]}$. The restriction of $Z_{[0,1]}$ to a particular instance $t \in [0,1]$ is denoted by $Z_t$ with marginal distribution $\pi_t$
and realisations $z_t = Z_t(\omega)$.

For a random variable $Z$ having only finitely many
outcomes $z^i$, $i=1,\ldots,M$, with probabilities $p_i$, that is,
\begin{equation*}
\mathbb{P}[Z(\omega)=z^i] = p_i,
\end{equation*}
we will work with either the probability vector $p = (p_1,\ldots,p_M)^\T$
or the empirical measure
\begin{equation*}
\pi(z) = \sum_{i=1}^M p_i \,\delta(z-z^i),
\end{equation*}
where $\delta (\cdot)$ denotes the standard Dirac delta function. 

We use the shorthand 
\begin{equation*}
\pi [f] = \int f(z)\,\pi(z)\,\dd z
\end{equation*}
for the expectation of a function $f$ under a PDF $\pi$. Similarly, integration with respect to a probability measure
$\mathbb{Q}$, not necessarily absolutely continuous with respect to Lebesgue, will be denoted by
\begin{equation*}
\mathbb{Q}[f] = \int f(z) \,\mathbb{Q}(\dd z).
\end{equation*}
The notation $\mathbb{E}[f]$ is used if we do not wish to specify the measure explicitly. 

The PDF of a Gaussian random variable, $Z$, with mean $\bar z$ and covariance matrix $B$ will be abbreviated by ${\rm n}(z;\bar z,B)$.
We also write $Z \sim {\rm N}(\bar z,B)$.

Let $u \in \mathbb{R}^N$, then $D(u) \in \mathbb{R}^{N\times N}$ denotes the diagonal matrix
with entries $(D(u))_{ii} = u_i$, $i=1,\ldots,N$. We also denote the $N\times 1$ vector of ones by
$\mathbb{1}_N = (1,\ldots,1)^\T \in \mathbb{R}^N$.

A matrix $P \in \mathbb{R}^{L\times M}$ is called bi-stochastic if all its entries are non-negativ, which we will abbreviate by $P\ge 0$, and 
\begin{equation*}
\sum_{l=1}^L q_{li} = p_0, \qquad \sum_{i=1}^M q_{li} = p_1,
\end{equation*}
where both $p_1 \in \mathbb{R}^L$ and $p_0 \in \mathbb{R}^M$ 
are probability vectors. A matrix $Q \in \mathbb{R}^{M\times M}$ defines a discrete Markov chain if all its entries are 
non-negative and
\begin{equation*}
\sum_{l=1}^L q_{li} = 1.
\end{equation*}
The Kullback--Leibler divergence between two bi-stochastic matrices $P \in \mathbb{R}^{L\times M}$ and $Q \in \mathbb{R}^{L\times M}$ 
is defined by
\begin{equation*}
{\rm KL}\,(P||Q) := \sum_{l,j} p_{lj} \log\frac{p_{lj}}{q_{lj}}.
\end{equation*}
Here we have assumed for simplicity that $q_{lj}>0$ for all entries of $Q$. This definition extends to the Kullback--Leibler divergence between 
two discrete Markov chains.

The transition probability going from state $z_0$ at time $t=0$ to state $z_1$ at time $t=1$ 
is denoted by $q_+(z_1|z_0)$. Hence, given an initial PDF $\pi_0(z_0)$ at $t=0$, 
the resulting (prediction or forecast) PDF at time $t=1$ is provided by
\begin{equation} \label{eq:pi1}
\pi_1(z_1) := \int q_+(z_1|z_0) \,\pi_0(z_0)\,\dd z_0.
\end{equation}
Given a twisting function $\psi(z)>0$, the twisted transition kernel 
$q_+^\psi(z_1|z_0)$ is defined by
\begin{equation} \label{eq:q-twisted}
q^\psi_+(z_1|z_0) := \psi(z_1)\,q_+(z_1|z_0)\,\widehat{\psi}(z_0)^{-1}
\end{equation}
provided
\begin{equation} \label{eq:psi_0}
\widehat{\psi}(z_0):= \int q_+(z_1|z_0)\,\psi(z_1)\,\dd z_1
\end{equation}
is non--zero for all $z_0$. See Definition \ref{def:twisted_kernel} for more details. 

If transitions are characterised by a discrete Markov chain $Q_+ \in \mathbb{R}^M$, then a twisted Markov chain
is provided by
\begin{equation*}
Q_+^u = D(u)\, Q_+ \,D(v)^{-1}
\end{equation*}
for given twisting vector $u\in \mathbb{R}^M$ with positive entries $u_i$, that is, $u>0$, and the vector 
$v \in \mathbb{R}^M$ determined by
\begin{equation*}
v = (D(u)\,Q_+)^\T \,\mathbb{1}_M.
\end{equation*}

The conditional probability of observing $y$ given $z$ is denoted by $\pi(y|z)$ and the likelihood of $z$ given an observed $y$ is abbreviated by
$l(z) = \pi(y|z)$. We will also use the abbreviations
\begin{equation*}
\widehat{\pi}_1(z_1) = \pi_1(z_1|y_1)
\end{equation*}
and 
\begin{equation*}
\widehat{\pi}_0(z_0) = \pi_0(z_0|y_1)
\end{equation*}
to denote the conditional PDFs of a process at time $t=1$ given data at time $t=1$ (filtering) and the conditional PDF at time $t=0$ given data
at time $t=1$ (smoothing), respectively. Finally, we also introduce the evidence 
\begin{equation*}
\beta := \pi_1[l] = \int p(y_1|z_1)\pi_1(z_1)\dd z_1
\end{equation*} 
of observing $y_1$ under the given model as represented by the forecast PDF \R{eq:pi1}. A more precise definition of these expressions 
will be given in the following section.

%
\section{Mathematical foundation of discrete-time DA} \label{sec:mf_dtDA}
%

Let us assume that we are given partial and noisy observations $y_{k}$, $k=1,\ldots,K,$ 
of a  stochastic process in regular time intervals of length $T=1$. Given a likelihood
function $\pi(y|z)$, a Markov transition kernel $q_+(z'|z)$ and an initial distribution
$\pi_0$, the associated prior and posterior PDFs are given by
\begin{equation} \label{eq:Prediction}
\pi(z_{0:K}) := \pi_0(z_0) \prod_{k=1}^K q_+(z_{k}|z_{k-1})
\end{equation}
and
\begin{equation} \label{eq:Smoothing}
\pi(z_{0:K}|y_{1:K}) := \frac{\pi_0(z_0) \prod_{k=1}^K \pi(y_k|z_k)\,q_+(z_k|z_{k-1})}{\pi(y_{1:K})},
\end{equation}
respectively \cite{sr:jazwinski,sr:sarkka}. While it is of broad interest to approximate the posterior or smoothing
PDF \R{eq:Smoothing}, we will focus  on the recursive approximation
of the filtering PDFs $\pi(z_k | y_{1:k})$ using sequential 
particle filters in this paper. More specifically, we wish to address the following
computational task.

\medskip

\begin{problem} \label{ps:problem1}
We have $M$ equally weighted Monte Carlo samples $z_{k-1}^i$, $i=1,\ldots,M$, from the filtering 
PDF $\pi(z_{k-1}|y_{1:k-1})$ at time $t=k-1$ available and we wish to produce $M$ equally 
weighted samples from the filtering PDF  $\pi(z_k|y_{1:k})$ at time $t=k$ having access to the transition 
kernel $q_+(z_k|z_{k-1})$ and the likelihood $\pi(y_k|z_k)$ only. Since the computational task is 
exactly the same for all indices $k\ge 1$, we simply set $k=1$ throughout this paper.
\end{problem}

\medskip

\noindent
We introduce some notations before we discuss several possibilities of addressing Problem \ref{ps:problem1}. 
Since we do not have direct access to the filtering distribution at time $k=0$, the PDF at $t_0$ becomes
\begin{equation} \label{eq:initial_pdf}
\pi_0(z_0) := \frac{1}{M} \sum_{i=1}^M \delta(z_0 - z_0^i),
\end{equation}
where $\delta(z)$ denotes the Dirac delta function and $z_0^i$, $i=1,\ldots,M$, are $M$ given Monte Carlo samples representing the
actual filtering distribution. Recall that we abbreviate the resulting filtering PDF $\pi(z_1|y_1)$ at $t=1$ by 
$\widehat{\pi}_1(z_1)$ and the likelihood $\pi(y_1|z_1)$ by $l(z_1)$. Because of \R{eq:pi1}, the forecast PDF is given by
\begin{equation} \label{eq:prediction}
\pi_1(z_1)  = \frac{1}{M}\sum_{i=1}^M q_+(z_1|z_0^i)
\end{equation}
and the filtering PDF at time $t=1$ by
\begin{equation} \label{eq:filter_PDF}
\widehat{\pi}_1(z_1) := \frac{l(z_1)\,\pi_1(z_1)}{\pi_1[l]} = \frac{1}{\pi_1[l]} 
\frac{1}{M} \sum_{i=1}^M l(z_1)\,q_+(z_1|z_0^i)
\end{equation}
according to Bayes' theorem.

\medskip

\begin{remark} \label{rem:state_estimation}
The normalisation constant $\pi(y_{1:K})$ in \R{eq:Smoothing}, also called the evidence, can be determined recursively using
\begin{align} \nonumber
\pi(y_{1:k}) &= \pi(y_{1:k-1})\,\int \pi(y_k,z_{k-1})\,\pi(z_{k-1}|y_{1:{k-1}}) \dd z_{k-1}\\ \nonumber
&= \pi(y_{1:k-1})\,\int \int \pi(y_k|z_{k})\,q_+(z_k|z_{k-1})\,\pi(z_{k-1}|y_{1:{k-1}}) \dd z_{k-1} \dd z_k\\
&= \pi(y_{1:k-1})\,\int  \pi(y_k|z_k)\,\pi(z_{k}|y_{1:{k-1}}) \dd z_k \label{eq:recursive_evidence}
\end{align}
\cite{sr:sarkka,sr:reichcotter15}. Since, as for the state estimation problem, the computational task is the same
for each index $k\ge 1$, we simply set $k=1$ and formally use $\pi(y_{1:0}) \equiv 1$.  We are then left with
\begin{equation} \label{eq:evidence}
\beta := \pi_1[l] = \frac{1}{M} \sum_{i=1}^M \int l(z_1)\,q_+(z_1|z_0^i) \,\dd z_1
\end{equation}
within the setting of Problem \ref{ps:problem1} and  $\beta$ is a shorthand for $\pi(y_1)$.
If the model depends on parameters, $\lambda$, or different models are to be compared, then it is important to evaluate 
the evidence \R{eq:evidence} for each parameter value $\lambda$ or model, respectively. More specifically, if $q_+(z_1|z_0;\lambda)$, 
then $\beta = \beta (\lambda)$ in \R{eq:evidence} and larger values of $\beta(\lambda)$ indicate a better fit of the transition kernel 
to the data for that parameter value. One can then perform Bayesian parameter inference based upon  appropriate 
approximations to the likelihood $\pi(y_1|\lambda) = \beta (\lambda)$ and a given prior PDF $\pi(\lambda)$. The extension to the complete 
data set $y_{1:K}$, $K>1$, is straightforward using \R{eq:recursive_evidence} and an appropriate data assimilation algorithm, that is,
algorithms that can tackle problem \ref{ps:problem1} sequentially.

Alternatively,  one can treat a combined state--parameter estimation problem as a particular case of problem \ref{ps:problem1} by introducing the
extended state variable $(z,\lambda)$ and augmented transition probabilities $Z_1 \sim q_+(\cdot|z_0,\lambda_0)$ and 
$\mathbb{P}[\Lambda_1 = \lambda_0] = 1$. The state augmentation technique allows one to extend all approaches discussed in 
this paper for Problem \ref{ps:problem1} to combined state--parameter estimation.

See \citeasnoun{sr:KDSMC15} for a detailed survey of the topic of combined state and parameter estimation.
\end{remark}

\medskip

\noindent
The filtering distribution $\widehat{\pi}_1$ at time $t=1$ implies a smoothing distribution at time $t=0$, which is given by
\begin{equation} \label{eq:smoothing_PDF0}
\widehat{\pi}_0(z_0) := \frac{1}{\beta} \int l(z_1) \,q_+(z_1|z_0)\,\pi_0(z_0)\,\dd z_1 = \frac{1}{M} \sum_{i=1}^M \gamma^i 
\delta(z_0-z_0^i)
\end{equation}
with weights 
\begin{equation} \label{eq:gamma_i}
\gamma^i := \frac{1}{\beta}  \int l(z_1) \,q_+(z_1|z_0^i)\,\dd z_1.
\end{equation}
It is important to note that the filtering PDF $\widehat{\pi}_1$ can be obtained from $\widehat{\pi}_0$ using the
transition kernels
\begin{equation} \label{eq:smoothing_kernel}
\widehat{q}_+(z_1|z_0^i) := \frac{l(z_1)\,q_+(z_1|z_0^i)}{\beta\,\gamma^i},
\end{equation}
that is,
\begin{equation*}
\widehat{\pi}_1(z_1) = \frac{1}{M} \sum_{i=1}^M \widehat{q}_+(z_1|z_0^i)\,\gamma^i .
\end{equation*}
See Figure \ref{fig:overview2} for a schematic illustration of these distributions and their mutual relationships. 

\medskip
\begin{remark} The modified transition kernel \R{eq:smoothing_kernel} can be seen as a particular instance
of a twisted transition kernel \R{eq:q-twisted} with $\psi (z) = l(z)/\beta$ and $\widehat{\psi}(z_0^i) = \gamma^i$.
Such twisting kernels will play a prominent role in this survey, not only in the context of optimal proposals 
\cite{sr:Doucet,sr:arul02} but also in the context of the Schr\"odinger approach to data assimilation, 
that is, to scenario (C) below.
\end{remark}

\medskip

\noindent
The following scenarios of how to tackle Problem \ref{ps:problem1}, that is, how to
produce the desired samples $\widehat{z}_1^i$, $i=1,\ldots,M$, 
from the filtering PDF \R{eq:filter_PDF}, will be considered in this paper.

\medskip

\begin{figure}[h]
\begin{center}
\includegraphics[width=0.8\textwidth,trim = 0 0 20 0,clip]{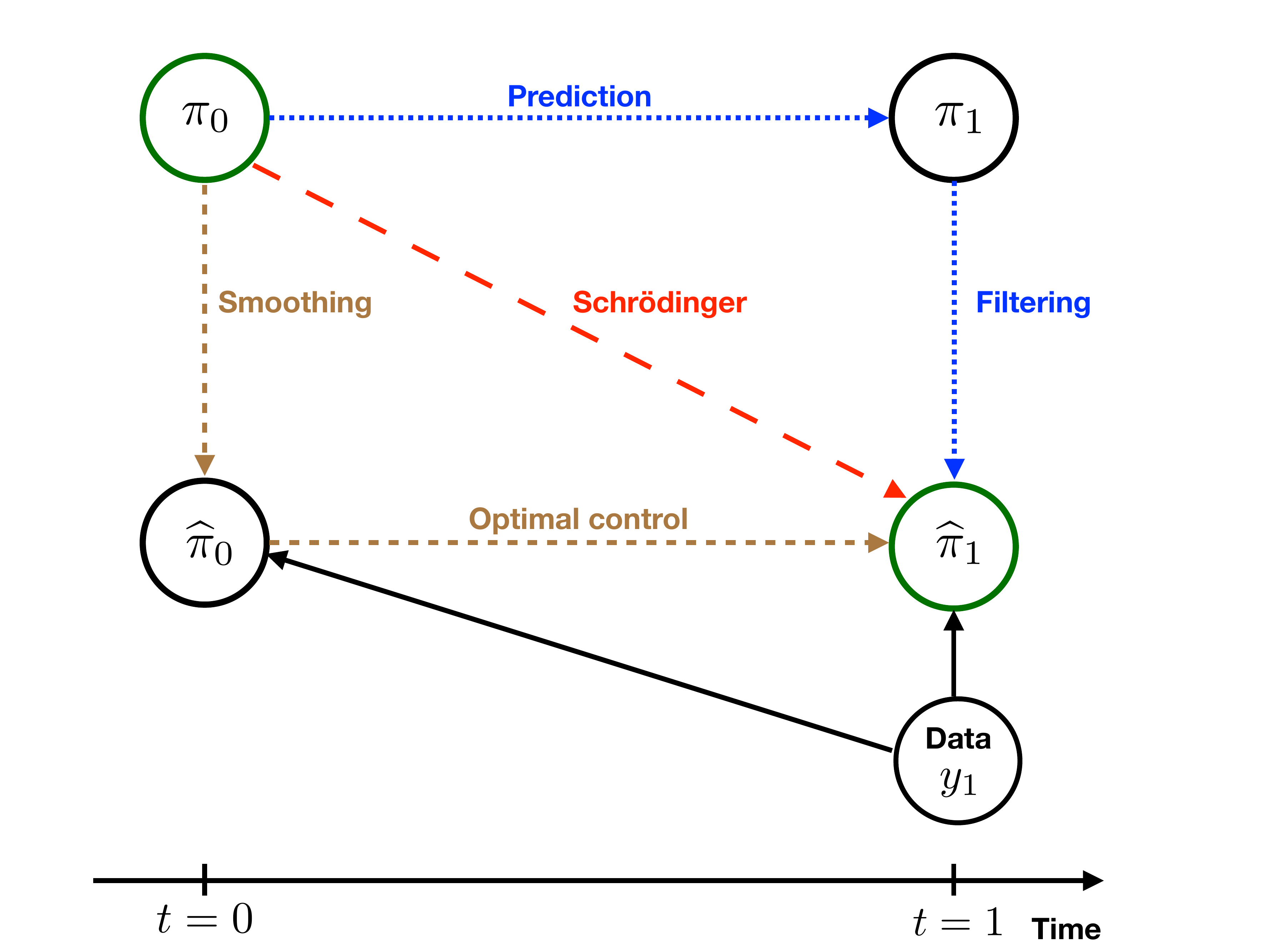}
\end{center}
\caption{Schematic illustration of a single data assimilation cycle. The distribution $\pi_0$ characterises the distribution of
states conditioned on all observations up to and including $t_0$, which we set here to $t=0$ for simplicity. The predictive distribution 
at time $t_1=1$, as generated by the model dynamics, is denoted by $\pi_1$. Upon assimilation of the data $y_1$ 
and application of Bayes' formula, one obtains the
filtering distribution $\widehat{\pi}_1$. The conditional distribution of states at time $t_0$ conditioned on all the available data
including $y_1$ is denoted by $\widehat{\pi}_0$. Control theory provides the adjusted model dynamics for
transforming $\widehat{\pi}_0$ into $\widehat{\pi}_1$. Finally, the Schr\"odinger problem links $\pi_0$ and $\widehat{\pi}_1$
in the form of a penalised boundary value problem in the space of joint probability measures. 
Data assimilation scenario (A) corresponds to the dotted lines, scenario (B) to the short-dashed lines, and scenario (C) to the long-dashed line.}
\label{fig:overview2}
\end{figure}

\begin{definition} \label{ps:scenarios}
We define the following three scenarios of how to tackle Problem \ref{ps:problem1}.

\begin{itemize}
\item[(A)] We first produces samples, $z_1^i$, from the forecast PDF $\pi_1$ and then transform those samples into samples,
$\widehat{z}_1^i$, from $\widehat{\pi}_1$. This can be viewed as introducing a Markov transition kernel $q_1(\widehat{z}_1|z_1)$ 
with the property that
\begin{equation} \label{eq:transform_step}
\widehat{\pi}_1(\widehat{z}_1) = \int q_1(\widehat{z}_1|z_1)\,\pi_1(z_1)\,\dd z_1.
\end{equation}
Techniques from optimal transportation can be used  to find appropriate 
transition kernels \cite{sr:Villani,sr:Villani2,sr:reichcotter15}. 

\medskip

\item[(B)] We first produce $M$ samples from the smoothing PDF \R{eq:smoothing_PDF0} via resampling with replacement 
and then sample from $\widehat{\pi}_1$ using the smoothing transition kernels \R{eq:smoothing_kernel}. 
The resampling can be represented in terms of a Markov transition matrix $Q_0 \in \mathbb{R}^{M\times M}$ such that
\begin{equation*}
\gamma = Q_0 \,p.
\end{equation*}
Here we have introduced the associated probability vectors 
\begin{equation} \label{eq:definition_eM}
\gamma = \left(\Frac{\gamma^1}{M},\ldots,\Frac{\gamma^M}{M}\right)^\T \in \mathbb{R}^M\,,\qquad  p = 
\left(\Frac{1}{M},\ldots,\Frac{1}{M} \right)^\T
\in \mathbb{R}^M .
\end{equation}
Techniques from optimal transport will be explored to find such Markov transition matrices in Section \ref{sec:numerics}.

\medskip

\item[(C)] We directly seek Markov transition kernels $q^\ast_+(z_1|z_0^i)$, $i=1,\ldots,M$, with the property that
\begin{equation} \label{eq:Schroedinger_kernel}
\widehat{\pi}_1(z_1) = \frac{1}{M} \sum_{i=1}^M q^\ast_+(z_1|z_0^i)
\end{equation}
and then draw a single sample, $\widehat{z}_1^i$, from each kernel $q^\ast_+(z_1|z_0^i)$. Such kernels 
can be found by solving a  Schr\"odinger problem \cite{sr:L14,sr:CGP14} as demonstrated in Section \ref{sec:SP}. 
\end{itemize}

\end{definition}

\medskip

\noindent
Scenario (A) forms the basis of the classical bootstrap particle filter \cite{sr:Doucet,sr:Liu,sr:crisan,sr:arul02} 
and also provides the starting point for many currently 
used ensemble-based data assimilation algorithms \cite{sr:evensen,sr:reichcotter15,sr:stuart15}. 
Scenario (B) is also well known in the context of particle filters under the notion of  
optimal proposal densities \cite{sr:Doucet,sr:arul02,sr:FK18}. Recently there has been a renewed interest in scenario
(B) from the perspective of optimal control and twisting approaches \cite{sr:GJL17,sr:HBDD18,sr:KR16,sr:RK17}.
Finally, scenario (C) has not yet been explored in the context of particle filters and data assimilation, primarily because the required kernels
$q^\ast_+$ are typically not available in closed form or cannot be easily sampled from. However, as we will argue in this paper, 
progress on the numerical solution of Schr\"odinger's problem \cite{sr:cuturi13,sr:PC18} turns scenario (C) into a viable option 
in addition to providing a unifying mathematical framework for data assimilation. 

We emphasise that not all existing particle
methods fit into these three scenarios. For example, the methods put forward by \citeasnoun{sr:leeuwen15} are based on proposal densities 
which attempt to overcome limitations of scenario (B) and which lead to less variable particle weights, 
thus attempting to obtain particle filter implementations closer to what we denote here as scenario (C).
More broadly speaking, the exploration of alternative proposal densities in the context of data assimilation has started only recently. 
See, for example, \citeasnoun{sr:VEW12}, \citeasnoun{sr:MTAC12}, \citeasnoun{sr:leeuwen15}, \citeasnoun{sr:PLKBJ18}, and
 \citeasnoun{sr:vLKNPR18}.

The accuracy of an ensemble--based data assimilation method can be characterised in terms of its effective sample
size $M_{\rm eff}$ \cite{sr:Liu}. The relevant effective sample size for scenario (B) is, for example, given by
\begin{equation*}
M_{\rm eff} = \frac{M^2}{\sum_{i=1}^M (\gamma^i)^2} = \frac{1}{\|\gamma\|^2}.
\end{equation*}
We find that $M\ge M_{\rm eff} \ge 1$ and the accuracy of a data assimilation step decreases 
with decreasing $M_{\rm eff}$, that is, the convergence rate $1/\sqrt{M}$ of a standard Monte Carlo method is replaced
by $1/\sqrt{M_{\rm eff}}$ \cite{sr:APPSS17}. 
Scenario (C) offers a route around this problem by bridging $\pi_0$ with $\widehat{\pi}_1$ directly, that is, solving the Schr\"odinger problem
delivers the best possible proposal densities leading to equally weighted particles without the need for resampling.\footnote{The kernel 
\R{eq:smoothing_kernel} is called the optimal proposal in the particle filter community. However, the kernel 
\R{eq:smoothing_kernel} is suboptimal in the broader framework considered in this paper.}

\begin{figure}
\begin{center}
\includegraphics[width=0.8\textwidth,trim = 0 0 0 0,clip]{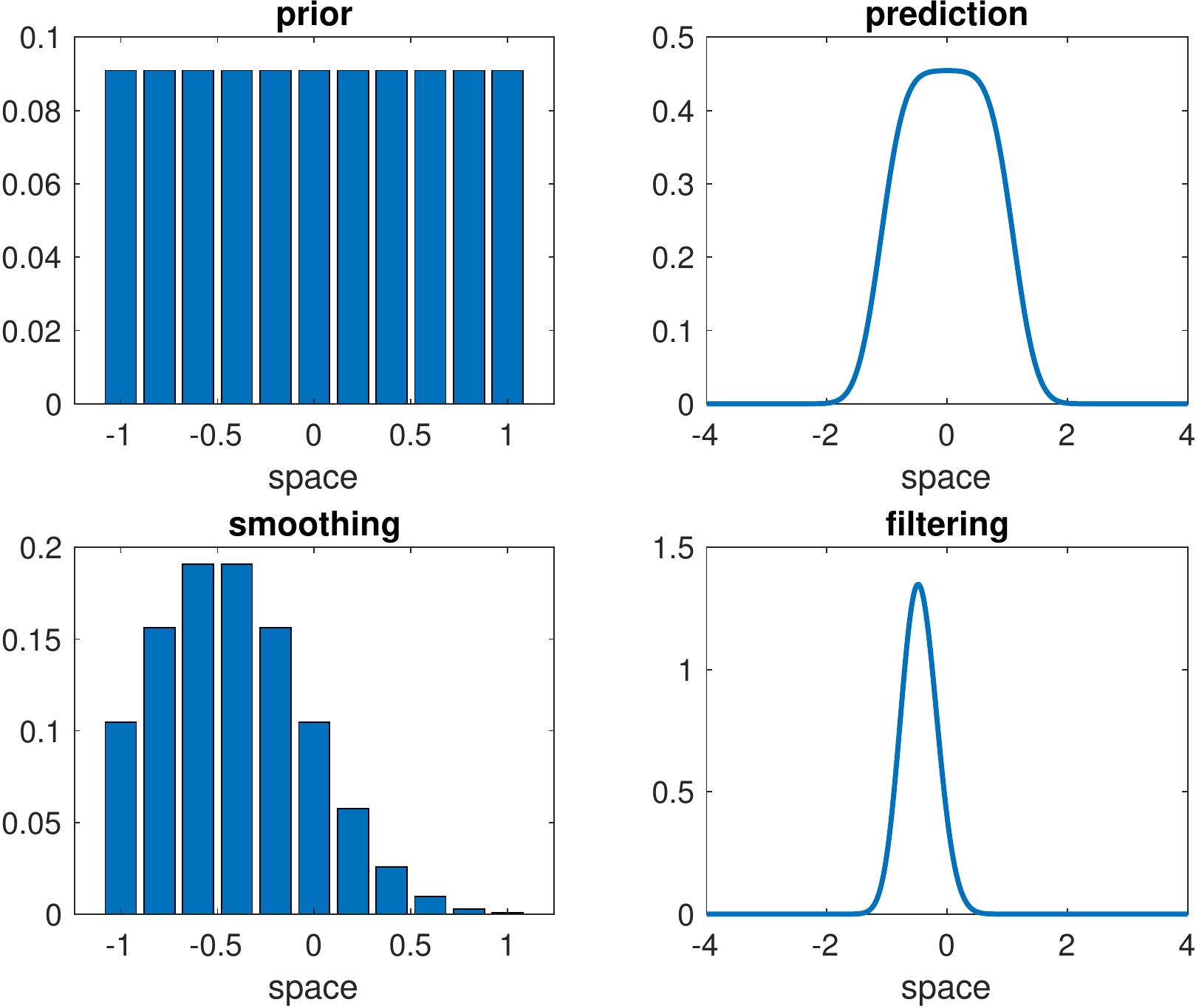}
\end{center}
\caption{The the initial PDF $\pi_0$, the forecast PDF $\pi_1$, the filtering PDF $\widehat{\pi}_1$, and
the smoothing PDF $\widehat{\pi}_0$ for a simple Gaussian transition kernel.}
\label{fig:example1a}
\end{figure}

\begin{example} \label{ex:example1}
We illustrate the three scenarios with a simple example. The prior samples are given by $M=11$ equally spaced
particles $z_0^i\in \mathbb{R}$ from the interval $[-1,1]$. The forecast PDF $\pi_1$ is provided by
\begin{equation*}
\pi_1(z) = \frac{1}{M} \sum_{i=1}^M \frac{1}{(2\pi)^{1/2}\sigma}\exp \left( -\Frac{1}{2\sigma^2} (z-z_0^i)^2\right)
\end{equation*}
with variance $\sigma^2 = 0.1$. The likelihood function is given by
\begin{equation*}
\pi(y_1|z) = \frac{1}{(2\pi R)^{1/2}}\exp \left(-\Frac{1}{2R} (y_1-z)^2 \right)
\end{equation*}
with $R = 0.1$ and $y_1 = -0.5$. The implied filtering and smoothing distributions can be found in 
Figure \ref{fig:example1a}. Since $\widehat{\pi}_1$ is in the form of a weighted Gaussian mixture distribution, the
Markov chain leading from $\widehat{\pi}_0$ to $\widehat{\pi}_1$ can be stated explicitly, that is, \R{eq:smoothing_kernel}
is provided by
\begin{equation} \label{eq:smoothing_example1}
\widehat{q}_+(z_1|z_0^i) = \frac{1}{ (2\pi)^{1/2} \widehat{\sigma}} \exp \left(- \Frac{1}{2\widehat{\sigma}^2} (
\bar z_1^i -z_1)^2 \right)
\end{equation}
with
\begin{equation*}
\widehat{\sigma}^2 = \sigma^2 - \frac{\sigma^4}{\sigma^2 + R}, \quad 
\bar z_1^i = z_0^i - \frac{\sigma^2}{\sigma^2 + R} (z_0^i - y_1)\,.
\end{equation*}
The resulting transition kernels are displayed in Figure \ref{fig:example1b} together with the corresponding 
transition kernels for the Schr\"odinger approach, which connects $\pi_0$ directly with $\widehat{\pi}_1$. 
\end{example}

\begin{figure}
\begin{center}
\includegraphics[width=0.45\textwidth,trim = 0 0 0 0,clip]{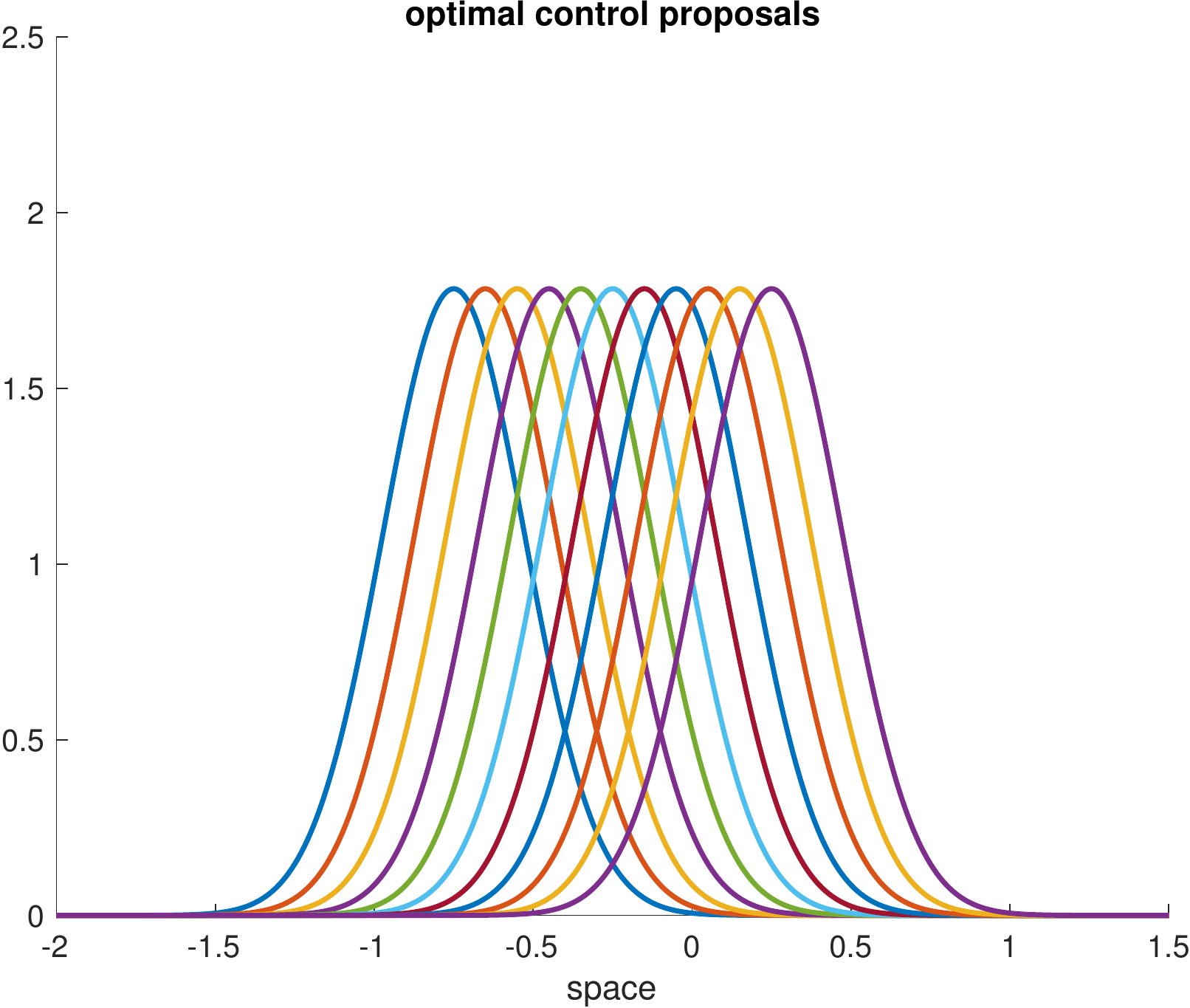} $\quad$
\includegraphics[width=0.45\textwidth,trim = 0 0 0 0,clip]{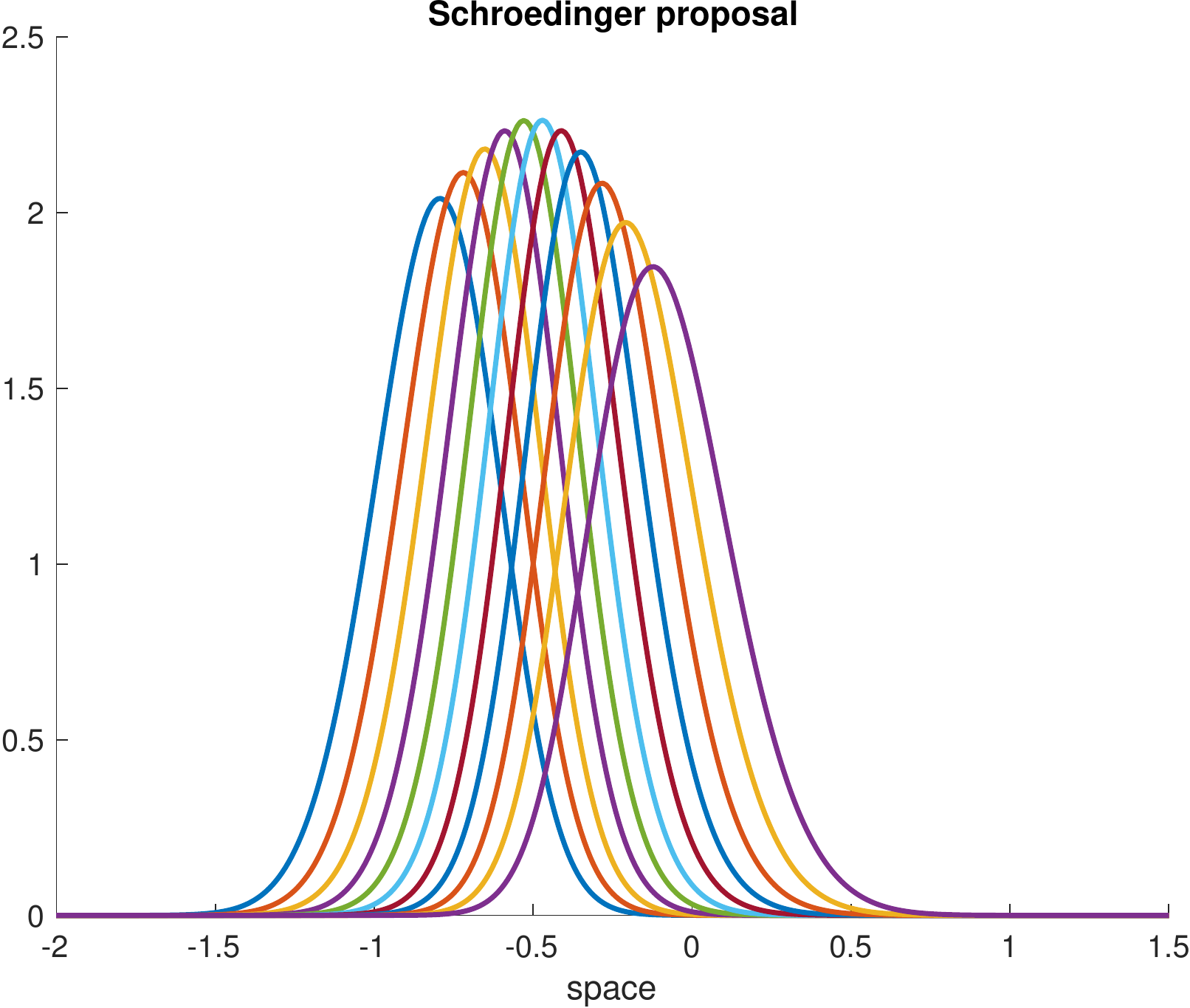}
\end{center}
\caption{Left panel: The transition kernels \R{eq:smoothing_example1} for the $M=11$ different
particles $z_0^i$. These correspond to the optimal control path in figure \ref{fig:overview2}. Right panel: The corresponding 
transitions kernels, which lead directly from $\pi_0$ to $\widehat{\pi}_1$. These
correspond to the Schr\"odinger path in figure \ref{fig:overview2}. Details of how to compute these Schr\"odinger 
transition kernels, $q_+^\ast(z_1|z_0^i)$, can be found in Section \ref{sec:Gaussian_Schrodinger}.}
\label{fig:example1b}
\end{figure}

\begin{remark}
It is often assumed in optimal control or rare event simulations arising from statistical mechanics that $\pi_0$ in \R{eq:Prediction} is a point measure, 
that is, the starting point of the simulation is known exactly. See, for example, \citeasnoun{sr:HRSZ17}. This corresponds to
\R{eq:initial_pdf} with $M=1$. It turns out that the associated smoothing problem becomes equivalent to Schr\"odinger's 
problem under this particular setting since the distribution at $t=0$ is fixed.
\end{remark}

\medskip

\noindent
The remainder of this section is structured as follows. We first recapitulate the pure prediction problem for discrete-time 
Markov processes and continuous-time diffusion processes, after which we discuss the filtering and smoothing problem for
a single data assimilation step as relevant for scenarios (A) and (B). The final subsection is devoted to the Schr\"odinger problem  
\cite{sr:L14,sr:CGP14} of bridging the filtering distribution, $\pi_0$, at $t=0$ directly with the  filtering distribution, $\widehat{\pi}_1$, 
at $t=1$, thus leading to scenario (C).

\subsection{Prediction} \label{sec:dp}
We assume under the chosen computational setting that we have access to $M$ samples $z_0^i \in \mathbb{R}^{N_z}$, 
$i=1,\ldots,M$, from the filtering distribution at $t=0$.  We also assume that we know (explicitly or implicitly) the 
forward transition probabilities, $q_+(z_1|z_0^i)$, of the underlying Markovian stochastic process. This leads to
the forecast PDF, $\pi_1$, as given by \R{eq:prediction}.

Before we consider two specific examples, we introduce two concepts related to the forward transition kernel which we will need later
in order to address scenarios (B) \& (C) from Definition \ref{ps:scenarios}.

We first introduce the backward transition kernel $q_-(z_0|z_1)$, which is defined 
via the equation
\begin{equation*} 
q_-(z_0|z_1) \,\pi_1 (z_1) = q_+(z_1|z_0)\,\pi_0(z_0).
\end{equation*}
Note that $q_-(z_0|z_1)$ as well as $\pi_0$ are not absolutely continuous with respect to the underlying Lebesque measure, that is,
\begin{equation} \label{eq:backward_transition_kernel}
q_-(z_0|z_1) = \frac{1}{M} \sum_{i=1}^M \frac{q_+(z_1|z_0^i)}{\pi_1(z_1)}\,\delta(z_0-z_0^i).
\end{equation}
The backward transition kernel $q_-(z_1|z_0)$ reverses the prediction process in the sense that
\begin{equation*}
\pi_0(z_0) = \int q_-(z_0|z_1)\,\pi_1(z_1)\,\dd z_1.
\end{equation*}

\medskip

\begin{remark} Let us assume that the detailed balance 
\begin{equation*}
q_+(z_1|z_0)\,\pi(z_0) = q_+(z_0|z_1)\,\pi(z_1)
\end{equation*}
holds for some PDF $\pi$ and forward transition kernel $q_+(z_1|z_0)$. Then $\pi_1 = \pi$ for $\pi_0 = \pi$ (invariance of $\pi$)
and $q_-(z_0|z_1) = q_+(z_1|z_0)$.
\end{remark}

\medskip

\noindent
We next introduce a class of forward transition kernels using the concept of twisting \cite{sr:GJL17,sr:HBDD18}, which is 
an application of Doob's H-transform technique \cite{sr:Doob84}. 
\medskip

\begin{definition} \label{def:twisted_kernel}
Given a non-negative twisting function $\psi(z_1)$ such that the modified transition kernel 
\R{eq:q-twisted} is well-defined, one can define the twisted forecast PDF
\begin{equation} \label{eq:twisted-predicted-PDF}
\pi_1^\psi (z_1) := \frac{1}{M} \sum_{i=1}^M q_+^\psi (z_1|z_0^i) = \frac{1}{M} \sum_{i=1}^M \frac{\psi(z_1)}{\widehat{\psi}(z_0^i)}\, q_+(z_1|z_0^i) . 
\end{equation}
The PDFs $\pi_1$ and $\pi_1^\psi$ are related by
\begin{equation} \label{eq:importance_proposal}
\frac{\pi_1(z_1)}{\pi_1^\psi (z_1)} =  \frac{\sum_{i=1}^M q_+(z_1|z_0^i)}{\sum_{i=1}^M 
\Frac{\psi(z_1)}{ \widehat{\psi}(z_0^i)}\,q_+(z_1|z_0^i)}\,.
\end{equation}
\end{definition}

\medskip

\noindent
Equation \R{eq:importance_proposal} gives rise to importance weights
\begin{equation} \label{eq:importance_weights00}
w^i \propto \frac{\pi_1(z_1^i)}{\pi_1^\psi (z_1^i)}
\end{equation}
for samples $z_1^i=Z_1^i(\omega)$ drawn from the twisted forecast PDF, that is,
\begin{equation*}
Z_1^i \sim q^\psi_+(\cdot \,|z_0^i)
\end{equation*}
and
\begin{equation*}
\pi_1(z) \approx \frac{1}{M} \sum_{i=1}^M w^i \,\delta(z-z_1^i)
\end{equation*}
in a weak sense. Here we have assumed that the normalisation constant in \R{eq:importance_weights00} is chosen such that
\begin{equation} \label{eq:normalised_w}
\sum_{i=1}^M w^i = M .
\end{equation}
Such twisted transition kernels will become important when looking at the filtering and smoothing as well as 
the Schr\"odinger problem later in this section.

Let us now discuss a couple of specific models which give rise to transition kernels $q_+(z_1|z_0)$. 
These models will be used throughout
this paper to illustrate mathematical and algorithmic concepts.

\subsubsection{Gaussian model error} \label{sec:Gauss1}

Let us consider the discrete-time stochastic process
\begin{equation} \label{eq:discrete_time_Gaussian}
Z_1 = \Psi(Z_0) + \gamma^{1/2}  \Xi_0
\end{equation}
for given map $\Psi:\mathbb{R}^{N_z} \to \mathbb{R}^{N_z}$,  
scaling factor $\gamma>0$, and  Gaussian distributed random variable $\Xi_0$ with mean zero and covariance matrix $B\in 
\mathbb{R}^{N_z\times N_z}$. The associated  forward transition kernel is given by
\begin{equation} \label{eq:Gaussian_kernel}
q_+(z_1|z_0) = {\rm n}(z_1;\Psi(z_0),\gamma B) .
\end{equation}
Recall that we have introduced the shorthand ${\rm n}(z;\bar z,P)$ for the 
PDF of a Gaussian random variable with mean $\bar z$ and covariance matrix $P$.

Let us consider a twisting potential $\psi$ of the form
\begin{equation*}
\psi(z_1) \propto \exp \left( -\frac{1}{2}(Hz_1-d)^\T R^{-1}(Hz_1 - d)\right)
\end{equation*}
for given $H\in \mathbb{R}^{N_z\times N_d}$, $d\in \mathbb{R}^{N_d}$, and covariance matrix $R\in
\mathbb{R}^{N_d\times N_d}$. We define
\begin{equation} \label{eq:Kalman_gain1}
K := BH^\T (HBH^\T +\gamma^{-1}R)^{-1}
\end{equation}
and 
\begin{equation} \label{eq:Bz}
\bar B := B - KHB, \qquad \bar z_1^i := \Psi(z_0^i) - K(H\Psi(z_0^i) - d).
\end{equation}
The twisted forward transition kernels are given by
\begin{equation*}
q^\psi_+(z_1|z_0^i) = {\rm n}(z_1;\bar z_1^i,\gamma \bar B)
\end{equation*}
and 
\begin{equation*}
\widehat{\psi}(z_0^i) \propto \exp \left(-\frac{1}{2}(H \Psi(z_0^i)-d)^\T (R+\gamma HBH^\T)^{-1} (H\Psi(z_0^i)-d)\right)
\end{equation*}
for $i=1,\ldots,M$.

\subsubsection{SDE models}
Consider the (forward) SDE \cite{sr:P14}
\begin{equation} \label{eq:Forward-SDE}
\dd Z_t^+ = f_t(Z_t^+) \,\dd t + \gamma^{1/2}  \,\dd W_t^+
\end{equation}
with initial condition $Z_0^+ = z_0$ and $\gamma >0$. Here $W_t^+$ stands for standard Brownian motion
in the sense that the distribution of $W_{t+\Delta t}^+$, $\Delta t>0$, conditioned on $w_t^+ = W_t^+(\omega)$ 
is Gaussian with mean $w_t^+$ and covariance matrix $\Delta t\,I$ \cite{sr:P14} and the process $Z_t^+$ is adapted to 
$W_t^+$.

The resulting time-$t$ transition kernels $q_t^+(z|z_0)$  from time zero to time $t$, $t\in (0,1]$,
satisfy the Fokker-Planck equation \cite{sr:P14}
\begin{equation*}
\partial_t q_t^+(\cdot\,|z_0) = -\nabla_{z} \cdot \left( q_t^+(\cdot\,|z_0) f_t \right) +  \Frac{\gamma}{2} \Delta_{z} q_t^+(\cdot\,|z_0)
\end{equation*}
with initial condition $q_0^+(z|z_0) = \delta(z-z_0)$, and the time-one forward transition kernel $q_+(z_1|z_0)$ is given by
\begin{equation*}
q_+ (z_1|z_0) = q_1^+(z_1|z_0)\,.
\end{equation*}

We introduce the operator ${\cal L}_t$ by
\begin{equation*}
{\cal L}_t g := \nabla_{z} g \cdot f_t  +  \Frac{\gamma}{2} \Delta_{z} g 
\end{equation*}
and its adjoint ${\cal L}_t^\dagger$ by
\begin{equation} \label{eq:FP_operator}
{\cal L}_t^\dagger \pi  := -\nabla_{z} \cdot \left( \pi\, f_t \right) +  \Frac{\gamma}{2} \Delta_{z} \pi 
\end{equation}
\cite{sr:P14}. We call ${\cal L}_t^\dagger$ the Fokker-Planck operator and ${\cal L}_t$ the generator of
the Markov process associated to the SDE \R{eq:Forward-SDE}.

Solutions (realisations) $z_{[0,1]}= Z^+_{[0,1]}(\omega)$ of the SDE \R{eq:Forward-SDE} with initial conditions 
drawn from $\pi_0$ are continuous functions
of time, that is, $z_{[0,1]} \in \mathcal{C} := C([0,1],\mathbb{R}^{N_z})$, and define a probability measure 
$\mathbb{Q}$ on $\mathcal{C}$, that is,
\begin{equation*}
Z^+_{[0,1]} \sim \mathbb{Q}.
\end{equation*}
We note that the marginal distributions $\pi_t$ of $\mathbb{Q}$, given by
\begin{equation*}
\pi_t(z_t) = \int q_t^+(z_t|z_0)\,\pi_0(z_0)\,\dd z_0 ,
\end{equation*}
also satisfy the Fokker-Planck equation, that is,
\begin{equation} \label{eq:FPE1}
\partial_t \pi_t = {\cal L}_t^\dagger\,\pi_t = -\nabla_{z} \cdot \left( \pi_t\, f_t \right) +  \Frac{\gamma}{2} \Delta_{z} \pi_t 
\end{equation}
for given PDF $\pi_0$ at time $t=0$.

Furthermore, we can rewrite the Fokker--Planck equation \R{eq:FPE1} in the form
\begin{equation}  
\partial \pi_t 
= -\nabla_{z} \cdot \left( \pi_t (f_t - \gamma  \nabla_z \log \pi_t) \right) -  \Frac{\gamma}{2} \Delta_z \pi_t ,
\label{eq:BFP}
\end{equation}
which allows us to read off from \R{eq:BFP} the backward SDE
\begin{align} \nonumber
\dd Z_t^- &= f_t(Z_t^-) \,\dd t  - \gamma \nabla_z \log \pi_t \,\dd t+  \gamma^{1/2}  \,\dd W_t^- ,\\
&= b_t(Z_t^-) \,\dd t  +  \gamma^{1/2}  \,\dd W_t^-  \label{eq:Backward-SDE}
\end{align}
with final condition $Z_1^- \sim \pi_1$, $W_t^-$ backward Brownian motion, and density-dependent drift term
\begin{equation*}
b_t(z) := f_t(z)   - \gamma \nabla_z \log \pi_t 
\end{equation*}
\cite{sr:N84,sr:CGP14}. Here backward Brownian motion is to be understood in the sense that
the distribution of $W_{t-\Delta \tau}^-$, $\Delta \tau>0$, conditioned on $w_t^- = W_t^-(\omega)$ is Gaussian with mean 
$w_t^-$ and covariance matrix $\Delta \tau\,I$ and all other properties of Brownian motion appropriately adjusted. 
The process $Z_t^-$ is adapted to $W_t^-$.

\medskip

\begin{lemma}
The backward SDE \R{eq:Backward-SDE} induces a corresponding backward transition kernel  from time one
to time $t=1-\tau$ with $\tau \in [0,1]$, denoted by $q_{\tau}^-(z|z_1)$, which satisfies the time-reversed Fokker--Planck equation
\begin{equation*}
\partial_\tau q_{\tau}^-(\cdot\,|z_1) = \nabla_{z} \cdot \left( q_{\tau}^-(\cdot\,|z_1) \,b_{1-\tau} \right) +  \Frac{\gamma}{2} \Delta_{z} q_{\tau}^-(\cdot\,|z_1)
\end{equation*}
with boundary condition $q_0^-(z|z_1) = \delta(z-z_1)$ at $\tau = 0$ (or, equivalently, at $t=1$). 
The induced backward transition kernel $q_-(z_0|z_1)$ is then given by
\begin{equation*}
q_-(z_0|z_1) = q^-_1(z_0|z_1)
\end{equation*}
and satisfies \R{eq:backward_transition_kernel}.
\end{lemma}

\begin{proof} The lemma follows from the fact that the backward SDE \R{eq:Backward-SDE} implies the Fokker--Planck equation
\R{eq:BFP} and that we have reversed time by introducing $\tau = 1-t$.
\end{proof}

\medskip

\begin{remark} The notion of a backward SDE also arises in a different context where the driving Brownian motion is
still adapted to the past, that is, $W_t^+$ in our notation, and a final condition is prescribed as for
\R{eq:Backward-SDE}. See \R{eq:BSDE} below as well as Appendix D and \citeasnoun{sr:C16} for more details.
\end{remark}

\medskip

\noindent
We note that the mean-field equation, 
\begin{equation} \label{eq:mODE}
\frac{\dd}{\dd t} z_t = f_t(z_t) - \Frac{\gamma}{2} \nabla_z \log \pi_t(z_t) = \frac{1}{2}(f_t(z_t)+b_t(z_t)),
\end{equation}
resulting from \R{eq:FPE1},
leads to the same marginal distributions $\pi_t$ as the forward and backward SDEs, respectively. 
It should be kept in mind, however, that the path measure generated by \R{eq:mODE} is different from the 
path measure $\mathbb{Q}$ generated by \R{eq:Forward-SDE}.

Please also note that the backward SDE and the mean field equation \R{eq:mODE} become singular as $t\to 0$ for the given 
initial PDF \R{eq:initial_pdf}. A meaningful solution can be defined via regularisation of the Dirac delta function, that is,
\begin{equation*}
\pi_0 (z) \approx \frac{1}{M} \sum_{i=1}^M {\rm n}(z;z_0^i,\epsilon I),
\end{equation*}
and taking the limit $\epsilon \to 0$. 

We will find later that it can be advantageous to modify the given SDE (\ref{eq:Forward-SDE}) by a time-dependent
drift term $u_t(z)$, that is,
\begin{equation} \label{eq:Forward-SDE2}
\dd Z_t^+ = f_t(Z_t^+) \,\dd t + u_t(Z_t^+)\,\dd t + \gamma^{1/2}  \,\dd W_t^+ .
\end{equation}
In particular, such a modification leads to the time-continuous analog of the twisted transition kernel
\R{eq:q-twisted} introduced in Section \ref{sec:dp}. 

\medskip

\begin{lemma} \label{twisting_SDEs}
Let $\psi_t(z)$ denote the solutions of the backward Kolmogorov equation 
\begin{equation} \label{eq:BKE1}
\partial_t \psi_t = -{\cal L}_t \psi_t = -\nabla_{z} \psi_t \cdot f_t  -  \Frac{\gamma}{2} \Delta_{z} \psi_t
\end{equation}
for given final $\psi_1(z) >0$ and $t\in [0,1]$. The controlled SDE \R{eq:Forward-SDE2} with
\begin{equation} \label{eq:optimal_control_SDE}
u_t(z) := \gamma \nabla_z \log \psi_t(z)
\end{equation}
leads to a time-one forward transition kernel $q_+^\psi(z_1|z_0)$ which satisfies
\begin{equation*}
q_+^\psi (z_1|z_0) = \psi_1 (z_1)\,q_+(z_1|z_0)\,\psi_0(z_0)^{-1} ,
\end{equation*}
where $q_+(z_1|z_0)$ denotes the time-one forward transition kernel of the uncontrolled forward SDE \R{eq:Forward-SDE}.
\end{lemma}

\begin{proof} A proof of this lemma has, for example, been given by \citeasnoun{sr:DP91} (Theorem 2.1). See also Appendix D for
a closely related discussion based on the potential $\phi_t$ as introduced in (\ref{eq:def_phi}).
\end{proof}

\medskip

\noindent
More generally, the modified forward SDE \R{eq:Forward-SDE2} with $Z_0^+ \sim \pi_0$ generates a path measure which we denote by 
$\mathbb{Q}^u$ for given functions $u_t(z)$, $t\in [0,1]$. Realisations of this path measure 
are denoted by $z_{[0,1]}^u$. According to Girsanov's theorem \cite{sr:P14}, the two path measures $\mathbb{Q}$ and
$\mathbb{Q}^u$ are absolutely continuous with respect to each other, with Radon--Nikodym derivative 
\begin{equation} \label{eq:RND}
\frac{\dd \mathbb{Q}^u}{\dd \mathbb{Q}}_{|z_{[0,1]}^u} = \exp \left(\frac{1}{2\gamma} \int_0^1 \left( \|u_t\|^2\,\dd t +
2 \gamma^{1/2} u_t \cdot \dd W_t^+ \right)\right),
\end{equation}
provided that the Kullback--Leibler divergence, ${\rm KL}(\mathbb{Q}^u||
\mathbb{Q})$, between $\mathbb{Q}^u$ and $\mathbb{Q}$, given by
\begin{equation} \label{eq:KLu}
{\rm KL}(\mathbb{Q}^u||\mathbb{Q}) := \int \left[ \frac{1}{2\gamma } \int_0^1 \|u_t\|^2 \,\dd t  \right] \mathbb{Q}^u(\dd z^u_{[0,1]})\,,
\end{equation}
is finite. Recall that the Kullback--Leibler divergence between two path measures $\mathbb{P} \ll \mathbb{Q}$ on ${\cal C}$ 
is defined by
\begin{equation*}
{\rm KL}(\mathbb{P}||\mathbb{Q}) = \int \log \frac{\dd \mathbb{P}}{\dd \mathbb{Q}} \,\mathbb{P}(\dd z_{[0,1]}) .
\end{equation*}

If the modified SDE \R{eq:Forward-SDE2} is used to make predictions, then its solutions  $z^u_{[0,1]}$ 
need to be weighted according to the inverse Radon--Nikodym derivative
\begin{equation}\label{eq:inverse_RN}
\frac{\dd \mathbb{Q}}{\dd \mathbb{Q}^u}_{|z_{[0,1]}^u} = \exp \left(-\frac{1}{2\gamma} \int_0^1 \left( \|u_t\|^2\,\dd t +
2 \gamma^{1/2} u_t \cdot \dd W_t^+ \right)\right)
\end{equation}
in order to reproduce the desired forecast PDF $\pi_1$ of the original SDE (\ref{eq:Forward-SDE}). 

\medskip

\begin{remark} A heuristic derivation of equation \R{eq:RND} can be found in subsection \ref{sec:SDE_model_pred} below,
where we discuss the numerical approximation of SDEs by the Euler--Maruyama method. Equation \R{eq:KLu} follows immediately 
from \R{eq:RND} by noting that the expectation of Brownian motion under the path measure $\mathbb{Q}^u$ is zero.
\end{remark}

%
\subsection{Filtering and Smoothing} \label{sec:filtering_and_smoothing}
%

We now incorporate the likelihood 
\begin{equation*}
l(z_1) = \pi(y_1|z_1)
\end{equation*} 
of the data $y_1$ at time $t_1=1$. 
Bayes' theorem tells us that, given the forecast PDF $\pi_1$ at time $t_1$, the posterior PDF
$\widehat{\pi}_1$ is given by \R{eq:filter_PDF}.
The distribution $\widehat{\pi}_1$ solves the filtering problem at time $t_1$ given the data $y_1$. 
We also recall the definition of the evidence \R{eq:evidence}. The quantity ${\cal F} = -\log \beta$ is called the free energy 
in statistical physics \cite{sr:HRSZ17}.

An appropriate transition kernel $q_1(\widehat{z}_1|z_1)$, satisfying \R{eq:transform_step}, is required 
in order to complete the transition from $\pi_0$ to $\widehat{\pi}_1$ following scenario (A) from definition \ref{ps:scenarios}.
A suitable framework for finding such transition kernels is via the theory of optimal transportation
\cite{sr:Villani}. More specifically, let $\Pi_{\rm c}$ denote the set of all joint probability measures $\pi(z_1,\widehat{z}_1)$ with marginals
\begin{equation*}
\int \pi(z_1,\widehat{z}_1)\,\dd \widehat{z}_1 = \pi_1(z_1), \qquad \int \pi(z_1,\widehat{z}_1)\,\dd z_1 = \widehat{\pi}_1(\widehat{z}_1) .
\end{equation*}
We seek the joint measure $\pi^\ast (z_1,\widehat{z}_1) \in \Pi_{\rm c}$ which minimises the expected Euclidean distance between the 
two associated random variables $Z_1$ and $\widehat{Z}_1$, that is,
\begin{equation} \label{eq:optimal_coupling1}
\pi^\ast = \arg \inf_{\pi \in \Pi_{\rm c}} \int \int \|z_1-\widehat{z}_1\|^2\, \pi(z_1,\widehat{z}_1)\,\dd z_1\,\dd \widehat{z}_1 .
\end{equation}
The minimising joint measure is of the form
\begin{equation} \label{eq:optimal_coupling2}
\pi^\ast(z_1,\widehat{z}_1) = \delta(\widehat{z}_1 - \nabla_z \Phi (z_1))\,\pi_1(z_1)
\end{equation}
with suitable convex potential $\Phi$ under appropriate conditions on the PDFs 
$\pi_1$ and $\widehat{\pi}_1$ \cite{sr:Villani}. These conditions are
satisfied for dynamical systems with Gaussian model errors and typical SDE models.  Once
the potential $\Phi$ (or an approximation) is available, samples $z_1^i$, $i=1,\ldots,M$, 
from the forecast PDF $\pi_1$ can be converted into
samples $\widehat{z}_1^i$, $i=1,\ldots,M$, from the filtering distribution $\widehat{\pi}_1$ via the deterministic
transformation
\begin{equation} \label{eq:optimal_coupling3}
\widehat{z}_1^i = \nabla_z \Phi (z_1^i)\,.
\end{equation}
We will discuss in Section \ref{sec:numerics} how to approximate the transformation
\R{eq:optimal_coupling3} numerically. We will find that many of the popular data assimilation schemes, such as the ensemble
Kalman filter, can be viewed as approximations to \R{eq:optimal_coupling3} \cite{sr:reichcotter15}.

We recall at this point that classical particle filters start from the importance weights
\begin{equation*}
w^i \propto \frac{\widehat{\pi}_1(z_1^i)}{\pi_1(z_1^i)} = \frac{l(z_1^i)}{\beta} ,
\end{equation*}
and obtain the desired samples $\widehat{z}^i$ by an appropriate resampling with replacement scheme 
\cite{sr:Doucet,sr:arul02,sr:DC05} instead of applying a deterministic transformation of the form 
\R{eq:optimal_coupling3}.

\medskip

\begin{remark} 
If one replaces the forward transition kernel $q_+(z_1|z_0)$ with a twisted kernel \R{eq:q-twisted}, then, using
\R{eq:importance_proposal}, the filtering distribution \R{eq:filter_PDF} satisfies
\begin{equation} \label{eq:filter_PDF2}
\frac{\widehat{\pi}_1(z_1)}{\pi_1^\psi (z_1)} =   \frac{l(z_1)\sum_{j=1}^M q_+(z_1|z_0^j)}{\beta \sum_{j=1}^M 
\Frac{\psi(z_1)}{ \widehat{\psi}(z_0^j)}\, q_+(z_1|z_0^j) } .
\end{equation} 
Hence drawing samples $z_1^i$, $i=1,\ldots,M$, from $\pi_1^\psi$ instead of $\pi_1$ leads to 
modified importance weights
\begin{equation} \label{eq:importance_weights2}
w^i \propto \frac{l(z_1^i)\sum_{j=1}^M q_+(z_1^i|z_0^j)}{\beta \sum_{j=1}^M 
\Frac{\psi(z_1^i)}{ \widehat{\psi}(z_0^j)}\, q_+(z_1^i|z_0^j) } .
\end{equation}
We will demonstrate in Section \ref{sec:SP} that finding a twisting potential $\psi$ such that $\widehat{\pi}_1 = \pi_1^\psi$, leading to importance
weights $w^i = 1$ in \R{eq:importance_weights2}, is equivalent to solving the 
Schr\"odinger problem \R{eq:SS1a}--\R{eq:SS1d}.
\end{remark}

\medskip

\noindent
The associated smoothing distribution at time $t=0$ can be defined as follows. First introduce
\begin{equation} \label{eq:r_T}
\psi(z_1) := \frac{\widehat{\pi}_1(z_1)}{\pi_1(z_1)} = \frac{l(z_1)}{\beta} .
\end{equation}
Next we set
\begin{equation} \label{eq:backward_iteration}
\widehat{\psi}(z_0) := \int q_+(z_1|z_0)\,\psi (z_1)\,\dd z_1 = \beta^{-1} \int q_+(z_1|z_0)\,l(z_1)\,\dd z_1\,,
\end{equation}
and introduce $\widehat{\pi}_0 := \pi_0\,\widehat{\psi}$, that is,
\begin{align} \nonumber
\widehat{\pi}_0(z_0) &= \frac{1}{M} \sum_{i=1}^M \widehat{\psi}(z_0^i)   \,\delta (z_0-z_0^i) \\ \label{eq:smoothing_0}
&= \frac{1}{M} \sum_{i=1}^M \gamma^i\, \delta (z_0-z_0^i)
\end{align}
since $\widehat{\psi}(z_0^i) = \gamma^i$ with $\gamma^i$ defined by \R{eq:gamma_i}. 
\medskip

\begin{lemma}
The smoothing PDFs $\widehat{\pi}_0$ and $\widehat{\pi}_1$ satisfy
\begin{equation} \label{eq:backward_simulation}
\widehat{\pi}_0(z_0) =  \int q_-(z_0|z_1) \,\widehat{\pi}_1(z_1)\,\dd z_1
\end{equation}
with the backward transition kernel defined by \R{eq:backward_transition_kernel}. 
Furthermore, 
\begin{equation*}
\widehat{\pi}_1(z_1) = \int \widehat{q}_+(z_1|z_0)\,\widehat{\pi}_0(z_0)\,\dd z_0
\end{equation*}
with twisted forward transition kernels
\begin{equation}
 \widehat{q}_+(z_1|z_0^i) := \psi(z_1) \,q_+(z_1|z_0^i)\,\widehat{\psi}(z_0^i)^{-1}
 = \frac{l(z_1)}{\beta\,\gamma^i} q_+ (z_1|z_0^i)  \label{eq:forward_smoothing_kernel}
\end{equation}
and $\gamma^i$, $i=1,\ldots,M$, defined by \R{eq:gamma_i}.
\end{lemma}

\begin{proof} We note that
\begin{equation*}
q_-(z_0|z_1)\,\widehat{\pi}_1(z_1) = \frac{\pi_0(z_0)}{\pi_1(z_1)} q_+(z_1|z_0) 
\widehat{\pi}_1(z_1) = \frac{l(z_1)}{\beta} q_+(z_1|z_0)\,\pi_0(z_0) ,
\end{equation*}
which implies the first equation. The second equation follows from $\widehat{\pi}_0 = \widehat{\psi} \,\pi_0$ and
 \begin{equation*}
 \int \widehat{q}_+(z_1|z_0) \,\widehat{\pi}_0(z_0)\,\dd z_0 =  \frac{1}{M}\sum_{i=1}^M \frac{l(z_1)}{\beta}\,q_+(z_1|z_0^i) .
 \end{equation*}
In other words, we have defined a twisted forward transition kernel of the form \R{eq:q-twisted}. 
\end{proof}

\medskip

\noindent 
Seen from a more abstract perspective, we have provided an alternative formulation of the joint smoothing distribution
 \begin{equation} \label{eq:smoothing-PDF}
 \widehat{\pi}(z_0,z_1) := \frac{l(z_1)\,q_+(z_1|z_0)\,\pi_0(z_0)}{\beta}
 \end{equation}
 in the form of
 \begin{align} \nonumber
 \widehat{\pi}(z_0,z_1) &= \frac{l(z_1)}{\beta} \frac{\psi(z_1)}{\psi(z_1)} q_+(z_1|z_0) \frac{\widehat{\psi}(z_0)}{\widehat{\psi}(z_0)} \pi_0(z_0)\\
  &= \widehat{q}_+(z_1|z_0) \,\widehat{\pi}_0(z_0) \label{eq:optimal_proposal1}
  \end{align}
 because of \R{eq:r_T}.
 Note that the marginal distributions of $\widehat{\pi}$ are provided by $\widehat{\pi}_0$ and $\widehat{\pi}_1$, 
 respectively.

One can exploit these formulations computationally as follows. If one has 
generated $M$ equally weighted particles $\widehat{z}_0^j$ from the smoothing distribution \R{eq:smoothing_0} 
at time $t=0$ via resampling with replacement, then one can obtain equally weighted samples $\widehat{z}_1^j$ from the filtering distribution $\widehat{\pi}_1$ using the modified transition kernels \R{eq:forward_smoothing_kernel}.
This is the idea behind the optimal proposal particle filter \cite{sr:Doucet,sr:arul02,sr:FK18} and provides an implementation
of scenario (B) as introduced in Definition \ref{ps:scenarios}. 

\medskip

\begin{remark}
We remark that backward simulation methods use \R{eq:backward_simulation} in
order to address the smoothing problem \R{eq:Smoothing} in a sequential forward--backward manner. 
Since we are not interested in the general smoothing problem in this paper, we refer the reader to the survey by 
\citeasnoun{sr:LS13} for more details.
\end{remark}

 \begin{lemma}
 Let $\psi (z) >0$ be a twisting potential such that
 \begin{equation*}
 l^\psi(z_1) := \frac{l(z_1)}{\beta\,\psi(z_1)}\pi_0[\widehat{\psi}]
 \end{equation*}
 is well-defined with $\widehat{\psi}$ given by \R{eq:psi_0}. Then the smoothing
 PDF \R{eq:smoothing-PDF} can be represented as
 \begin{equation} \label{eq:tilted}
 \widehat{\pi}(z_0,z_1) = l^\psi(z_1) \,q^\psi_+(z_1|z_0)\,\pi^\psi_0(z_0) ,
 \end{equation}
where the modified forward transition kernel $q_+^\psi(z_1|z_0)$ is defined 
by \R{eq:q-twisted} and the modified initial PDF by
 \begin{equation*}
 \pi_0^\psi (z_0) := \frac{\widehat{\psi}(z_0)\,\pi_0(z_0)}{\pi_0[\widehat{\psi}]} .
 \end{equation*}
 \end{lemma}
 
 \begin{proof} This follows from the definition of the smoothing PDF $\widehat{\pi}(z_0,z_1)$ 
 and the twisted transition kernel $q_+^\psi(z_1|z_0)$.
 \end{proof}
 
 \medskip
 
 \begin{remark}
 As mentioned before, the choice \R{eq:r_T} implies $l^\psi = \mbox{const.}$, and leads to the well--known optimal proposal
 density for particle filters. The more general formulation \R{eq:tilted} has recently been explored 
 and expanded by \citeasnoun{sr:GJL17} and \citeasnoun{sr:HBDD18} in order to derive efficient proposal densities for 
 the general smoothing problem \R{eq:Smoothing}.  Within the simplified formulation \R{eq:tilted}, such approaches reduce to a 
 change of measure from $\pi_0$ to $\pi_0^\psi$ at $t_0$ followed by a forward transition according to $q_+^\psi$ and subsequent 
 reweighting by a modified likelihood $l^\psi$  at $t_1$ and hence lead
 to particle filters that combine scenarios (A) and (B) as introduced in Definition \ref{ps:scenarios}. 
 \end{remark}
  
\subsubsection{Gaussian model errors (cont.)}

We return to the discrete-time process \R{eq:discrete_time_Gaussian} and
assume a Gaussian measurement error leading to a Gaussian likelihood
\begin{equation*}
l(z_1) \propto \exp \left(-\frac{1}{2}(Hz_1-y_1)^T R^{-1} (Hz_1-y_1)\right) .
\end{equation*}
We set $\psi_1 = l/\beta$ in order to derive the optimal forward kernel for the associated smoothing/filtering problem. 
Following the discussion from subsection \ref{sec:Gauss1}, this leads to the modified transition kernels
\begin{equation*}
\widehat{q}_+(z_1|z_0^i) := {\rm n}(z_1;\bar z_1^i,\gamma \bar B)
\end{equation*}
with $\bar B$ and $K$ defined by \R{eq:Bz} and \R{eq:Kalman_gain1}, respectively, and
\begin{equation*}
\bar z_1^i := \Psi(z_0^i) - K(H\Psi(z_0^i) - y_1) .
\end{equation*}
The smoothing distribution $\widehat{\pi}_0$ is
given by
\begin{equation*}
\widehat{\pi}_0(z_0) = \frac{1}{M} \sum_{i=1}^M \gamma^i \,\delta(z-z_0^i) 
\end{equation*}
with coefficients 
\begin{equation*}
\gamma^i \propto \exp \left(-\frac{1}{2}(H \Psi(z_0^i)-y_1)^\T (R+\gamma HBH^\T)^{-1} (H \Psi(z_0^i)-y_1)\right) .
\end{equation*}
It is easily checked that, indeed,
\begin{equation*}
\widehat{\pi}_1(z_1) = \int q_+^\psi(z_1|z_0)\,\widehat{\pi}_0(z_0)\,\dd z_0 .
\end{equation*}

The results from this subsection have been used in simplified form in Example \ref{ex:example1} in order to
compute \R{eq:smoothing_example1}.
We also note that a non--optimal, that is, $\psi(z_1) \not= l(z_1)/\beta$, but Gaussian choice for $\psi$ 
leads to a Gaussian $l^\psi$ and the transition kernels $q^\psi_+(z_1|z_0^i)$ in \R{eq:tilted} remain Gaussian as well. 
This is in contrast to the Schr\"odinger problem, which we discuss in the following Section \ref{sec:SP} and which 
leads to forward transition kernels of the form \R{eq:SS_forward_kernel} below.

\subsubsection{SDE models (cont.)}
The likelihood $l(z_1)$ introduces a change of measure over path space $z_{[0,1]}\in {\cal C}$ from
the forecast measure $\mathbb{Q}$ with marginals $\pi_t$ to the smoothing measure $\widehat{\mathbb{P}}$
via the Radon--Nikodym derivative
\begin{equation} \label{eq:RND-SDE}
\frac{\dd \widehat{\mathbb{P}}}{\dd \mathbb{Q}}_{|z_{[0,1]}} = \frac{l(z_1)}{\beta} .
\end{equation}
We let $\widehat{\pi}_t$ denote the marginal distributions of the smoothing measure $\widehat{\mathbb{P}}$ at time $t$. 

\medskip

\begin{lemma} \label{lemma:control_smoothing}
Let $\psi_t$ denote the solution of the backward Kolmogorov equation \R{eq:BKE1} with final
 condition $\psi_1(z) = l(z)/\beta$ at $t=1$. Then the controlled forward SDE
\begin{equation} \label{eq:Forward-Smoothing-SDE}
\dd Z_t^+ =  \left( f_t(Z_t^+)  + \gamma \nabla_z \log \psi_t (Z_t^+) \right) \dd t 
 +  \gamma^{1/2}  \,\dd W_t^+ ,
 \end{equation}
 with $Z_0^+ \sim \widehat{\pi}_0$ at time $t=0$, implies $Z_1^+ \sim \widehat{\pi}_1$ at final time $t=1$.
\end{lemma}

\begin{proof}  The lemma is a consequence of Lemma \ref{twisting_SDEs} and 
Definition \R{eq:forward_smoothing_kernel} of the smoothing kernel $\widehat{q}_+(z_1|z_0)$ with
$\psi(z_1) = \psi_1(z_1)$ and $\widehat{\psi}(z_0) = \psi_0(z_0)$.
\end{proof}

\medskip

 \noindent
 The SDE \R{eq:Forward-Smoothing-SDE} is obviously a special case of \R{eq:Forward-SDE2} with control law, $u_t$, given by
 \R{eq:optimal_control_SDE}. Note, however, that the initial distributions for \R{eq:Forward-SDE2} and \R{eq:Forward-Smoothing-SDE} are different.
 We will reconcile this fact in the following subsection by considering the associated Schr\"odinger problem \cite{sr:FG97,sr:L14,sr:CGP14}.  

\medskip

\begin{lemma}
The solution $\psi_t$ of the backward Kolmogorov equation \R{eq:BKE1} with final
condition $\psi_1(z) = l(z)/\beta$ at $t=1$ satisfies 
\begin{equation} \label{eq:density_ratio}
\psi_t(z) = \frac{\widehat{\pi}_t(z)}{\pi_t(z)}
\end{equation}
and the PDFs $\widehat{\pi}_t$ coincide with the marginal PDFs of the backward SDE \R{eq:Backward-SDE} with final condition 
$Z_1^- \sim \widehat{\pi}_1$.
\end{lemma}

\begin{proof}
We first note that \R{eq:density_ratio} holds at final time $t=1$. Furthermore, equation \R{eq:density_ratio} implies
\begin{equation*}
\partial_t \widehat{\pi}_t = \pi_t \,\partial_t \psi_t - \psi_t \,\partial_t \pi_t .
\end{equation*}
Since $\pi_t$ satisfies the Fokker--Planck equation \R{eq:FPE1} and $\psi_t$ the backward Kolmogorov equation
\R{eq:BKE1}, it follows that
\begin{equation} \label{eq:BFPES}
\partial_t \widehat{\pi}_t = -\nabla_z \cdot (\widehat{\pi}_t (f - \gamma \nabla_z \log \pi_t)) - \frac{\gamma}{2} \Delta_z \widehat{\pi}_t ,
\end{equation}
which corresponds to the Fokker--Planck equation \R{eq:BFP} of the backward SDE \R{eq:Backward-SDE} with final condition 
$Z_1^- \sim \widehat{\pi}_1$ and marginal PDFs denoted by $\widehat{\pi}_t$ instead of $\pi_t$.
\end{proof}

\medskip

\noindent
Note that \R{eq:BFPES} is equivalent to
\begin{equation*}
\partial \widehat{\pi}_t = -\nabla_z \cdot (\widehat{\pi}_t (f_t - \gamma \nabla_z \log \pi_t + \gamma \nabla_z \log \widehat{\pi}_t) )
 + \frac{\gamma}{2} \Delta_z \widehat{\pi}_t ,
 \end{equation*}
 which in turn is equivalent to the Fokker--Planck equation of the forward smoothing SDE 
 \R{eq:Forward-Smoothing-SDE} since $\psi_t = \widehat{\pi}_t/\pi_t$. 
 
 We conclude from the previous lemma 
 that one can either solve the backward Kolmogorov
 equation \R{eq:BKE1} with $\psi_1(z) = l(z)/\beta$ or the backward SDE \R{eq:Backward-SDE}
with $Z_1^- \sim \widehat{\pi}_1 = l\,\pi_1/\beta$ in order to derive the desired control law $u_t (z) = \gamma
\nabla_z \log \psi_t$ in \R{eq:Forward-Smoothing-SDE}.
 
\medskip

\begin{remark} \label{rem:fbsde}
The notion of a backward SDE used throughout this paper is different from the notion of a backward SDE 
in the sense of \citeasnoun{sr:C16}, for example. More specifically, It\^o's formula 
\begin{equation*}
\dd \psi_t = \partial_t \psi_t\, \dd t + \nabla_z \psi_t \cdot \,\dd Z_t^+ + \frac{\gamma}{2}\Delta_z \psi_t  \,\dd t
\end{equation*}
and the fact that $\psi_t$ satisfies the backward Kolmogorov equation \R{eq:BKE1} imply that
\begin{equation} \label{eq:material_advection}
\dd \psi_t = \gamma^{1/2}\, \nabla_z \psi_t \cdot \dd W_t^+
\end{equation}
along solutions of the forward SDE \R{eq:Forward-SDE}. In other words, the quantities $\psi_t$ are materially advected
along solutions $Z_t^+$ of the forward SDEs in expectation or, in the language of stochastic analysis, $\psi_t$ is a martingale. Hence, by
the martingale representation theorem, we can reformulate the problem of determining $\psi_t$ as follows. Find the
solution $(Y_t,V_t)$ of the backward SDE  
\begin{equation} \label{eq:BSDE}
\dd Y_t = V_t \cdot \dd W_t^+
\end{equation}
subject to the final condition $Y_1 = l/\beta$ at $t=1$. Here \R{eq:BSDE} has to be understood as a backward SDE in the sense of \citeasnoun{sr:C16},
for example, where the solution $(Y_t,V_t)$ is adapted to the forward SDE \R{eq:Forward-SDE}, that is, to the past $s\le t$, whereas the solution $Z_t^-$  of the backward SDE  \R{eq:Backward-SDE} is adapted to the future $s\ge t$. The solution of \R{eq:BSDE} is given by $Y_t = \psi_t$ and $
V_t = \gamma^{1/2}\,\nabla_z \psi_t$ in agreement with \R{eq:material_advection} (compare \citeasnoun{sr:C16} page 42). See 
Appendix D for the numerical treatment of \R{eq:BSDE}.
\end{remark}

\medskip
 
\noindent
A variational characterisation of the smoothing path measure $\widehat{\mathbb{P}}$ is given by the Donsker--Varadhan principle
\begin{equation} \label{eq:DV_principle}
\widehat{\mathbb{P}} =  \arg \inf_{\mathbb{P}\ll \mathbb{Q}} \left\{ -\mathbb{P}[\log(l)] + {\rm KL}(\mathbb{P}||\mathbb{Q}) \right\} ,
\end{equation}
that is, the distribution $\widehat{\mathbb{P}}$ is chosen such that the expected loss, $-\mathbb{P}[\log(l)]$, 
is minimised subject to the penalty introduced by the Kullback--Leibler divergence with respect to the original 
path measure $\mathbb{Q}$. Note that
\begin{equation} \label{eq:free_energy}
\inf_{\mathbb{P}\ll \mathbb{Q}}   \left\{\mathbb{P}[-\log(l)] + {\rm KL}(\mathbb{P}||\mathbb{Q}) \right\} = -\log \beta
\end{equation}
with $\beta = \mathbb{Q}[l]$. 
The connection between smoothing for SDEs and the Donsker--Varadhan principle has, for example, been discussed by 
\citeasnoun{sr:MN03}. See also \citeasnoun{sr:HRSZ17} for an in-depth discussion of variational formulations and 
their numerical implementation in the context of rare event simulations for which it is generally assumed that 
$\pi_0(z) = \delta(z-z_0)$ in \R{eq:initial_pdf},  that is, the ensemble size is $M=1$ when viewed within the  context of this paper.

\medskip

\begin{remark}
One can choose $\psi_t$ differently from the choice made in \R{eq:density_ratio}  
by changing the final condition for the backward Kolmogorov equation \R{eq:BKE1} to any suitable $\psi_1$. 
As already discussed for twisted discrete--time smoothing, such modifications give rise to alternative representations of 
the smoothing distribution $\widehat{\mathbb{P}}$ 
in terms of modified forward SDEs, likelihood functions and initial distributions. See  \citeasnoun{sr:KR16} and \citeasnoun{sr:RK17} 
for an application of these ideas to importance sampling in the context of partially observed diffusion processes. 
More specifically, let $u_t$ denote the associated control law \R{eq:optimal_control_SDE} 
for the forward SDE \R{eq:Forward-SDE2} with given initial distribution
$Z_0^+ \sim q_0$. Then
\begin{equation*}
\frac{\dd \widehat{\mathbb{P}}}{\dd \mathbb{Q}}_{|z^u_{[0,1]}} = \frac{\dd \widehat{\mathbb{P}}}{\dd \mathbb{Q}^u}_{|z_{[0,1]}^u}\,
\frac{\dd \mathbb{Q}^u}{\dd \mathbb{Q}}_{|z_{[0,1]}^u} ,
\end{equation*}
which, using \R{eq:RND} and \R{eq:RND-SDE}, implies the modified Radon--Nikodym derivative
\begin{equation} \label{eq:modified_likelihood}
\frac{\dd \widehat{\mathbb{P}}}{\dd \mathbb{Q}^u}_{|z_{[0,1]}^u} =  \frac{l(z_1^u)}{\beta} \frac{\pi_0(z_0^u)}{q_0(z_0^u)} 
\exp \left(-\frac{1}{2\gamma} \int_0^1 \left( \|u_t\|^2\,\dd t +
2 \gamma^{1/2} u_t \cdot \dd W_t^+ \right)\right) .
\end{equation}
Recall from lemma \ref{lemma:control_smoothing} that the control law \R{eq:optimal_control_SDE} with $\psi_t$ defined by
\R{eq:density_ratio} together with $q_0 = \widehat{\pi}_0$ 
leads to $\widehat{\mathbb{P}} = \mathbb{Q}^u$.
\end{remark}

%
\subsection{Schr\"odinger Problem} \label{sec:SP}
%

In this subsection, we show that scenario (C) from definition \ref{ps:scenarios} leads to a certain boundary value problem
first considered by \citeasnoun{sr:S31}. More specifically, we state the, so-called, Schr\"odinger problem and 
show how it is linked to data assimilation scenario (C). 

In order to introduce the Schr\"odinger problem, we return to the twisting potential approach as utilised in section \ref{sec:filtering_and_smoothing}
with two important modifications. These modifications are, first, that the twisting potential $\psi$ is determined implicitly and, second, that the
modified transition kernel $q_+^\psi$ is applied to $\pi_0$ instead of the tilted initial density $\pi_0^\psi$ as in
\R{eq:tilted}. More specifically, we have the following.

\medskip

\begin{definition}
We seek the pair of functions $\widehat{\psi}(z_0)$ and $\psi(z_1)$ which solve the boundary value problem
\begin{align} \label{eq:SS1a}
\pi_0(z_0) &= \pi^\psi_0(z_0)\,\widehat{\psi}(z_0) , \\ \label{eq:SS1b}
\widehat{\pi}_1(z_1) &= \pi^\psi_1(z_1) \, \psi(z_1) , \\ \label{eq:SS1c}
\pi^\psi_1(z_1) &= \int q_+(z_1|z_0)\,\pi^\psi_0(z_0)\,\dd z_0 ,\\
\widehat{\psi}(z_0) &= \int q_+(z_1|z_0)\,\psi(z_1)\,\dd z_1 , \label{eq:SS1d}
\end{align}
for given marginal (filtering) distributions $\pi_0$ and $\widehat{\pi}_1$ at $t=0$ and $t=1$, respectively. 
The required modified PDFs $\pi_0^\psi$ and $\pi_1^\psi$ are defined by \R{eq:SS1a} and \R{eq:SS1b}, respectively.
The solution $(\widehat{\psi},\psi)$ of the Schr\"odinger system \R{eq:SS1a}--\R{eq:SS1d} leads to the modified transition 
kernel
\begin{equation} \label{eq:SS_kernel}
q^\ast_+(z_1|z_0) := \psi(z_1) \,q_+(z_1|z_0)\, \widehat{\psi}(z_0)^{-1} ,
\end{equation}
which satisfies
\begin{equation*}
\widehat{\pi}_1(z_1) = \int q_+^\ast (z_1|z_0)\,\pi_0(z_0)\,\dd z_0
\end{equation*}
by construction.
\end{definition}

\medskip

\noindent
The modified transition kernel $q_+^\ast(z_1|z_0)$ couples the two marginal distributions $\pi_0$ and $\widehat{\pi}_1$ 
with the twisting potential $\psi$ implicitly defined. In other words, $q_+^\ast$ provides the transition kernel for going from the initial
distribution \R{eq:initial_pdf} at time $t_0$ to the filtering distribution at time $t_1$ without the need for any reweighting, that is,
the desired transition kernel for scenario (C). See \citeasnoun{sr:L14} and \citeasnoun{sr:CGP14} for more mathematical details on the 
Schr\"odinger problem.

\medskip

\begin{remark}
Let us compare the Schr\"odinger system to the twisting potential approach \R{eq:tilted} for the smoothing problem from
subsection \ref{sec:filtering_and_smoothing} in some more detail. First, note that the twisting potential approach to smoothing 
replaces \R{eq:SS1a} with
\begin{equation*}
\pi^\psi_0(z_0) = \pi_0(z_0) \,\widehat{\psi}(z_0)
\end{equation*}
and \R{eq:SS1b} with
\begin{equation*}
\pi^\psi_1(z_1) = \pi_1(z_1) \,\psi(z_1),
\end{equation*}
where $\psi$ is 
a given twisting potential normalised such that $\pi_1[\psi] = 1$. The associated $\widehat{\psi}$ is determined by \R{eq:SS1d} as in the twisting approach.
In both cases, the modified transition kernel is given by \R{eq:SS_kernel}. Finally, \R{eq:SS1c} is replaced by the 
prediction step \R{eq:prediction}.
\end{remark} 

\medskip

\noindent
In order to solve the Schr\"odinger system for our given initial distribution \R{eq:initial_pdf} 
and the associated filter distribution $\widehat{\pi}_1$, we make the {\it ansatz}
\begin{equation*}
\pi^\psi_0(z_0) = \frac{1}{M} \sum_{i=1}^M \alpha^i \,\delta(z_0-z_0^i), \qquad \sum_{i=1}^M \alpha^i = M .
\end{equation*}
This {\it ansatz} together with \R{eq:SS1a}--\R{eq:SS1d} immediately implies
\begin{equation*}
\widehat{\psi}(z_0^i) = \frac{1}{\alpha^i}\,, \qquad
\pi^\psi_1(z_1) = \frac{1}{M} \sum_{i=1}^M \alpha^i \,q_+(z_1|z_0^i) ,
\end{equation*}
as well as
\begin{equation} \label{eq:Schroedinger_twist}
\psi(z_1) = \frac{\widehat{\pi}_1(z_1)}{\pi^\psi_1(z_1)} =\frac{l(z_1)}{\beta} \frac{\Frac{1}{M} 
\sum_{j=1}^M q_+ (z_1|z_0^j)}{\Frac{1}{M} \sum_{j=1}^M \alpha^j \,
q_+ (z_1|z_0^j)} .
\end{equation}
Hence we arrive at the equations
\begin{align} \label{eq:Sinkhorn1a}
\widehat{\psi}(z_0^i) &= \frac{1}{\alpha^i}\\ \nonumber
&= \int \psi(z_1)\,q_+(z_1|z_0^i)\,\dd z_1\\
&= \int  \frac{l(z_1)}{\beta} \frac{\sum_{j=1}^M q_+(z_1|z_0^j)}{\sum_{j=1}^M \alpha^j \,
q_+ (z_1|z_0^j)} \,q_+(z_1|z_0^i)\,\dd z_1\label{eq:Sinkhorn1c}
\end{align}
for $i=1,\ldots,M$. These $M$ equations have to be solved for the $M$ unknown coefficients $\{\alpha^i\}$. 
In other words, the Schr\"odinger problem becomes finite-dimensional in the context of this paper.
More specifically, we have the following result.

\medskip

\begin{lemma}
The forward Schr\"odinger transition kernel
\R{eq:SS_kernel} is  given by
\begin{align} \nonumber
q_+^\ast (z_1|z_0^i) &=  \frac{\widehat{\pi}_1(z_1)}{\pi_1^\psi(z_1)} q_+(z_1|z_0^i) \,\alpha^i\\ 
&=   \frac{ \alpha^i \sum_{j=1}^M  \,q_+ (z_1|z_0^j)}{\sum_{j=1}^M \alpha^j \,
q_+ (z_1|z_0^j)} \,\frac{l(z_1)}{\beta}\,q_+(z_1|z_0^i) , \label{eq:SS_forward_kernel}
\end{align}
for each particle, $z_0^i$, with the coefficients $\alpha^j$, $j=1,\ldots,M$, defined by 
\R{eq:Sinkhorn1a}--\R{eq:Sinkhorn1c}.
\end{lemma}

\begin{proof}
Because of \R{eq:Sinkhorn1a}--\R{eq:Sinkhorn1c}, the forward transition kernels \R{eq:SS_forward_kernel}
satisfy
\begin{equation} \label{eq:SS_cond1}
\int q_+^\ast (z_1|z_0^i)\,\dd z_1 = 1 
\end{equation}
for all $i = 1,\ldots,M$ and
\begin{align} \nonumber
\int q_+^\ast (z_1|z_0) \,\pi_0(z_0)\,\dd z &=
\frac{1}{M}\sum_{i=1}^M q_+^\ast (z_1|z_0^i) \\ \nonumber
&= \frac{1}{M} \sum_{i=1}^M  \alpha^i  q_+(z_1|z_0^i)\,
 \frac{\widehat{\pi}_1(z_1)}{\Frac{1}{M} \sum_{j=1}^M \alpha^j \,
q_+ (z_1|z_0^j)}\\
& = \widehat{\pi}_1(z_1) , \label{eq:SS_cond2}
\end{align}
as desired. 
\end{proof}

\medskip

\noindent
Numerical implementations will be discussed in Section \ref{sec:numerics}. Note that 
knowledge of the normalising constant $\beta$ is not required {\it a priori} for solving
\R{eq:Sinkhorn1a}--\R{eq:Sinkhorn1c} since it appears as a common scaling factor.

We note that the coefficients $\{\alpha^j\}$ together with the associated potential $\psi$ from the Schr\"odinger system 
provide the optimally twisted prediction kernel \R{eq:q-twisted} with respect to the filtering distribution $\widehat{\pi}_1$, 
that is, we set  $\widehat{\psi}(z_0^i) = 1/\alpha^i$ in \R{eq:twisted-predicted-PDF} and define the potential $\psi$ by
\R{eq:Schroedinger_twist}.  

\medskip

\begin{lemma} The Schr\"odinger transition kernel \R{eq:SS_kernel} satisfies the following constrained variational
principle. Consider the joint PDFs given by $\pi(z_0,z_1) := q_+(z_1|z_0)\pi_0(z_0)$ and $\pi^\ast (z_0,z_1) := q_+^\ast(z_1|z_0)
\pi_0(z_0)$. Then
\begin{equation} \label{eq:SVP}
\pi^\ast = \arg \inf_{\widetilde{\pi} \in \Pi_{\rm S}} \mbox{KL}(\widetilde{\pi}||\pi) .
\end{equation}
Here a joint PDF $\widetilde{\pi}(z_0,z_1)$ is an element of $\Pi_{\rm S}$ if
\begin{equation*}
\int \widetilde{\pi}(z_0,z_1)\,\dd z_1 = \pi_0(z_0), \qquad \int \widetilde{\pi}(z_0,z_1)\,\dd z_0 = \widehat{\pi}_1(z_1) .
\end{equation*}
\end{lemma}

\medskip

\begin{proof} See \citeasnoun{sr:FG97} for a proof and Remark \ref{rem:Var_S} for a heuristic derivation in the
case of discrete measures.
\end{proof}

\medskip

\noindent
The constrained variational formulation \R{eq:SVP} of Schr\"odinger's problem should be compared to the unconstrained 
Donsker--Varadhan variational principle 
\begin{equation}
\widehat{\pi} = \arg \inf \left\{ - \widetilde{\pi}[\log(l)] + \mbox{KL}(\widetilde{\pi}||\pi) \right\}
\end{equation}
for the associated smoothing problem. See Remark \ref{rem:DVP-SVP} below.

\medskip

\begin{remark} The Schr\"odinger problem is closely linked to optimal transportation
\cite{sr:cuturi13,sr:L14,sr:CGP14}. For example, consider the Gaussian
transition kernel \R{eq:Gaussian_kernel} with $\Psi(z) = z$ and $B=I$. Then the solution \R{eq:SS_forward_kernel} 
to the associated Schr\"odinger problem of coupling $\pi_0$ and $\widehat{\pi}_1$ reduces to the solution 
$\pi^\ast$ of the associated optimal transport problem 
\begin{equation*} 
\pi^\ast = \arg \inf_{\widetilde{\pi} \in \Pi_{\rm S}} \int \int \|z_0-z_1\|^2\, \widetilde{\pi}(z_0,z_1)\,\dd z_0\,\dd z_1
\end{equation*}
in the limit $\gamma \to 0$. 
\end{remark}

\subsubsection{SDE models (cont.)}

At the SDE level, Schr\"odinger's problem amounts to  continuously bridging the given initial PDF $\pi_0$ with the
PDF $\widehat{\pi}_1$ at final time using an appropriate modification of the
stochastic process $Z_{[0,1]}^+ \sim \mathbb{Q}$ defined by the forward SDE 
(\ref{eq:Forward-SDE}) with initial distribution $\pi_0$ at $t=0$. The desired modified
stochastic process $\mathbb{P}^\ast$ is defined as the minimiser of
\begin{equation*}
{\cal L}(\widetilde{\mathbb{P}}) := \mbox{KL}(\widetilde{\mathbb{P}}||\mathbb{Q})
\end{equation*}
subject to the constraint that the marginal distributions $\widetilde{\pi}_t$ of $\widetilde{\mathbb{P}}$ at
time $t=0$ and $t=1$ satisfy $\pi_0$ and $\widehat{\pi}_{1}$, 
respectively \cite{sr:FG97,sr:L14,sr:CGP14}. 

\medskip

\begin{remark} \label{rem:DVP-SVP}
We note that the Donsker--Varadhan variational principle \R{eq:DV_principle}, characterising the smoothing path 
measure $\widehat{\mathbb{P}}$, can be 
replaced by
\begin{equation*}
\mathbb{P}^\ast = \arg \inf_{\widetilde{\mathbb{P}}\in \Pi} \left\{-\widetilde{\pi}_1[\log (l)] + 
{\rm KL}(\widetilde{\mathbb{P}}||\mathbb{Q}) \right\}
\end{equation*}
with
\begin{equation*}
\Pi = \{\widetilde{\mathbb{P}} \ll \mathbb{Q}: \,\widetilde{\pi}_1 = \widehat{\pi}_1,\,\widetilde{\pi}_0 = \pi_0\}
\end{equation*}
in the context of Schr\"odinger's problem. The associated 
\begin{equation*}
- \log \beta^\ast :=  \inf_{\widetilde{\mathbb{P}}\in \Pi} \left\{ - \widetilde{\pi}_1[\log(l)] +
{\rm KL}(\widetilde{\mathbb{P}}||\mathbb{Q}) \right\}
 = - \widehat{\pi}[\log(l)] + {\rm KL}(\mathbb{P}^\ast||\mathbb{Q}) 
\end{equation*}
can be viewed as a generalisation of \R{eq:free_energy} and gives rise to a generalised evidence $\beta^\ast$, 
which could be used for model comparison and parameter estimation.
\end{remark}

\medskip

\noindent
The Schr\"odinger process $\mathbb{P}^\ast$ corresponds to a 
Markovian process across the whole time domain $[0,1]$ \cite{sr:L14,sr:CGP14}. More specifically,
consider the controlled forward SDE \R{eq:Forward-SDE2}
with initial conditions
\begin{equation*}
Z_0^+ \sim \pi_{0}
\end{equation*}
and a given control law $u_t$ for $t\in [0,1]$. Let $\mathbb{P}^u$ denote the path measure associated to this process.
Then, as detailed discussed in detail by \citeasnoun{sr:DP91}, one can find time-dependent potentials $\psi_t$ with associated control laws
\R{eq:optimal_control_SDE} such that 
the marginal of the associated path measure $\mathbb{P}^u$ at times $t=1$  satisfies
\begin{equation*}
\pi^u_1 = \widehat{\pi}_1
\end{equation*}
and, more generally,
\begin{equation*}
\mathbb{P}^\ast = \mathbb{P}^u .
\end{equation*}
We summarise this result in the following lemma.

\medskip

\begin{lemma}
The Schr\"odinger path measure $\mathbb{P}^\ast$ can be generated by a controlled SDE \R{eq:Forward-SDE2} with
control law \R{eq:optimal_control_SDE}, where the desired potential $\psi_t$ can be obtained as follows. Let $(\widehat{\psi},\psi)$ denote the solution of
the associated Schr\"odinger system (\ref{eq:SS1a})--(\ref{eq:SS1d}), where $q_+(z_1|z_0)$ denotes the time--one forward 
transition kernel of \R{eq:Forward-SDE}. Then $\psi_t$ in \R{eq:optimal_control_SDE} is  the solution of the 
backward Kolmogorov equation (\ref{eq:BKE1}) with prescribed $\psi_1 = \psi$ at final time $t=1$. 
\end{lemma}

\medskip

\begin{remark} As already pointed out in the context of smoothing, the desired potential $\psi_t$ can also be obtained by solving
an appropriate backward SDE. 
More specifically, given the solution $(\widehat{\psi},\psi)$ and the implied PDF $\widetilde{\pi}_0^+ := \pi_0^\psi = \pi_0/\widehat{\psi}$ of
the Schr\"odinger system \R{eq:SS1a}--\R{eq:SS1d}, let $\widetilde{\pi}_t^+$, $t\ge 0$, denote 
the marginals of the forward SDE \R{eq:Forward-SDE} with $Z_0^+ \sim \widetilde{\pi}_0^+$. 
Furthermore, consider the backward SDE \R{eq:Backward-SDE}
with drift term
\begin{equation} \label{eq:drift_schroedinger}
b_t(z) = f_t(z) -\gamma \nabla_z \log \widetilde{\pi}_t^+ (z) ,
\end{equation}
and final time condition $Z_1^- \sim \widehat{\pi}_1$. Then the choice of $\widetilde{\pi}_0^+$ ensures that $Z_0^- \sim \pi_0$.
Furthermore the desired control in \R{eq:Forward-SDE2} is provided by
\begin{equation*}
u_t = \gamma \nabla_z \log \frac{\widetilde{\pi}_t^-}{\widetilde{\pi}_t^+} ,
\end{equation*}
where $\widetilde{\pi}_t^-$ denotes the marginal distributions of the backward SDE \R{eq:Backward-SDE}
with drift term \R{eq:drift_schroedinger} and $\widetilde{\pi}_1^- = \widehat{\pi}_1$. We will return to this reformulation of
the Schr\"odinger problem in Section \ref{sec:numerics} when considering it as the limit of a sequence of smoothing problems.
\end{remark}

\medskip

\begin{remark}
The solution to the Schr\"odinger problem for linear SDEs and Gaussian marginal distributions has been discussed 
in detail by \citeasnoun{sr:CGP14b}.
\end{remark}


\subsubsection{Discrete measures} \label{sec:SP_DM}

We finally discuss the Schr\"odinger problem in the context of finite-state Markov chains in more detail.
These results will be needed in the following sections on the numerical implementation of the Schr\"odinger approach
to sequential data assimilation.

Let us therefore consider an example which will be closely related to the discussion in Section \ref{sec:numerics}.
We are given a bi-stochastic matrix $Q \in \mathbb{R}^{L\times M}$ with all entries satisfying $q_{lj}>0$ 
and two discrete probability measures represented by vectors $p_1 \in \mathbb{R}^L$ and $p_0 \in \mathbb{R}^M$, respectively. 
Again we assume for simplicity that all entries in $p_1$ and $p_0$ are strictly positive. 
We introduce the set of all bi-stochastic $L\times M$ matrices 
with those discrete probability measures as marginals, that is,
\begin{equation} \label{eq:constraint_Sinkhorn}
\Pi_{\rm s} := \left\{ P \in \mathbb{R}^{L\times M}: \,P \ge 0,\, P^\T \mathbb{1}_L  = p_0,\, P\mathbb{1}_M = p_1\,\right\} .
\end{equation}
Solving Schr\"odinger's system \R{eq:SS1a}--\R{eq:SS1d} 
corresponds to finding two non--negative vectors $u \in \mathbb{R}^L$ and $v\in \mathbb{R}^M$
such that
\begin{equation*}
P^\ast := D(u)\,Q\,D(v)^{-1} \in \Pi_{\rm s}\,.
\end{equation*}
In turns out that $P^\ast$ is uniquely determined and minimises the Kullback-Leibler divergence between 
all $P\in \Pi_{\rm s}$ and the reference matrix $Q$, that is,
\begin{equation} \label{eq:disc_KL}
P^\ast = \arg \min_{P\in \Pi_{\rm s}} {\rm KL}\,(P||Q)\,.
\end{equation}
See \citeasnoun{sr:PC18} and the following remark for more details.

\medskip

\begin{remark} \label{rem:Var_S}
If we make the {\it ansatz}
\begin{equation*}
p_{lj} = \frac{u_l q_{lj}}{v_j} ,
\end{equation*}
then the minimisation problem \R{eq:disc_KL} becomes equivalent to
\begin{equation*}
P^\ast = \arg \min_{(u,v)>0} \sum_{l,j} p_{lj} \left( \log u_l - \log v_j\right)
\end{equation*}
subject to the additional constraints
\begin{equation*}
P\, \mathbb{1}_M = D(u)QD(v)^{-1} \,\mathbb{1}_M = p_1,\, P^\T \,\mathbb{1}_L = D(v)^{-1} Q^\T D(u)\,\mathbb{1}_L = p_0\,.
\end{equation*}
Note that these constraint determine $u> 0$ and $v> 0$ up to a common scaling factor. Hence
\R{eq:disc_KL} can be reduced to finding $(u,v) >0$ such that
\begin{equation*}
u^\T \mathbb{1}_L  = 1,\quad P\mathbb{1}_M = p_1, \quad P^\T \mathbb{1}_L = p_0 .
\end{equation*}
Hence we have shown that solving the Schr\"odinger system is equivalent to solving the minimisation problem \R{eq:disc_KL}
for discrete measures. Thus
\begin{equation*}
\min_{P\in \Pi_{\rm s}} {\rm KL}\,(P||Q) = p_1^\T \log u - p_0^\T \log v .
\end{equation*}
\end{remark}

\medskip

\begin{lemma} \label{lemma_Sinkhorn}
The Sinkhorn iteration \cite{sr:S67}
\begin{align} \label{eq:Sinkhorn_al1a}
u^{k+1} &:= D(P^{2k} \mathbb{1}_M)^{-1} \,p_1,\\ \label{eq:Sinkhorn_al1b}
P^{2k+1} &:= D(u^{k+1})\,P^{2k} ,\\ \label{eq:Sinkhorn_al1c}
v^{k+1} &:= D(p_0)^{-1}\,(P^{2k+1})^\T \, \mathbb{1}_L ,\\
P^{2k+2} &:= P^{2k+1} \,D(v^{k+1})^{-1} ,\label{eq:Sinkhorn_al1d}
\end{align}
$k=0,1,\ldots$, with initial $P^0 = Q \in \mathbb{R}^{L\times M}$ provides an algorithm for computing $P^\ast$, that is,
\begin{equation} \label{eq:convergence_Sinkhorn}
\lim_{k\to \infty} P^k = P^\ast .
\end{equation}
\end{lemma}

\medskip

\begin{proof} See, for example, 
\citeasnoun{sr:PC18} for a proof of \R{eq:convergence_Sinkhorn},
which is based on the contraction property of the iteration \R{eq:Sinkhorn_al1a}--\R{eq:Sinkhorn_al1d} with respect to
the Hilbert metric on the projective cone of positive vectors. 
\end{proof}

\medskip

\noindent
It follows from \R{eq:convergence_Sinkhorn} that
\begin{equation*}
\lim_{k\to \infty} u^k = \mathbb{1}_L, \qquad \lim_{k\to \infty} v^k = \mathbb{1}_M .
\end{equation*}
The essential idea of the Sinkhorn iteration is to enforce 
\begin{equation*}
P^{2k+1} \, \mathbb{1}_M = p_1,\qquad (P^{2k})^\T \,\mathbb{1}_L = p_0
\end{equation*}
at each iteration step and that  the matrix $P^\ast$ satisfies both constraints simultaneously in the limit $k\to \infty$.
See \citeasnoun{sr:cuturi13}  for a computationally efficient
and robust implementation of the Sinkhorn iteration.

\medskip

\begin{remark} One can introduce a similar iteration for the Schr\"odinger system \R{eq:SS1a}--\R{eq:SS1d}.
For example, pick $\widehat{\psi}(z_0) = 1$ initially. Then \R{eq:SS1a} implies $\pi_0^\psi = \pi_0$ and \R{eq:SS1c}
$\pi^\psi_1 = \pi_1$. Hence $\psi = l/\beta$ in the first iteration. The second iteration starts with $\widehat{\psi}$ determined
by \R{eq:SS1d} with $\psi = l/\beta$. We again cycle through \R{eq:SS1a}, \R{eq:SS1c} and \R{eq:SS1b} in order to find the next approximation
to $\psi$. The third iteration takes now this $\psi$ and computes the associated $\widehat{\psi}$ from \R{eq:SS1d} \etc \,A numerical implementation
of this procedure requires the approximation of two integrals which essentially leads back to a Sinkhorn type algorithm in the weights of
an appropriate quadrature rule.
\end{remark}

%
%
%
\section{Numerical methods} \label{sec:numerics}
%
%
%
%

Having summarised the relevant mathematical foundation for prediction, filtering (data assimilation 
scenario (A)) and smoothing (scenario (B)), 
and the Schr\"odinger problem (scenario (C)), we now discuss numerical approximations suitable for ensemble-based data assimilation. 
It is clearly impossible to cover all available methods, and we will focus on a selection of approaches which are built around the idea of
optimal transport, ensemble transform methods and Schr\"odinger systems. We will also focus on methods that 
can be applied or extended to problems with high-dimensional state spaces even though we will not explicitly cover this topic
in this survey. See \citeasnoun{sr:reichcotter15}, \citeasnoun{sr:leeuwen15}, and \citeasnoun{sr:ABN16} instead.

\subsection{Prediction}

Generating samples from the forecast distributions $q_+(\cdot\,|z_0^i)$ is in most cases straightforward. The computational
expenses can, however, vary dramatically, and this impacts on the choice of algorithms for sequential data assimilation.
We demonstrate in this subsection how samples from the prediction PDF $\pi_1$ can be used to construct an associated 
finite state Markov chain that transforms $\pi_0$ into an empirical approximation of $\pi_1$.

\medskip

\begin{definition} \label{def:sample_based1}
Let us assume that we have $L\ge M$ independent samples $z_1^l$ from the $M$ 
forecast distributions $q_+(\cdot\,|z_0^j)$, $j=1,\ldots,M$. We introduce the $L \times M$ matrix $Q$ with 
entries
\begin{equation} \label{eq:discrete_forward}
q_{lj} := q_+(z_1^l|z_0^j) .
\end{equation}
We now consider the associated bi-stochastic matrix $P^\ast \in \mathbb{R}^{L\times M}$, 
as defined by \R{eq:disc_KL}, with the two probability vectors in \R{eq:constraint_Sinkhorn} given by $p_1 = \mathbb{1}_L/L
\in \mathbb{R}^L$ and 
$p_0 = \mathbb{1}_M/M \in \mathbb{R}^M$, respectively. The finite-state Markov chain
\begin{equation} \label{eq:discrete_Markov_kernel}
Q_+ := M P^\ast
\end{equation}
provides a sample-based approximation to the forward transition kernel $q_+(z_1|z_0)$.
\end{definition}

\medskip

\noindent
More precisely, the $i$th column of $Q_+$ provides an empirical approximation to $q_+(\cdot|z_0^i)$ and 
\begin{equation*}
Q_+ \,p_0 = p_1 = \Frac{1}{L} \mathbb{1}_L ,
\end{equation*}
which is in agreement with the fact that the $z_1^l$ are equally weighted samples from the forecast PDF $\pi_1$.

\medskip

\begin{remark} 
Because of the simple relation between a bi-stochastic matrix $P \in \Pi_{\rm s}$ with $p_0$ in 
\R{eq:constraint_Sinkhorn} given by $p_0 = \mathbb{1}_M/M$ and its associated finite-state Markov chain $Q_+ = MP$, 
one can reformulate the minimisation problem \R{eq:disc_KL} in those cases directly in terms of Markov chains $Q_+ \in \Pi_{\rm M}$ 
with the definition of $\Pi_{\rm s}$ adjusted to
\begin{equation} \label{eq:Pi_M}
\Pi_{\rm M} := \left\{ Q \in \mathbb{R}^{L\times M}: \,Q \ge 0,\, Q^\T \mathbb{1}_L  = \mathbb{1}_M ,\, \Frac{1}{M} Q\mathbb{1}_M = p_1\,\right\} .
\end{equation}
\end{remark}

\begin{remark}
The associated backward transition kernel $Q_- \in \mathbb{R}^{M\times L}$ satisfies
\begin{equation*}
Q_- \,D(p_1)  = (Q_+ \,D(p_0))^\T
\end{equation*}
and is hence given by
\begin{equation*}
Q_- =  (Q_+\,D(p_0))^\T \,D(p_1)^{-1} = \frac{L}{M} Q_+^\T .
\end{equation*}
Thus
\begin{equation*}
Q_- \,p_1 = D(p_0) \,Q_+^\T \,\mathbb{1}_L = D(p_0)\,\mathbb{1}_M = p_0 ,
\end{equation*}
as desired.
\end{remark}

\begin{definition}
We can extend the concept of twisting to discrete Markov chains such as \R{eq:discrete_Markov_kernel}. A twisting
potential $\psi$ gives rise to a vector $u \in \mathbb{R}^L$ with normalised entries
\begin{equation*}
u_l = \frac{\psi(z_1^l)}{\sum_{k=1}^L \psi(z_1^k)} ,
\end{equation*}
$l=1,\ldots,L$. The twisted finite-state Markov kernel is now defined by 
\begin{equation} \label{eq:twisted_discrete_Markov_kernel}
Q^\psi_+ := D(u) \,Q_+ \,D(v)^{-1}, \qquad v := (D(u)\,Q_+)^\T \,\mathbb{1}_L \in \mathbb{R}^M ,
\end{equation}
and thus $\mathbb{1}_L^\T\, Q^\psi_+ = \mathbb{1}_M^\T$, as required for a Markov kernel. The twisted forecast 
probability is given by 
\begin{equation*}
p_1^\psi :=Q^\psi_+ \,p_0
\end{equation*}
with $p_0 = \mathbb{1}_M/M$. Furthermore, if we set $p_0 = v$ then $p_1^\psi = u$.
\end{definition}


\subsubsection{Gaussian model errors (cont.)}

The proposal density is given by \R{eq:Gaussian_kernel} and it is easy to produce $K> 1$ samples from
each of the $M$ proposals $q_+(\cdot\,|z_0^j)$. Hence we can make the total sample size $L = K\,M$ 
as large as desired. In order to produce $M$ samples, $\widetilde{z}_1^j$, from a twisted finite-state Markov chain
\R{eq:twisted_discrete_Markov_kernel}, we draw a single realisation from each of the $M$ associated 
discrete random variables $\widetilde{Z}_1^j$, $j=1,\ldots,M$, 
with probabilities
\begin{equation*}
\mathbb{P}[\widetilde{Z}^j_1(\omega) = z_1^l] = (Q_+^\psi)_{lj} .
\end{equation*}
We will provide more details when discussing the Schr\"odinger problem in the context
of Gaussian model errors in Section \ref{sec:Gaussian_Schrodinger}.


\subsubsection{SDE models (cont.)}  \label{sec:SDE_model_pred}
The Euler-Maruyama method \cite{sr:Kloeden}
\begin{equation} \label{eq:EMM1}
Z_{n+1}^+ = Z_{n}^+ + f_{t_{n}}(Z_{n}^+) \, \Delta t+ (\gamma \Delta t)^{1/2} \, \Xi_{n} ,\qquad \Xi_{n} \sim {\rm N}(0,I),
\end{equation}
$n=0,\ldots,N-1$, will be used for the numerical approximation of \R{eq:Forward-SDE} 
with step-size $\Delta t := 1/N$, $t_n = n\,\Delta t$.  In other words, we replace $Z_{t_n}^+$ with its numerical approximation
$Z_n^+$. A numerical approximation (realisation) of the whole solution path $z_{[0,1]}$ 
will be denoted by $z_{0:N}  = Z_{0:N}^+(\omega)$ and can be 
computed recursively due to the Markov property of the Euler-Maruyama scheme. 
The marginal PDFs of $Z_n^+$ are denoted by $\pi_n$.

For any finite number of time-steps $N$, we can define a joint 
PDF $\pi_{0:N}$ on $\mathcal{U}_N = \mathbb{R}^{N_z\times (N+1)}$ via
\begin{equation} \label{eq:measure_filter_SDE}
\pi_{0:N}(z_{0:N}) \propto \exp \left( -\frac{1}{2\Delta t} \sum_{n=0}^{N-1} \| \eta_n \|^2
\right) \pi_0(z_0)
\end{equation}
with 
\begin{equation} \label{eq:residual}
\eta_n := \gamma^{-1/2} \left(z_{n+1} - z_{n} -f_{t_{n}}(z_{n})\,\Delta t \right)
\end{equation}
and $\eta_n = \Delta t^{1/2} \Xi_n(\omega)$. Note that the joint PDF $\pi_{0:N}(z_{0:N})$ can also be expressed in
terms of $z_0$ and $\eta_{0:N-1}$.

The numerical approximation of SDEs provides an example for which the increase in computational cost for
producing $L>M$ samples from the PDF $\pi_{0:N}$  versus $L=M$ is non-trivial, in general. 

%

\medskip

We now extend Definition \ref{def:sample_based1} to the case of temporally discretised SDEs in the form of
\R{eq:EMM1}. 

\begin{definition} \label{def:sample_based2}
Let us assume that we have $L=M$ independent numerical solutions $z_{0:N}^i$ of \R{eq:EMM1}. We introduce
an $M\times M$ matrix $Q_n$ for each $n=1,\ldots,N$ with entries
\begin{equation*}
q_{lj} = q_+(z_{n}^l|z_{n-1}^j) := {\rm n}(z_n^l;z_{n-1}^j+\Delta t f(z_{n-1}^j),\gamma \Delta t\,I) .
\end{equation*}
With each $Q_n$ we associate a finite-state Markov chain $Q^+_n$ as defined by
\R{eq:discrete_Markov_kernel} for general transition densities $q_+$ in Definition \ref{def:sample_based1}. An approximation
of the Markov transition from time $t_0=0$ to $t_1=1$ is now provided by
\begin{equation} \label{eq:SDE_transition}
Q_+ := \prod_{n=1}^N Q^+_n\,.
\end{equation}
\end{definition}

\medskip

\begin{remark} The approximation \R{eq:discrete_Markov_kernel} can be 
related to the diffusion map approximation of the infinitesimal generator
of Brownian dynamics 
\begin{equation} \label{eq:BD5}
{\rm d}Z_t^+ = -\nabla_z U(Z_t^+)\,{\rm d}t + \sqrt{2}\, \dd W_t^+
\end{equation}
with potential $U(z) = - \log \pi^\ast (z)$ in the following sense. 
First note that $\pi^\ast$ is invariant under the associated Fokker--Planck equation
\R{eq:FPE1} with (time-independent) operator ${\cal L}^\dagger$ written in the form
\begin{equation*}
{\cal L}^\dagger \pi =  \nabla_z \cdot \left( \pi^\ast \nabla_z \frac{\pi}{\pi^\ast}\right) .
\end{equation*}
Let $z^i$, $i=1,\ldots,M$, denote $M$ samples from the invariant PDF $\pi^\ast$ and define the symmetric matrix $Q\in
\mathbb{R}^{M\times M}$ with entries
\begin{equation*}
q_{lj} = {\rm n}(z^l;z^j,2\Delta t\, I) .
\end{equation*}
Then the associated (symmetric) matrix \R{eq:discrete_Markov_kernel}, as introduced in definition \ref{def:sample_based1}, 
provides a discrete approximation to the evolution of a probability vector $p_0 \propto \pi_0/\pi^\ast$ over a time-interval 
$\Delta t$ and, hence, to the semigroup operator ${\rm e}^{\Delta t {\cal L}}$ with the infinitesimal generator ${\cal L}$ given by
\begin{equation} \label{eq:diffusion_map1}
{\cal L} g= \frac{1}{\pi^\ast} \nabla_z \cdot (\pi^\ast \nabla_z  g) .
\end{equation}
We formally obtain
\begin{equation} \label{eq:diffusion_map2}
{\cal L} \approx \frac{Q_+ - I}{\Delta t}
\end{equation}
for $\Delta t$ sufficiently small. The symmetry of $Q_+$ reflects the fact
that ${\cal L}$ is self-adjoint with respect to the weighted inner product
\begin{equation*}
\langle f,g\rangle_{\pi^\ast} = \int f(z)\,g(z)\,\pi^\ast (z)\,\dd z .
\end{equation*}
See \citeasnoun{sr:H18} for a discussion of alternative diffusion map approximations to the infinitesimal generator
${\cal L}$ and Appendix A for an application to the feedback particle filter formulation of continuous-time 
data assimilation.
\end{remark}

\medskip

\noindent
We also consider the discretisation 
\begin{equation} \label{eq:EMM2}
Z_{n+1}^+ = Z_{n}^+ + \left( f_{t_{n}}(Z_{n}^+) + u_{t_{n}}(Z_{n}^+) \right)
\Delta t+ (\gamma \Delta t)^{1/2}  
\,\Xi_{n} ,
\end{equation}
$n=0,\ldots,N-1$, of a controlled SDE \R{eq:Forward-SDE2} with associated PDF $\pi^u_{0:N}$ defined by
\begin{equation} \label{eq:measure_filter_SDE2}
\pi_{0:N}^u(z_{0:N}^u) \propto \exp \left( -\frac{1}{2\Delta t} \sum_{n=0}^{N-1} \| \eta^u_n \|^2 ,
\right) \pi_0(z_0)
\end{equation}
where
\begin{align*}
\eta_n^u &:= \gamma^{-1/2} \left\{ z_{n+1}^u- z_{n}^u - \left( f_{t_{n}}(z_{n}^u) + u_{t_{n}}(z_{n}^u)\right)\Delta t \right\}\\
&= \eta_n - \Frac{\Delta t}{\gamma^{1/2}} u_{t_n}(z_n^u) .
\end{align*}
Here $z_{0:N}^u$ denotes a realisation of the discretisation \R{eq:EMM2} with control laws $u_{t_n}$.
We find that
\begin{align*}
\frac{1}{2\Delta t} \|\eta_n^u\|^2 &= \frac{1}{2\Delta t} \|\eta_n\|^2 -
\frac{1}{\gamma^{1/2}} u_{t_{n}}(z_{n}^u)^\T \eta_n + \frac{\Delta t}{2\gamma} \|u_{t_n}(z_n^u)\|^2\\
&= \frac{1}{2\Delta t} \|\eta_n\|^2 - \frac{1}{\gamma^{1/2}} u_{t_{n}}(z_{n}^u)^\T \eta_n^u - \frac{\Delta t}{2\gamma} \|u_{t_n}(z_n^u)\|^2 ,
\end{align*}
and hence
\begin{equation} \label{eq:dRND}
\frac{\pi_{0:N}^u(z_{0:N}^u)}{\pi_{0:N}(z_{0:N}^u)} =
\exp \left( \frac{1}{2\gamma} \sum_{n=0}^{N-1} \left( \|u_{t_n}(z_n^u)\|^2\Delta t + 2 \gamma^{1/2}
u_{t_n}(z_n^u)^\T \eta_n^u\right) \right) ,
\end{equation}
which provides a discrete version of \R{eq:RND} since $\eta_n^u = \Delta t^{1/2} \,\Xi_n(\omega)$ are increments
of Brownian motion over time intervals of length $\Delta t$. 

\medskip

\begin{remark} Instead of discretising the forward SDE \R{eq:Forward-SDE} in order to produce samples
from the forecast PDF $\pi_1$, one can also start from the mean-field formulation \R{eq:mODE} and its
time discretisation, for example,
\begin{equation} \label{eq:interaction_prediction}
z_{n+1}^i = z_{n}^i + \left( f_{t_{n}}(z_{n}^i) + u_{t_{n}}(z_{n}^i) \right)
\Delta t
\end{equation}
for $i=1,\ldots,M$ and
\begin{equation*}
u_{t_{n}}(z) = -\frac{\gamma}{2} \nabla_z \log \widetilde{\pi}_n(z) .
\end{equation*}
Here $\widetilde{\pi}_n$ stands for an approximation to the marginal PDF $\pi_{t_n}$ based on the available 
samples $z_{n}^i$, $i=1,\ldots,M$. A simple approximation is obtained by the Gaussian PDF
\begin{equation*}
\widetilde{\pi}_n(z) = {\rm n}(z;\bar{z}_n,P_n^{zz})
\end{equation*}
with empirical mean
\begin{equation*}
\bar{z}_n = \frac{1}{M} \sum_{i=1}^M z_n^i
\end{equation*}
and empirical covariance matrix
\begin{equation*}
P_n^{zz} = \frac{1}{M-1} \sum_{i=1}^M z_n^i (z_n^i - \bar{z}_n)^\T .
\end{equation*}
The system \R{eq:interaction_prediction} becomes
\begin{equation*}
z_{n+1}^i = z_{n}^i + \left( f_{t_{n}}(z_{n}^i) + \gamma (P_n^{zz})^{-1}(z_{n}^i-\bar{z}_n) \right)\Delta t ,
\end{equation*}
$i=1,\ldots,M$, and provides an example of an interacting particle approximation. Similar mean-field
formulations can be found for the backward SDE \R{eq:Backward-SDE}.
\end{remark}

\subsection{Filtering} \label{sec:num_filtering}

Let us assume that we are given $M$ samples, $z_1^i$, from the forecast PDF
using forward transition kernels $q_+(\cdot\,|z_0^i)$, $i=1,\ldots,M$.
The likelihood function $l(z)$ leads to importance weights
\begin{equation} \label{eq:importance_weights_filtering1}
w^i \propto l(z_1^i) .
\end{equation}
We also normalise these importance weights such that \R{eq:normalised_w} holds.

\begin{remark}
The model evidence, $\beta$, can be estimated from the samples, $z_1^i$, and the likelihood, $l(z)$, as follows:
\begin{equation*}
\widetilde{\beta}  := \frac{1}{M} \sum_{i=1}^M l(z_1^i) .
\end{equation*}
If the likelihood is of the form
\begin{equation*}
l(z) \propto \exp \left( -\Frac{1}{2} (y_1-h(z))^\T R^{-1} (y_1-h(z) \right)
\end{equation*}
and the prior distribution in $y = h(z)$ can be approximated as being Gaussian with covariance
\begin{equation*}
P^{hh} := \frac{1}{M-1} \sum_{i=1}^M h(z_1^i)(h(z_1^i)-\bar{h})^\T, \qquad \bar{h} := \frac{1}{M} \sum_{i=1}^M
h(z_1^i) ,
\end{equation*}
then the evidence can be approximated by
\begin{equation*}
\widetilde{\beta} \approx \frac{1}{(2\pi)^{N_y/2} |P^{yy}|^{1/2}} \exp \left( -\Frac{1}{2} (y_1-\bar{h})^\T (P^{yy})^{-1} (y_1-
\bar{h})\right)
\end{equation*}
with
\begin{equation*}
P^{yy}:= R + P^{hh} .
\end{equation*}
Such an approximation has been used, for example, in \citeasnoun{sr:CBHG17}. See also \citeasnoun{sr:reichcotter15}
for more details on how to compute and use model evidence in the context of sequential data assimilation.
\end{remark}

\medskip

\noindent
Sequential data assimilation requires that we produce $M$ equally weighted samples $\widehat{z}_1^j \sim \widehat{\pi}_1$ 
from the $M$ weighted samples $z_1^i \sim \pi_1$ with weights $w^i$. This is a standard problem in Monte Carlo integration and 
there are many ways to tackle this problem, among which are multinomial, 
residual, systematic and stratified resampling \cite{sr:DC05}. Here we focus on those resampling methods which 
are based on a discrete Markov chain $P \in \mathbb{R}^{M\times M}$ with the property that
\begin{equation} \label{eq:resampling_MC1}
w = \Frac{1}{M} P\,\mathbb{1}_M, \qquad w = \left(\Frac{w^1}{M},\ldots,\Frac{w^M}{M}\right)^\T .
\end{equation}
The Markov property of $P$ implies that $P^\T \mathbb{1}_M = \mathbb{1}_M$. We now consider the set of all Markov chains
$\Pi_{\rm M}$, as defined by \R{eq:Pi_M}, with $p_1 = w$. Any Markov chain $P\in \Pi_{\rm M}$ can now be used for resampling, but 
we seek the Markov chain $P^\ast \in \Pi_{\rm M}$ which minimises the expected
distance between the samples, that is,
\begin{equation} \label{eq:optimal_transport1}
P^\ast = \arg \min_{P\in \Pi_{\rm M}} \sum_{i,j=1}^M p_{ij} \|z_1^i-z_1^j\|^2 .
\end{equation}
Note that \R{eq:optimal_transport1} is a special case of the optimal transport problem \R{eq:optimal_coupling1} 
with the involved probability measures being discrete measures. Resampling can now be performed according to
\begin{equation} \label{eq:optimal_transport2}
\mathbb{P}[\widehat{Z}_1^j (\omega)= z_1^i] = p_{ij}^\ast
\end{equation}
for $j = 1,\ldots,M$.  

Since, it is known that \R{eq:optimal_transport1}
converges to \R{eq:optimal_coupling1} as $M\to \infty$ \cite{sr:mccann95} and since \R{eq:optimal_coupling1} leads
to a transformation \R{eq:optimal_coupling3}, the resampling step \R{eq:optimal_transport2} has been replaced by
\begin{equation} \label{eq:optimal_transport3}
\widehat{z}_1^j = \sum_{i=1}^M z_1^i \,p_{ij}^\ast
\end{equation}
in the so-called ensemble transform particle filter (ETPF) \cite{sr:reich13,sr:reichcotter15}. In 
other words, the ETPF replaces resampling with probabilities $p_{ij}^\ast$ by its mean \R{eq:optimal_transport3} 
for each $j=1,\ldots,M$. The ETPF leads to a biased  approximation to the
resampling step which is consistent in the limit $M\to \infty$

The general formulation \R{eq:optimal_transport3} with the coefficients $p_{ij}^\ast$ chosen appropriately\footnote{The coefficients
$p_{ij}^\ast$ of an ensemble transform particle filter do not need to be non-negative and only satisfy 
$\sum_{i=1}^M p_{ij}^\ast = 1$ \cite{sr:AdWR17}.} leads to a large class of so-called ensemble transform particle filters
\cite{sr:reichcotter15}.  Ensemble transform particle filters generally result in biased and inconsistent but robust estimates 
which have found applications  to high-dimensional state space models \cite{sr:evensen,sr:survey18} for which traditional 
particle filters fail due to the \lq curse of dimensionality\rq{} \cite{sr:bengtsson08}. More specifically,  the class of ensemble transform 
particle filters includes the popular ensemble Kalman filters  \cite{sr:evensen,sr:reichcotter15,sr:survey18,sr:CBBE18} and
so-called second-order accurate particle filters with coefficients $p_{ij}^\ast$  in \R{eq:optimal_transport3} chosen such that 
the weighted ensemble mean
\begin{equation*}
\bar z_1 := \frac{1}{M} \sum_{i=1}^M w^i z_1^i
\end{equation*}
and the weighted ensemble covariance matrix
\begin{equation*}
\widetilde{P}^{zz} := \frac{1}{M}\sum_{i=1}^M w^i (z_1^i-\bar z_1)(z_1^i-\bar z_1)^\T
\end{equation*}
are exactly reproduced by the transformed and equally weighted particles $\widehat{z}_1^j$, $j=1,\ldots,M$, defined by
\R{eq:optimal_transport3}, that is,
\begin{equation*}
\frac{1}{M}\sum_{j=1}^M \widehat{z}_1^j = \bar z_1, \qquad \frac{1}{M-1}\sum_{j=1}^M
(\widehat{z}_1^i-\bar z_1)(\widehat{z}_1^i-\bar z_1)^\T = \widetilde{P}^{zz} .
\end{equation*}
See the survey paper by \citeasnoun{sr:survey18} and the paper by \citeasnoun{sr:AdWR17} for more details. A summary of
the ensemble Kalman filter can be found in Appendix C.

In addition, hybrid methods \cite{sr:frei13,sr:CRR15}, which bridge between classical particle filters and the ensemble Kalman filter, 
have recently been successfully applied to atmospheric fluid dynamics \cite{sr:RLK18}.

\begin{remark}
Another approach for transforming samples, $z_1^i$, from the forecast PDF $\pi_1$ into samples, $\widehat{z}_1^i$, 
from the filtering PDF $\widehat{\pi}_1$ 
is provided through the mean-field interpretation 
\begin{equation} \label{eq:IP_FP}
\frac{\dd}{\dd s} \breve{Z}_s = - \nabla_z \log \frac{\breve{\pi}_s(\breve{Z}_s)}{\widehat{\pi}_1(\breve{Z}_s)}\,,
\end{equation} 
of the Fokker-Planck equation \R{eq:FPE1} for a random variable $\breve{Z}_s$ with law  $\breve{\pi}_s$, 
drift term $f_s(z) = \nabla_z \log \widehat{\pi}_1$ and $\gamma = 2$, that is,
\begin{equation*}
\partial_s \breve{\pi}_s = \nabla_z \cdot \left( \breve{\pi}_s \nabla_z \log \frac{\breve{\pi}_s}{\widehat{\pi}_1} \right)
\end{equation*}
in artificial time $s\ge 0$. It holds under fairly general assumptions that
\begin{equation*}
\lim_{s \to \infty} \breve{\pi}_s = \widehat{\pi}_1
\end{equation*}
\cite{sr:P14} and one can set $\breve{\pi}_0 = \pi_1$. The more common approach would be to solve Brownian dynamics
\begin{equation*}
\dd \breve{Z}_s = \nabla_z \log \widehat{\pi}_1(Z_s)\,\dd t + \sqrt{2}\dd W_s^+
\end{equation*}
for each sample, $z_1^i$, from $\pi_1$, that is, $\breve{Z}_0(\omega) = z_1^i$, $i=1,\ldots,M$, at initial time and 
\begin{equation*}
\widehat{z}_1^i = \lim_{s\to \infty} \breve{Z}_s (\omega) .
\end{equation*}
In other words, formulation \R{eq:IP_FP} replaces stochastic Brownian dynamics with a deterministic 
interacting particle system. See Appendix A and Remark \ref{remark_opt} for further details.
\end{remark}

\subsection{Smoothing} \label{sec:num_smoothing}

Recall that the joint smoothing distribution $\widehat{\pi}(z_0,z_1)$ can be represented in the form
\R{eq:optimal_proposal1} with modified transition kernel \R{eq:smoothing_kernel} and smoothing distribution
\R{eq:smoothing_PDF0} at time $t_0$ with weights $\gamma^i$ determined by \R{eq:gamma_i}. 

Let us assume that it is possible to sample from $\widehat{q}_+(z_1|z_0^i)$ and that the weights $\gamma^i$ are
available. Then we can utilise \R{eq:optimal_proposal1} in sequential data assimilation as follows.
We first resample the $z_0^i$ at time $t_0$ using a discrete Markov chain $P \in \mathbb{R}^{M\times M}$
satisfying
\begin{equation} \label{eq:resampling_MC0}
\widehat{p}_0 = \Frac{1}{M} P\,\mathbb{1}_M\,, \quad \widehat{p}_0 := \gamma ,
\end{equation}
with $\gamma$ defined in \R{eq:definition_eM}.
Again optimal transportation can be used to identify a suitable $P$. More explicitly, 
we now consider the set of all Markov chains $\Pi_{\rm M}$, as defined by \R{eq:Pi_M}, with $p_1 = \gamma$. Then 
the Markov chain $P^\ast$ arising from the associated optimal transport problem \R{eq:optimal_transport1} can be used for resampling, 
that is,
\begin{equation*}
\mathbb{P}[\widetilde{Z}_0^j (\omega) = z_0^i] = p_{ij}^\ast .
\end{equation*}
Once equally weighted samples $\widehat{z}_0^i$, $i=1,\ldots,M$, from $\widehat{\pi}_0$ have been determined,
the desired samples $\widehat{z}_1^i$ from $\widehat{\pi}_1$ are simply given by
\begin{equation*}
\widehat{z}_1^i := \widehat{Z}_1^i(\omega), \qquad \widehat{Z}_1^i \sim \widehat{q}_+(\cdot\,|\widehat{z}_0^i) ,
\end{equation*}
for $i = 1,\ldots,M$.

The required transition kernels \R{eq:smoothing_kernel} are explicitly available for state space models with Gaussian model 
errors and Gaussian likelihood functions. In many other cases, these kernels are not explicitly available or are difficult to draw from. 
In such cases, one can resort to sample-based transition kernels.

For example, consider the twisted discrete Markov kernel \R{eq:twisted_discrete_Markov_kernel} with twisting potential
$\psi(z) = l(z)$. The associated vector $v$ from \R{eq:twisted_discrete_Markov_kernel} 
then gives rise to a probability vector  $\widehat{p}_0 = c\,v \in \mathbb{R}^M$ with $c>0$ an appropriate scaling factor, and
\begin{equation} \label{eq:discrete_smoothing_step}
\widehat{p}_1 := Q^\psi_+ \widehat{p}_0
\end{equation}
approximates the filtering distribution at time $t_1$. The Markov transition matrix $Q_+^\psi \in \mathbb{R}^{L\times M}$ 
together with $\widehat{p}_0$ provides an approximation to the smoothing kernel $\widehat{q}_+(z_1|z_0)$ 
and $\widehat{\pi}_0$, respectively. 

The approximations $Q_+^\psi \in \mathbb{R}^{L\times M}$ and $\widehat{p}_0 \in \mathbb{R}^M$ 
can be used to first generate equally weighted samples 
$\widehat{z}_0^i\in \{z_0^1,\ldots,z_0^M\}$ with distribution $\widehat{p}_0$ via, for example, resampling with replacement. 
If $\widehat{z}_0^i = z_0^k$ for an index $k=k(i) \in \{1,\ldots,M\}$, then 
\begin{equation*} 
\mathbb{P}[\widehat{Z}_1^i(\omega) = z_1^l] = (Q_+^\psi)_{lk}
\end{equation*}
for each $i=1,\ldots,M$. The $\widehat{z}_1^i$ are equally weighted samples  from the discrete filtering distribution 
$\widehat{p}_1$, which is an approximation to the continuous filtering PDF $\widehat{\pi}_1$.


\medskip

\begin{remark}
One has to take computational complexity and robustness 
into account when deciding whether to utilise methods from Section \ref{sec:num_filtering} or
this subsection to advance $M$ samples $z_0^i$ from the prior distribution $\pi_0$
into $M$ samples $\widehat{z}_1^i$ from the posterior distribution $\widehat{\pi}_1$. While the methods from
section \ref{sec:num_filtering} are easier to implement, the methods of this subsection benefit
from the fact that
\begin{equation*}
M > \frac{1}{ \|\gamma\|^2} \ge \frac{1}{\|w\|^2} \ge 1,
\end{equation*}
in general, where the importance weights $\gamma \in \mathbb{R}^M$ and $w \in \mathbb{R}^M$ are defined 
in \R{eq:definition_eM} and \R{eq:resampling_MC1}, respectively. 
In other words, the methods from this subsection lead to larger effective sample sizes \cite{sr:Liu,sr:APPSS17}.
\end{remark}

\begin{remark}
We mention that finding efficient methods for solving the more general smoothing problem \R{eq:Smoothing} is an active 
area of research. See, for example, the recent contributions by \citeasnoun{sr:GJL17} and \citeasnoun{sr:HBDD18} for 
discrete-time Markov processes, and \citeasnoun{sr:KR16} as well as \citeasnoun{sr:RK17} for smoothing in the context of SDEs.
Ensemble transform methods of the form \R{eq:optimal_transport3} can also be extended to the general smoothing problem. 
See, for example, \citeasnoun{sr:evensen} and \citeasnoun{sr:CBBE18} for extensions of the ensemble Kalman filter, 
and \citeasnoun{sr:KTAN17} for an extension of the nonlinear ensemble transform filter to the smoothing problem.
\end{remark}


\subsubsection{SDE models (cont.)}
After discretization in time, smoothing leads to a change from the forecast PDF (\ref{eq:measure_filter_SDE}) to
\begin{align*}
\widehat{\pi}_{0:N}(z_{0:N}) &:= \frac{l(z_N)\pi_{0:N}(z_{0:N})}{\pi_{0:N}[l]} \\
& \propto \exp \left( -\frac{1}{2 \Delta t} \sum_{n=0}^{N-1} \| \xi_n \|^2 
\right) \pi_0(z_0) \,l(z_N)
\end{align*}
with $\xi_n$ given by (\ref{eq:residual}),
or, alternatively,
\begin{equation*}
\frac{\widehat{\pi}_{0:N}}{\pi_{0:N}}(z_{0:N}) = \frac{l(z_N)}{\pi_{0:N}[l]} .
\end{equation*}

\medskip

\begin{remark} Efficient MCMC methods for sampling high-dimensional smoothing distributions can be found 
in \citeasnoun{sr:BGLFS17} and \citeasnoun{sr:BPSSS11}. Improved sampling can also be achieved by using 
regularized St\"ormer--Verlet time-stepping
methods \cite{sr:RH11} in a hybrid Monte Carlo method \cite{sr:Liu}. See Appendix B form more details.
\end{remark}


\subsection{Schr\"odinger Problem}

Recall that the Schr\"odinger system \R{eq:SS1a}--\R{eq:SS1d} reduces in our context to solving 
equations \R{eq:Sinkhorn1a}--\R{eq:Sinkhorn1c} for the unknown coefficients $\alpha^i$, $i=1,\ldots,M$. 
In order to make this problem tractable we need to replace the required expectation values with respect to $q_+(z_1|z_0^j)$ 
by Monte Carlo approximations. More specifically, let us assume that we have $L\ge M$ samples 
$z_1^l$ from the forecast PDF $\pi_1$. The associated $L\times M$ matrix $Q$ with entries
\R{eq:discrete_forward} provides a discrete approximation to the underlying Markov process defined by $q_+(z_1|z_0)$ 
and initial PDF \R{eq:initial_pdf}. 

The importance weights in the associated approximation to the filtering distribution
\begin{equation*}
\widehat{\pi}_1(z) = \frac{1}{L} \sum_{l=1}^L w^l \,\delta (z-z_1^l)
\end{equation*}
are given by \R{eq:importance_weights_filtering1} with the weights normalised such that 
\begin{equation} \label{eq:IW_normalised}
\sum_{l=1}^L w^l = L .
\end{equation}
Finding the coefficients $\{\alpha^i\}$ in 
\R{eq:Sinkhorn1a}--\R{eq:Sinkhorn1c} can now be reformulated as finding
two vectors $u\in \mathbb{R}^{L}$ and $v\in \mathbb{R}^{M}$ such
that 
\begin{equation} \label{eq:Sinkhorn1}
P^\ast := D(u) Q D(v)^{-1}
\end{equation}
satisfies $P^\ast \in \Pi_{\rm M}$ with $p_1 = w$ in \R{eq:Pi_M}, that is, more explicitly
\begin{equation} \label{eq:SS_cond3}
\Pi_{\rm M} = \left\{ P\in \mathbb{R}^{L\times M}: \,p_{lj} \ge 0,\, \sum_{l=1}^L p_{lj} = 1, \,
\Frac{1}{M} \sum_{j=1}^M p_{lj} = \Frac{w^l}{L} \right\} .
\end{equation}
We note that \R{eq:SS_cond3} are discrete approximations to \R{eq:SS_cond1} and \R{eq:SS_cond2}, respectively. 
The scaling factor $\widehat{\psi}$ in \R{eq:SS1a} is approximated by the vector $v$ up to a normalisation constant,
while the vector $u$ provides an approximation to $\psi$ in \R{eq:SS1b}. Finally, the desired approximations to the 
Schr\"odinger transition kernels $q_+^\ast (z_1|z_0^i)$, $i=1,\ldots,M$, are provided by the columns of $P^\ast$, that is,
\begin{equation*}
\mathbb{P}[\widehat{Z}_1^i(\omega) = z_1^l] = p^\ast_{li}
\end{equation*}
characterises the desired equally weighted samples $\widehat{z}_1^i$, $i=1,\ldots,M$, from the filtering distribution $\widehat{\pi}_1$.
See the following subsection for more details.

The required vectors $u$ and $v$ can be computed using the iterative Sinkhorn algorithm 
\R{eq:Sinkhorn_al1a}--\R{eq:Sinkhorn_al1c} \cite{sr:cuturi13,sr:PC18}.

\medskip

\begin{remark} Note that one can replace the forward transition kernel $q_+(z_1|z_0)$ in \R{eq:discrete_forward} with
any suitable twisted prediction kernel \R{eq:q-twisted}. This results in a modified matrix $Q$ in \R{eq:Sinkhorn1} 
and weights $w^l$ in \ref{eq:SS_cond3}. The resulting matrix \R{eq:Sinkhorn1} still provides an approximation to
the Schr\"odinger problem. 
\end{remark}

\medskip

\begin{remark} The approximation \R{eq:Sinkhorn1} can be extended to an approximation of the Schr\"odinger forward
transition kernels \R{eq:SS_forward_kernel} in the following sense. We use $\alpha^i = 1/v_i$ in \R{eq:SS_forward_kernel}
and note that the resulting approximation satisfies \R{eq:SS_cond2} while \R{eq:SS_cond1} now longer holds exactly. However, since
the entries $u_l$ of the vector $u$ appearing in \R{eq:Sinkhorn1} satisfy
\begin{equation*}
u_l = \frac{w^l}{L} \frac{1}{\sum_{j=1}^M q_+(z_1^l|z_0^j)/v_j} ,
\end{equation*}
it follows that
\begin{align*}
\int q_+^\ast (z_1|z_0^i)\,\dd z_1 & 
\approx \frac{1}{L} \sum_{l=1}^L \frac{l(z_1^l)}{\beta} \frac{q_+(z_1^l|z_0^i)/v_i}{
\sum_{j=1}^M q_+(z_1^l|z_0^j)/v_j} \\
&\approx \frac{1}{L} \sum_{l=1}^L w^l \frac{q_+(z_1^l|z_0^i)/v_i}{
\sum_{j=1}^M q_+(z_1^l|z_0^j)/v_j} = \sum_{l=1}^L p_{li}^\ast = 1 .
\end{align*}
Furthermore, one can use such continuous approximations in combination with Monte Carlo sampling methods which do not require 
normalised target PDFs.
\end{remark}

\subsubsection{Gaussian model errors (cont.)} \label{sec:Gaussian_Schrodinger}

One can easily generate $L$, $L\ge M$, i.i.d.~samples $z_1^l$ from the forecast PDF
\R{eq:Gaussian_kernel}, that is,
\begin{equation*}
Z_1^l \sim  \frac{1}{M}\sum_{j=1}^M {\rm n}(\cdot\,;\Psi(z_0^j),\gamma B) ,
\end{equation*}
and with the filtering distribution $\widehat{\pi}_1$ characterised through the importance weights
\R{eq:importance_weights_filtering1}.

We define the distance matrix $D \in \mathbb{R}^{L\times M}$ with entries
\begin{equation*}
d_{lj} := \frac{1}{2} \|z_1^l-\Psi(z_0^j)\|_B^2, \qquad \|z\|_B^2 := z^\T B^{-1} z ,
\end{equation*}
and the matrix $Q  \in \mathbb{R}^{L\times M}$ with entries
\begin{equation*}
q_{lj} := {\rm e}^{-d_{lj}/\gamma} .
\end{equation*}
The Markov chain $P^\ast \in \mathbb{R}^{L\times M}$ is now given by
\begin{equation*}
P^\ast = \arg \min_{P\in \Pi_{\rm M}} {\rm KL}(P||Q)
\end{equation*}
with the set $\Pi_{\rm M}$ defined by \R{eq:SS_cond3}.

Once $P^\ast$ has been computed, the desired Schr\"odinger transitions from $\pi_0$ to $\widehat{\pi}_1$
can be represented as follows.  The Schr\"odinger transition kernels 
$q^\ast_+(z_1|z_0^i)$ are approximated for each $z^i_0$ by 
\begin{equation} \label{eq:discS}
\tilde{q}^\ast_+(z_1|z_0^i) := \sum_{l=1}^L p^\ast_{li} \,\delta(z_1-z_1^l)\,, \qquad i=1,\ldots,M .
\end{equation}
The empirical measure in (\ref{eq:discS}) converges weakly to the desired $q_+^\ast (z_1|z_0^i)$ as $L\to \infty$ and
\begin{equation*}
\widehat{\pi}_1(z_1) \approx \frac{1}{M} \sum_{i=1}^M \delta(z-\widehat{z}_1^i),
\end{equation*}
with
\begin{equation} \label{eq:discS2}
\widehat{z}_1^i = \widehat{Z}_1^i(\omega), \quad \widehat{Z}_1^i \sim \tilde{q}^\ast_+ (\cdot\,|z_0^i),
\end{equation}
provides the desired approximation of $\widehat{\pi}_1$ by $M$ equally weighted particles 
$\widehat{z}_1^i$, $i=1,\ldots,M$.

We remark that \R{eq:discS} has been used to produce the Schr\"odinger transition kernels for 
Example \ref{ex:example1} and the right panel of Figure \ref{fig:example1b} in particular. More specifically, we
have $M=11$ and used $L = 11\,000$. Since the particles $z_1^l \in \mathbb{R}$, $l=1,\ldots,L$, 
are distributed according to  the forecast PDF $\pi_1$, a function representation of 
$\tilde{q}^\ast_+(z_1|z_0^i)$ over all of $\mathbb{R}$ is provided by  interpolating $p^\ast_{li}$ onto 
$\mathbb{R}$ and multiplication of this interpolated function by $\pi_1(z)$.

For $\gamma\ll 1$, the measure in (\ref{eq:discS}) can also be approximated by a Gaussian measure with mean
\begin{equation*}
\bar{z}_1^i := \sum_{l=1}^L z_1^l p^\ast_{li} 
\end{equation*}
and covariance matrix $\gamma B$, that is, we replace (\ref{eq:discS2}) with
\begin{equation*}
\widehat{Z}_1^i \sim {\rm N}(\bar z_1^i,\gamma B)
\end{equation*}
for $i=1,\ldots,M$.

\subsubsection{SDE (cont.)}

One can also apply \R{eq:Sinkhorn1} in order to approximate the Schr\"odinger
problem associated with SDE models. We typically use $L=M$ in this case and utilise \R{eq:SDE_transition}
in place of $Q$ in \R{eq:Sinkhorn1}. The set $\Pi_{\rm M}$ is still given by \R{eq:SS_cond3}.

\begin{figure}
\begin{center}
\includegraphics[width=0.7\textwidth]{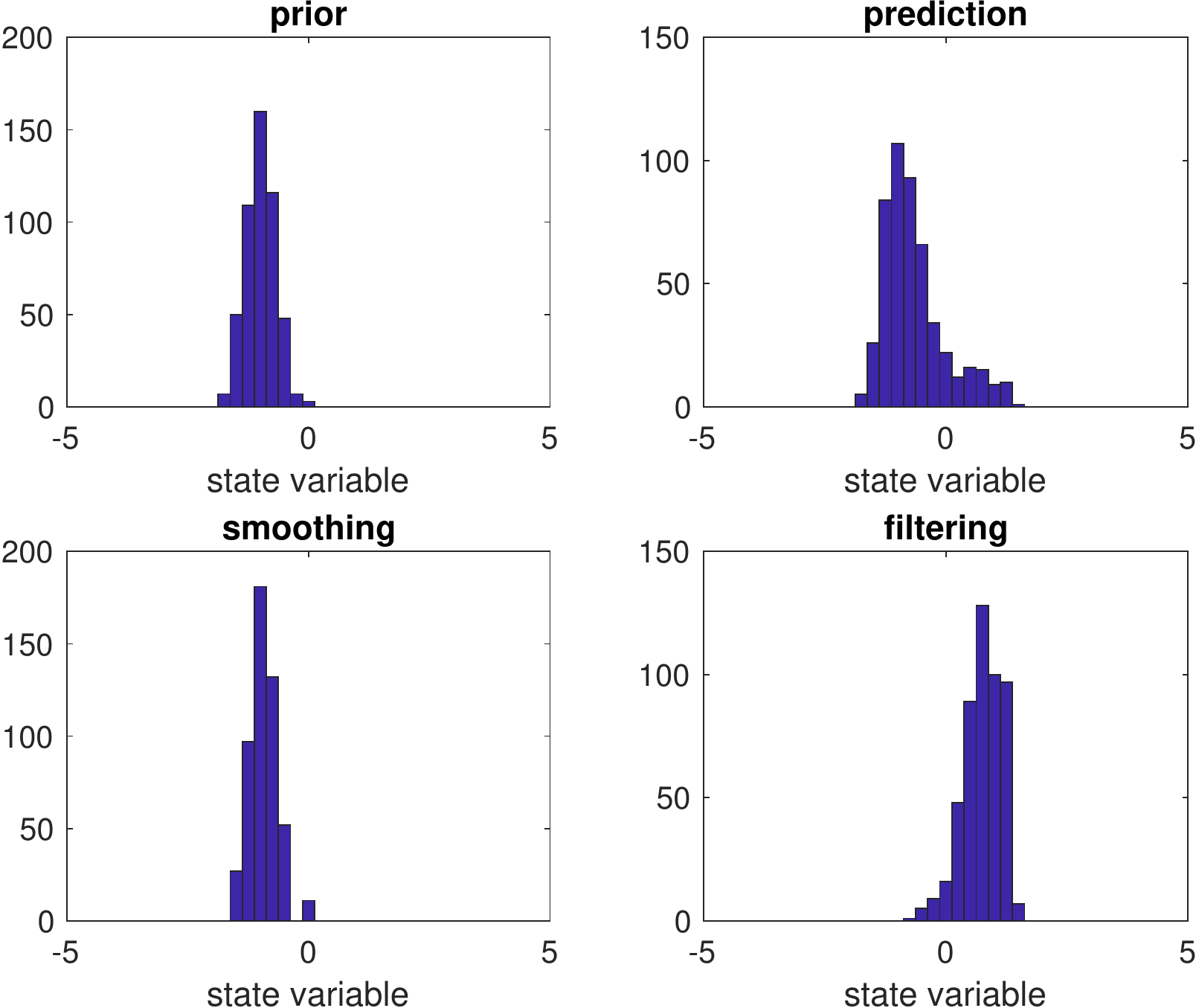}
\end{center}
\caption{Histograms produced from $M=200$ Monte Carlo samples of the initial PDF $\pi_0$, the forecast PDF
$\pi_2$ at time $t=2$, the filtering distribution $\widehat{\pi}_2$ at time $t=2$, 
and the smoothing PDF $\widehat{\pi}_0$ at time $t=0$ for a Brownian particle
moving in a double well potential.}
\label{fig:ex2a}
\end{figure}

\begin{example} We consider scalar-valued motion of a Brownian particle in a bimodal potential, that is,
\begin{equation} \label{eq:ex2a}
\dd Z_t^+ = Z_t^+ \dd t - (Z_t^+)^3 \dd t + \gamma^{1/2} \dd W_t^+
\end{equation}
with $\gamma = 0.5$ and initial distribution $Z_0 \sim {\rm N}(-1,0.3)$.  At time $t=2$ we measure the location
$y=1$ with measurement error variance $R = 0.2$. We simulate the dynamics using $M=200$ particles and
a time-step of $\Delta t = 0.01$ in the Euler--Maruyama discretisation \R{eq:EMM1}. One can find histograms produced
from the Monte Carlo samples in Figure \ref{fig:ex2a}. The samples from the filtering and smoothing distributions are obtained
by resampling with replacement from the weighted distributions with weights given by \R{eq:importance_weights_filtering1}.
Next we compute \R{eq:SDE_transition} from the $M=200$ Monte Carlo samples of \R{eq:ex2a}. Eleven out of the $200$
transition kernels from $\pi_0$ to $\pi_2$ (prediction problem) and $\pi_0$ to $\widehat{\pi}_2$ (Schr\"odinger problem) 
are displayed in Figure \ref{fig:ex2b}.
\end{example}

\begin{figure}
\begin{center}
\includegraphics[width=0.45\textwidth]{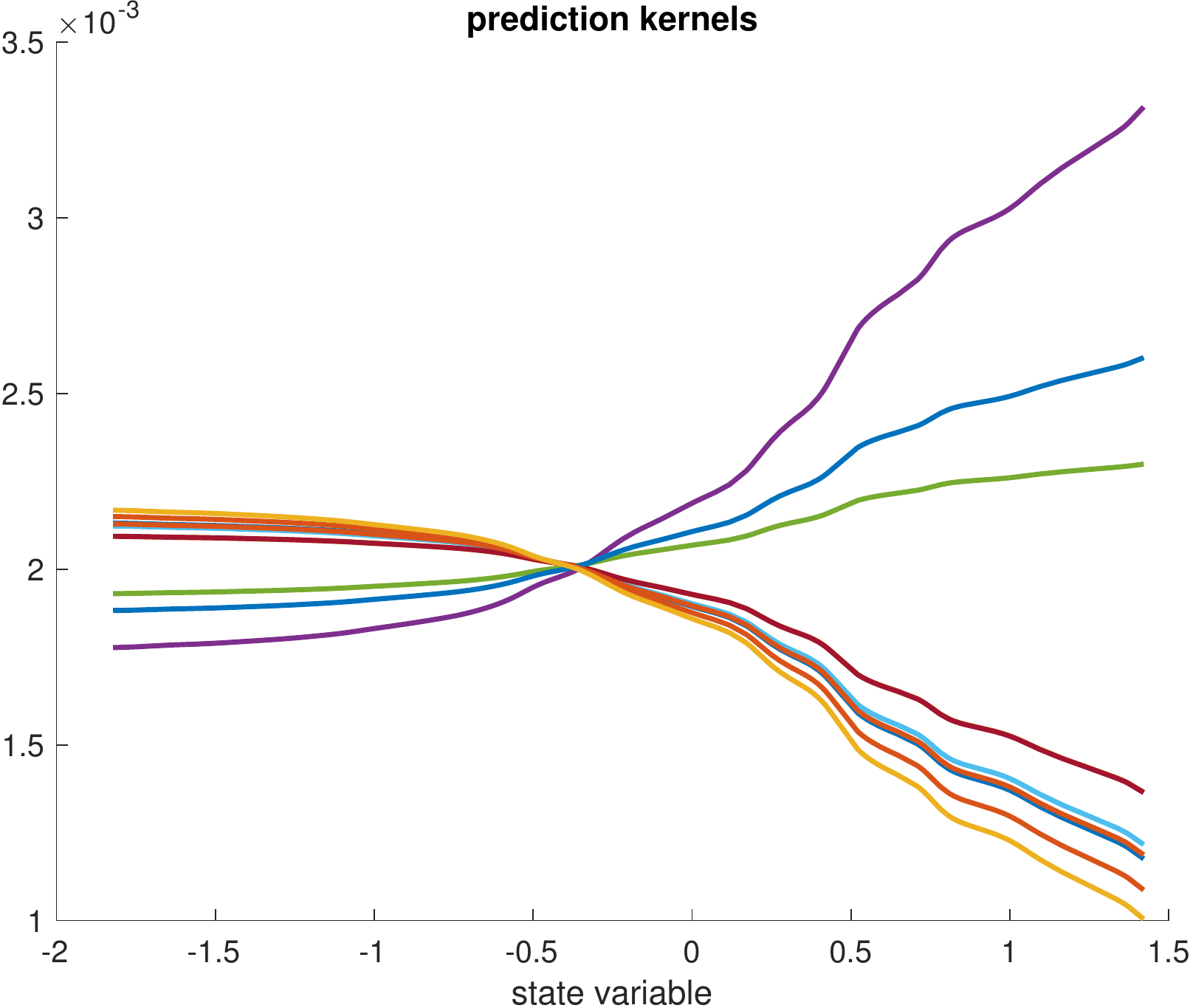}
$\quad$
\includegraphics[width=0.45\textwidth]{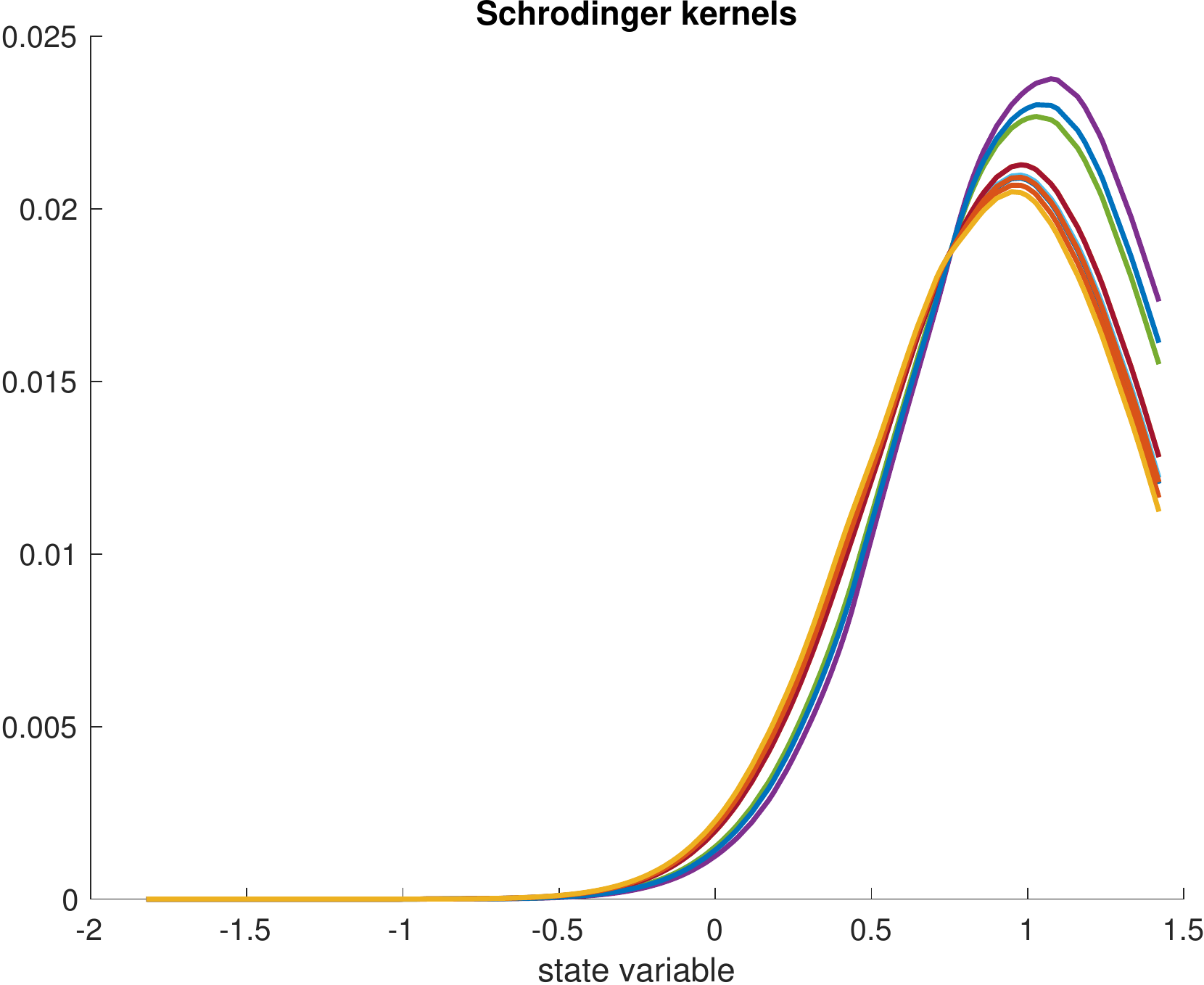}
\end{center}
\caption{Left panel: Approximations of typical transition kernels from time $\pi_0$ to $\pi_2$ under the Brownian dynamics 
model \R{eq:ex2a}. Right panel: Approximations of typical Schr\"odinger transition kernels from $\pi_0$ to $\widehat{\pi}_2$.
All approximations were computed using the Sinkhorn algorithm and by linear interpolation between the $M=200$ data points.}
\label{fig:ex2b}
\end{figure}

\medskip

\noindent
The Sinkhorn approach requires relatively large sample sizes $M$ in order to
lead to useful approximations. Alternatively we may
assume that there is an approximative control term $u^{(0)}_t$ with associated forward SDE
\begin{equation} \label{eq:C-SDE0}
\dd {Z}_t^+ = f_t({Z}_t)\,\dd t + 
u_t^{(0)}({Z}_t^+)\,\dd t + \gamma^{1/2} \,\dd W_t^+, \qquad t\in [0,1] ,
\end{equation}
and $Z_0^+ \sim \pi_0$. We denote the associated path measure by $\mathbb{Q}^{(0)}$. Girsanov's
theorem implies that the Radon--Nikodym derivative of $\mathbb{Q}$ with respect to $\mathbb{Q}^{(0)}$ is 
given by (compare \R{eq:RND}) 
\begin{equation*}
\frac{\dd \mathbb{Q}}{\dd \mathbb{Q}^{(0)}}_{|z_{[0,1]}^{(0)}} = \exp (-V^{(0)}) ,
\end{equation*}
where $V^{(0)}$ is defined via the stochastic integral
\begin{equation*}
V^{(0)} := \frac{1}{2\gamma} \int_0^1 \left(\|u_t^{(0)}\|^2 \dd t +  2\gamma^{1/2} u_t^{(0)}\cdot \dd W_t^+   \right)
\end{equation*}
along solution paths $z_{[0,1]}^{(0)}$ of \R{eq:C-SDE0}.
Because of 
\begin{equation*}
\frac{\dd \widehat{\mathbb{P}}}{\dd \mathbb{Q}^{(0)}}_{|z_{[0,1]}^{(0)}} = 
\frac{\dd \widehat{\mathbb{P}}}{\dd \mathbb{Q}}_{|z_{[0,1]}^{(0)}} \,
\frac{\dd \mathbb{Q}}{\dd \mathbb{Q}^{(0)}}_{|z_{[0,1]}^{(0)}}  \propto l(z_1^{(0)}) \,\exp(-V^{(0)})\,,
\end{equation*}
we can now use (\ref{eq:C-SDE0}) to importance-sample from the filtering PDF $\widehat{\pi}_1$. The control $u_t^{(0)}$
should be chosen such that the variance in the modified likelihood function
\begin{equation} \label{eq:modified_l}
l^{(0)}(z_{[0,1]}^{(0)}) := l(z_1^{(0)}) \,\exp (-V^{(0)})
\end{equation}
is reduced compared to the uncontrolled case $u_t^{(0)} \equiv 0$. In particular, the filter distribution $\widehat{\pi}_1$ 
at time $t=1$ satisfies 
\begin{equation*}
\widehat{\pi}_1(z_1^{(0)}) \propto l^{(0)}(z_{[0,1]}^{(0)})\,\pi_1^{(0)}(z_1^{(0)}),
\end{equation*}
where $\pi_t^{(0)}$, $t\in (0,1]$, denote the marginal PDFs generated by \R{eq:C-SDE0}.

We now describe an iterative algorithm for the associated Schr\"odinger problem in the spirit of the 
Sinkhorn iteration from Section \ref{sec:SP_DM}. 

\medskip

\begin{lemma} The desired optimal control law can be computed iteratively,
\begin{equation} \label{eq:iterated_control}
u^{(k+1)}_t = u^{(k)}_t + \gamma \nabla_z \log \psi_t^{(k)}, \qquad k=0,1,\ldots ,
\end{equation}
for given $u_t^{(k)}$ and the potential $\psi_t^{(k)}$ obtained as 
the solutions to the backward Kolmogorov equation
\begin{equation}\label{eq:iter_SP1}
\partial_t \psi_t^{(k)} = -{\cal L}^{(k)}_t \psi_t^{(k)}, \qquad {\cal L}^{(k)}_t g :=
\nabla_z g \cdot (f_t + u^{(k)}_t) + \Frac{\gamma}{2} \Delta_z g ,
\end{equation}
with final time condition
\begin{equation} \label{eq:IRF}
\psi_1^{(k)}(z):= \frac{\widehat{\pi}_1(z)}{\pi_1^{(k)}(z)} .
\end{equation}
Here $\pi_1^{(k)}$ denotes the time-one marginal of the path measure $\mathbb{Q}^{(k)}$ induced 
by \R{eq:Forward-SDE2} with control term $u_t = u^{(k)}_t$ and initial PDF $\pi_0$. 
The recursion \R{eq:iterated_control} is stopped whenever the final time condition (\ref{eq:IRF}) 
is sufficiently close to a constant function.
\end{lemma}

\begin{proof}
The extension of the Sinkhorn algorithm to continuous PDFs and its convergence has been discussed by \citeasnoun{sr:CGP15}.
\end{proof}

\medskip

\begin{remark} Note that $\psi_t^{(k)}$ needs to be 
determined up to a constant of proportionality only since the associated control law is determined from $\psi_t^{(k)}$ by 
\R{eq:iterated_control}. One can also replace \R{eq:iter_SP1} with any other method for solving
the smoothing problem associated to the SDE \R{eq:Forward-SDE2} with $Z_0^+\sim \pi_0$, 
control law $u_t = u_t^{(k)}$, and likelihood function $l(z) = \psi_1^{(k)}(z)$. 
See Appendix D for a forward--backward SDE formulation in particular. \end{remark}

\medskip

\noindent
We need to restrict the class of possible control laws $u_t^{(k)}$ in order to obtain a computationally feasible implementations
in practice. For example, a simple class of control laws is provided 
by linear controls of the form
\begin{equation*}
u_t^{(k)}(z) = -B_t^{(k)} (z-m_t^{(k)})
\end{equation*}
with appropriately chosen symmetric positive definite matrices $B_t^{(k)}$ and vectors $m_t^{(k)}$. Such approximations
can, for example, be obtained from the smoother extensions of ensemble transform methods mentioned earlier.
See also the recent work by \citeasnoun{sr:KR16} and \citeasnoun{sr:RK17} on numerical methods for 
the SDE smoothing problem.

%
%
%
\section{DA for continuous-time data} \label{sec:mf_ct_DA}
%
%
%
%

In this section, we focus on the continuous-time filtering problem over the time interval $[0,1]$, that is, on the
assimilation of data that arrive continuously in time. If one is only interested in transforming samples from the 
prior distribution at $t=0$ into samples of the filtering distribution at time $t=1$, then all methods from the previous sections 
can be applied once the associated filtering distribution $\widehat{\pi}_1$ is available. However, it is more natural to 
consider the associated filtering distributions $\widehat{\pi}_t$ for all $t \in (0,1]$ and to derive appropriate transformations 
in the form of mean-field equations in continuous time. We distinguish between smooth and 
non-smooth data $y_t$, $t\in [0,1]$.

\subsection{Smooth data}

We start from a forward SDE model \R{eq:Forward-SDE} with associated path
measure $\mathbb{Q}$ over the space of continuous functions ${\cal C}$. 
However, contrary to the previous sections, the likelihood, $l$, is defined along a whole solution
path $z_{[0,1]}$ as follows:
\begin{equation*}
\frac{\dd \widehat{\mathbb{P}}}{\dd \mathbb{Q}}_{|z_{[0,1]}} 
\propto l(z_{[0,1]}), \qquad l(z_{[0,1]}) := \exp \left( -\int_0^1 V_t(z_t)\,\dd t\right)
\end{equation*}
with the assumption that $\mathbb{Q}[l] < \infty$ and $V_t(z)\ge 0$. A specific example of a suitable $V_t$ is provided by
\begin{equation} \label{eq:loglikelihood}
V_t(z) = \frac{1}{2} \|h(z)-y_t\|^2 ,
\end{equation}
where the data function $y_t \in \mathbb{R}$, $t\in [0,1]$, is a smooth function of time and $h(z)$ is a forward operator connecting the
model states to the observations/data. The associated estimation problem has, for example, been addressed by \citeasnoun{sr:M68} 
and \citeasnoun{sr:F97} from an optimal control perspective and has led to what is called the minimum energy estimator. Recall that
the filtering PDF $\widehat{\pi}_1$ is the marginal PDF of $\widehat{\mathbb{P}}$ at time $t=1$.

The associated time-continuous smoothing/filtering problems are based on the time-dependent 
path measures $\widehat{\mathbb{P}}_t$ defined by
\begin{equation*}
\frac{\dd \widehat{\mathbb{P}}_t}{\dd \mathbb{Q}}_{|z_{[0,1]}} 
\propto l(z_{[0,t]}), \qquad l(z_{[0,t]}) := \exp \left( -\int_0^t V_s(z_s)\,\dd s\right)
\end{equation*}
for $t\in (0,1]$. We let  $\widehat{\pi}_t$ denote the marginal PDF of $\widehat{\mathbb{P}}_t$ at time $t$. Note that $\widehat{\pi}_t$
is the filtering PDF, that is, the marginal PDF at time $t$ conditioned on all the data available until time $t$. Also note that $\widehat{\pi}_t$ is
different from the marginal (smoothing) PDF of $\widehat{\mathbb{P}}$ at time $t$.

We now state a modified Fokker--Planck equation which describes the time evolution of the filtering PDFs $\widehat{\pi}_t$. 
\medskip

\begin{lemma}
The marginal distributions $\widehat{\pi}_t$ of $\widehat{\mathbb{P}}_t$ satisfy the modified Fokker--Planck equation
\begin{equation} \label{eq:FP-smooth-data}
\partial_t \widehat{\pi}_t = {\cal L}_t^\dagger \widehat{\pi}_t - \widehat{\pi}_t (V_t - \widehat{\pi}_t[V_t])
\end{equation}
with ${\cal L}_t^\dagger$ defined by \R{eq:FP_operator}. 
\end{lemma}

\begin{proof}
This can be seen by setting $\gamma = 0$ and $f_t \equiv 0$ in
\R{eq:Forward-SDE} for simplicity and by considering the incremental change of measure induced by the likelihood, that is,
\begin{equation*}
\frac{\widehat{\pi}_{t+\delta t}}{\widehat{\pi}_t} \propto  {\rm e}^{-V_t\delta t} \approx 1 - V_t \,\delta t ,
\end{equation*}
and taking the limit $\delta t \to 0$ under the constraint that $\widehat{\pi}_t[1] = 1$ is preserved.
\end{proof}

\medskip

\noindent
We now derive a mean-field interpretation of \R{eq:FP-smooth-data} and rewrite \R{eq:FP-smooth-data} in the form
\begin{equation} \label{eq:FP-smooth-data2}
\partial_t \widehat{\pi}_t = {\cal L}_t^\dagger \widehat{\pi}_t + \nabla_z \cdot(\widehat{\pi}_t \nabla_z \phi_t) ,
\end{equation}
where the potential $\phi_t:\mathbb{R}^{N_z}\to \mathbb{R}$ satisfies the elliptic PDE
\begin{equation} \label{eq:elliptic-smooth-data}
\nabla_z \cdot (\widehat{\pi}_t \nabla_z \phi_t) = - \widehat{\pi}_t (V_t - \widehat{\pi}_t[V_t]) .
\end{equation}

\medskip
\begin{remark} Necessary conditions  for the elliptic PDE \R{eq:elliptic-smooth-data} to be solvable
and to lead to bounded gradients $\nabla_z \phi$ for given $\pi_t$ have been discussed by \citeasnoun{sr:LMMR15}. It is an open problem
to demonstrate that continuous-time data assimilation problems actually satisfy such conditions.
\end{remark}

\medskip

\noindent
With \R{eq:FP-smooth-data2} in place, we formally obtain the mean-field equation
\begin{equation} \label{eq:controlled-SDE-smooth-data}
\dd Z_t^+ = \left\{ f_t(Z_t^+) - \nabla_z \phi_t(Z_t^+)\right\} \dd t + \gamma^{1/2} \,\dd W_t^+ ,
\end{equation} 
and the marginal distributions $\pi_t^u$ of this controlled SDE 
agree with the marginals $\widehat{\pi}_t$ of the path measures $\widehat{\mathbb{P}}_t$ at times $t \in (0,1]$. 

The control $u_t$ is not uniquely determined. For example, one can replace \R{eq:elliptic-smooth-data}
with
\begin{equation} \label{eq:elliptic-smooth-data2}
\nabla_z \cdot (\pi_t M_t \nabla_z \phi_t) = - \pi_t (V_t - \pi_t[V_t]) ,
\end{equation}
where $M_t$ is a symmetric positive definite matrix. More specifically, let us assume that $\pi_t$ is Gaussian with 
mean $\bar{z}_t$ and covariance matrix $P_t^{zz}$ and that $h(z)$ is linear, \ie, $h(z) = Hz$. Then
\R{eq:elliptic-smooth-data2} can be solved analytically for $M_t = P_t^{zz}$ with
\begin{equation*}
\nabla_z \phi_t(z) = \frac{1}{2} H^\T \left( Hz + H\bar{z}_t - 2y_t\right) .
\end{equation*}
The resulting mean-field equation becomes
\begin{equation} \label{eq:controlled-SDE-smooth-data2}
\dd Z_t^+ =  \left\{ f_t(Z_t^+) - \Frac{1}{2} 
P_t^{zz} H^\T (HZ_t^+ + H\bar{z}_t-2y_t) \right\} \dd t + \gamma^{1/2} \,\dd W_t^+ ,
\end{equation} 
which gives rise to the ensemble Kalman--Bucy filter upon Monte Carlo discretization \cite{sr:br11}. See Section
\ref{sec:NM4} below and Appendix C for further details.

\medskip

\begin{remark} \label{remark_opt}
The approach described in this subsection can also be applied to standard Bayesian inference without model dynamics.
More specifically, let us assume that we have samples $z_0^i$, $i=1,\ldots,M$, 
from a prior distribution $\pi_0$ which we would like to transform into
samples from a posterior distribution
\begin{equation*}
\pi^\ast(z) := \frac{l(z)\,\pi_0(z)}{\pi_0[l]}
\end{equation*}
with likelihood $l(z) = \pi(y|z)$. One can introduce a homotopy connecting $\pi_0$ with $\pi^\ast$, for example, via
\begin{equation} \label{eq:homotopy0}
\breve{\pi}_s(z) := \frac{l(z)^s\,\pi_0(z)}{\pi_0[l^s]}
\end{equation}
with $s\in [0,1]$. We find that
\begin{equation} \label{eq:homotopy1}
\frac{\partial \breve{\pi}_s}{\partial s} = \breve{\pi}_s \left( \log l - \breve{\pi}_s[\log l]\right) .
\end{equation}
We now seek a differential equation
\begin{equation} \label{eq:homotopy2}
\frac{\dd}{\dd s} \breve{Z}_s = u_s(\breve{Z}_s)
\end{equation}
with $\breve{Z}_0 \sim \pi_0$ such that its marginal distributions $\breve{\pi}_s$ satisfy
\R{eq:homotopy1} and, in particular $\breve{Z}_1 \sim \pi^\ast$. This condition together with Liouville's equation for the time evolution of
marginal densities under a differential equation \R{eq:homotopy2} leads to
\begin{equation} \label{eq:homotopy3}
-\nabla_z \cdot (\breve{\pi}_s \,u_s) = \breve{\pi}_s \left( \log l - \breve{\pi}_s[\log l]\right) .
\end{equation}
In order to define $u_s$ in \R{eq:homotopy2} uniquely, we make the {\it ansatz}
\begin{equation}\label{eq:homotopy4}
u_s(z) = -\nabla_z \phi_s(z)
\end{equation}
which leads to the elliptic PDE
\begin{equation} \label{eq:homotopy5}
\nabla_z \cdot (\breve{\pi}_s \,\nabla_z \phi_s) = \breve{\pi}_s \left(  \log l - \breve{\pi}_s[\log l]\right)
\end{equation}
in the potential $\phi_s$. The desired samples from ${\pi}^\ast$ are now obtained as the time-one solutions
of \R{eq:homotopy2} with \lq control law\rq \,\R{eq:homotopy4} satisfying \R{eq:homotopy5} and initial conditions
$z_0^i$, $i=1,\ldots,M$. There are many modifications of this basic procedure \cite{sr:daum11,sr:reich10,sr:marzouk11},
some of them leading to explicit expressions for \R{eq:homotopy2} such as Gaussian PDFs 
\cite{sr:br10b} and Gaussian mixture PDFs \cite{sr:reich11}. We finally mention that the limit 
$s\to \infty$ in \R{eq:homotopy0} leads, formally, to the PDF $\breve{\pi}_\infty = \delta (z-z_{\rm ML})$, where $z_{\rm ML}$ denotes
the minimiser of $V(z) = -\log \pi(y|z)$, that is, the maximum likelihood estimator, 
which we assume here to be unique, for example, $V$ is convex. In other words these homotopy methods 
can be used to solve optimisation problems via derivative-free mean-field equations and their interacting particle approximations. 
See, for example, \citeasnoun{sr:ZTM17} and \citeasnoun{sr:SS17} as well as Appendices A \& C for more details.
\end{remark}

\subsection{Random data} \label{sec:random_data}

We now replace \R{eq:loglikelihood} with an observation model of the form
\begin{equation*}
\dd Y_t = h(Z_t^+)\,\dd t +  \dd V_t^+ ,
\end{equation*}
where we set $Y_t \in \mathbb{R}$ for simplicity and $V_t^+$ denotes standard Brownian motion.
The forward operator $h:\mathbb{R}^{N_z} \to \mathbb{R}$ is also assumed to be known. The marginal
PDFs $\widehat{\pi}_t$ for $Z_t$ conditioned on all observations $y_s$ with $s\in [0,t]$ satisfy the Kushner--Stratonovitch 
equation \cite{sr:jazwinski}
\begin{equation} \label{eq:KS}
\dd \widehat{\pi}_t = {\cal L}_t^\dagger \widehat{\pi}_t \,\dd t+ (h-\widehat{\pi}_t[h]) (\dd Y_t - \widehat{\pi}_t[h]\,\dd t)
\end{equation}
with ${\cal L}^\dagger$ defined by \R{eq:FP_operator}. The following observation is important for the subsequent discussion.

\medskip

\begin{remark}
Consider state-dependent diffusion 
\begin{equation} \label{eq:SDE_Strat}
\dd Z_t^+ = \gamma_t (Z_t^+) \circ \dd U_t^+ ,
\end{equation}
in its Stratonovitch interpretation \cite{sr:P14}, where $U_t^+$ is scalar-valued Brownian motion 
and $\gamma_t(z) \in \mathbb{R}^{N_z \times 1}$. Here the Stratonovitch interpretation is to be applied to the
implicit time-dependence of $\gamma_t(z)$ through $Z_t^+$ only, that is, the explicit time-dependence of $\gamma_t$ remains to
be It\^{o}-interpreted. The associated Fokker--Planck equation for the marginal
PDFs $\pi_t$ takes the form 
\begin{equation} \label{eq:FP_Stratonovitch}
\partial_t \pi_t = \frac{1}{2} \nabla_z \cdot ( \gamma_t \nabla_z \cdot (\pi_t \gamma_t ) ) 
\end{equation}
and expectation values $\bar g = \pi_t[g]$ evolve in time according to
\begin{equation} \label{eq:time_evolved_EV}
\pi_t[g] = \pi_0[g] + \int_0^t \pi_s[{\cal A}_tg]\,\dd s
\end{equation}
with operator ${\cal A}_t$ defined by
\begin{equation*}
{\cal A}_t g = \frac{1}{2}  \gamma_t^\T  \nabla_z  (\gamma_t^\T \nabla_z g) .
\end{equation*}
Now consider the mean-field equation
\begin{equation} \label{eq:reformulated_diffusion}
\frac{\dd}{\dd t} \widetilde{Z}_t = -\frac{1}{2} 
\gamma_t (\widetilde{Z}_t) \, J_t\,,\qquad J_t := \widetilde{\pi}_t^{-1} \nabla_z \cdot (\widetilde{\pi}_t \gamma_t) ,
\end{equation}
with $\widetilde{\pi}_t$ the law of $\widetilde{Z}_t$.
The associated Liouville equation is
\begin{equation*} 
\partial_t \widetilde{\pi}_t = \frac{1}{2}  \nabla_z \cdot (\widetilde{\pi}_t \gamma_t J_t) = 
\frac{1}{2} \nabla_z \cdot ( \gamma_t \nabla_z \cdot (\gamma_t \widetilde{\pi}_t) ) .
\end{equation*}
In other words, the marginal PDFs  and the associated expectation values  evolve identically under 
\R{eq:SDE_Strat} and \R{eq:reformulated_diffusion}, respectively.
\end{remark}

\medskip

\noindent
We now state a formulation of the continuous-time filtering problem in terms of appropriate mean-field equations. These
equations follow the framework of the feedback particle filter (FPF) as first introduced by \citeasnoun{sr:meyn13} and theoretically justified 
by \citeasnoun{sr:LMMR15}. See \citeasnoun{sr:crisan10} and \citeasnoun{sr:Xiong11} for an alternative formulation.

\medskip

\begin{lemma}
The mean-field SDE
\begin{equation} \label{eq:FPF1}
\dd Z_t^+ = f_t(Z_t^+)\,\dd t + \gamma^{1/2} \,\dd W_t^+ - K_t(Z_t^+) \circ \dd I_t
\end{equation}
with 
\begin{equation*}
\dd I_t := h(Z_t^+)\,\dd t - \dd Y_t + \dd U_t^+ , 
\end{equation*}
$U_t^+$ standard Brownian motion, and $K_t := \nabla_z \phi_t$, where the potential $\phi_t$ satisfies the elliptic PDE
\begin{equation} \label{eq:elliptic_FPF}
\nabla_z \cdot (\pi_t \nabla_z \phi_t) = -\pi_t (h-\pi_t[h]) ,
\end{equation}
leads to the same evolution of its conditional marginal distributions $\pi_t$ as \R{eq:KS}. 
\end{lemma}

\medskip

\begin{proof}
We set $\gamma = 0$ and $f_t \equiv 0$ in \R{eq:FPF1} for
simplicity. Then, following \R{eq:FP_Stratonovitch} with $\gamma_t = K_t$, 
the Fokker-Planck equation for the marginal distributions $\pi_t$ of \R{eq:FPF1} conditioned on
$\{Y_s\}_{s\in [0,t]}$ is given by
\begin{align} \label{eq:FPP_proof1}
\dd \pi_t &= \nabla_z \cdot( \pi_t K_t(h(z)\,\dd t-\dd Y_t)) + \nabla_z \cdot ( K_t \nabla_z \cdot (\pi_t K_t))\,\dd t\\ 
\nonumber
&=  (\pi_t[h]\,\dd t- \dd Y_t)\,\nabla_z \cdot (\pi_t K_t)  + \nabla_z \cdot( \pi_t K_t(h(z)-\pi_t[h])\,\dd t \\ \label{eq:FPP_proof2}
& \qquad \qquad + \,\,\nabla_z \cdot 
( K_t \nabla_z \cdot (\pi_t K_t))\,\dd t\\
\label{eq:FPP_proof3}
&= \pi_t (h-\pi_t[h]) (\dd Y_t - \pi_t[h]\,\dd t) ,
\end{align}
as desired, where we have used \R{eq:elliptic_FPF} twice to get from \R{eq:FPP_proof2} to
\R{eq:FPP_proof3}. Also note that both $Y_t$ and $U_t^+$ contributed to the diffusion-induced final term in 
\R{eq:FPP_proof1} and hence the factor $1/2$ in \R{eq:FP_Stratonovitch} is replaced by one.
\end{proof}

\medskip

\begin{remark}
Using the reformulation \R{eq:reformulated_diffusion} of \R{eq:SDE_Strat} 
in Stratonovitch form with $\gamma_t = K_t$ together
with \R{eq:elliptic_FPF},  one can replace $K_t \circ {\rm d}U_t^+$ with
$\Frac{1}{2} K_t (\pi_t[h] - h)\,\dd t$, which leads to the alternative 
\begin{equation*}
\dd I_t = \frac{1}{2} (h+\pi_t[h])\,\dd t - \dd Y_t
\end{equation*}
for the innovation $I_t$, as originally proposed by \citeasnoun{sr:meyn13} in their FPF formulation. We also note
that the feedback particle formulation (\ref{eq:FPF1}) can be extended to systems for which the measurement and model errors 
are correlated. See \citeasnoun{sr:NRR19} for more details.
\end{remark}

\medskip

\noindent
The ensemble Kalman--Bucy filter
\cite{sr:br11} with the Kalman gain factor $K_t$ being independent of the state variable
$z$ and of the form
\begin{equation} \label{eq:gain_KBF}
K_t = P_t^{zh}
\end{equation}
can be viewed as a special case of an FPF. Here $P_t^{zh}$ denotes the covariance matrix 
between $Z_t$ and $h(Z_t)$ at time $t$.

\section{Numerical methods} \label{sec:NM4} 

In this section, we discuss some numerical implementations of the mean field approach to continuous-time
data assimilation. An introduction to standard particle filter implementations can, for example, be found in 
\citeasnoun{sr:crisan}. We start with the continuous-time formulation of the ensemble Kalman filter and state 
a numerical implementation of the FPF using a Schr\"odinger formulation in the second part of this section. 
See also Appendix A for some more details on a particle-based solution of the elliptic PDEs
\R{eq:elliptic-smooth-data}, \R{eq:homotopy5}, and \R{eq:elliptic_FPF}, respectively.

\subsection{Ensemble Kalman--Bucy filter} \label{sec:EnKBF}
Let us start with the ensemble Kalman--Bucy filter (EnKBF), which arises naturally from the mean-field 
equations \R{eq:controlled-SDE-smooth-data2} and \R{eq:FPF1}, respectively, with Kalman gain
\R{eq:gain_KBF} \cite{sr:br11}. We state the EnKBF here in the form
\begin{equation} \label{eq:EnKBF2}
{\rm d}Z_t^i = f_t(Z_t^i)\dd t + \gamma^{1/2}\dd W_t^+ - K_t^M \,\dd I_t^i
\end{equation}
for $i=1,\ldots,M$ and
\begin{equation*}
K_t^M := \frac{1}{M-1} \sum_{i=1}^M Z_t^i (h(Z_t^i) - \bar h^M_t)^\T, \qquad 
\bar h_t^M := \frac{1}{M} \sum_{i=1}^M h(Z_t^i) .
\end{equation*}
The innovations $\dd I_t^i$ take different forms depending on whether the data are smooth in time, that is,
\begin{equation*}
\dd I_t^i = \frac{1}{2} \left( h(Z_t^i) + \bar h_t^M - 2y_t\right) \dd t ,
\end{equation*}
or contains stochastic contributions, that is,
\begin{align}
\dd I_t^i &= \frac{1}{2} \left( h(Z_t^i) + \bar h_t^M \right) \dd t - \dd y_t ,
\end{align}
or, alternatively,
\begin{equation*}
\dd I_t^i = h(Z_t^i) \dd t + \dd U_t^i - \dd y_t ,
\end{equation*}
where $U_t^i$ denotes standard Brownian motion. The SDEs \R{eq:EnKBF2} can be discretised in time by
any suitable time-stepping method such as the Euler--Maruyama scheme \cite{sr:Kloeden}. However, one has
to be careful with the choice of the step-size $\Delta t$ due to potentially stiff contributions from $K_t^M \,\dd I_t^i$.
See, for example, \citeasnoun{sr:akir11} and \citeasnoun{sr:BSW18}.

\medskip

\begin{remark}
It is of broad interest to study the stability and accuracy of interacting particle filter algorithms such as the discrete-time
EnKF and the continuous-time EnKBF for fixed particle numbers $M$. On the negative side, it has been shown by 
\citeasnoun{sr:KMT15} that such algorithms can undergo finite-time instabilities while it has also been demonstrated 
\cite{sr:hunt13,sr:KellyEtAl14,sr:majda15,sr:dWRS18} that such algorithms can be stable and accurate under appropriate 
conditions on the dynamics and measurement process. Asymptotic properties of the EnKF and EnKBF in the limit of
$M\to \infty$ have also been studied, for example, by \citeasnoun{sr:legland11}, \citeasnoun{sr:KM15}, 
and \citeasnoun{sr:dWRS18}.
\end{remark}

\subsection{Feedback particle filter}

A Monte Carlo implementation of the FPF \R{eq:FPF1} faces two main obstacles. First, one needs to approximate
the potential $\phi_t$ in \R{eq:elliptic_FPF} with the density $\pi_t$, which is only available in terms of an empirical measure
\begin{equation*}
\pi_t(z) = \frac{1}{M} \sum_{i=1}^M \delta (z-z_t^i) .
\end{equation*}
Several possible approximations have been discussed by \citeasnoun{sr:TM16} and \citeasnoun{sr:TdWMR17}.
Here we would like to mention in particular an approximation based on diffusion maps which we summarise in Appendix A.
Second, one needs to apply suitable time-stepping methods for the SDE \R{eq:FPF1} in Stratonovitch form. Here we 
suggest using the Euler--Heun method \cite{sr:BBT04}
\begin{align*}
\tilde z_{n+1}^i &= z_n^i + \Delta t f_{t_n}(z_n^i) + (\gamma \Delta t)^{1/2} \xi_n^i - K_n(z_n^i) \Delta I_n^i,\\ \nonumber
z_{n+1}^i &= z_n^i + \Delta t  f_{t_n}(z_n^i)  + 
(\gamma \Delta t)^{1/2} \xi_n^i  -\frac{1}{2} \left( K_{n}(z_n^i) + K_{n}(\tilde z_{n+1}^i) \right)  \Delta I_n^i,
\end{align*}
$i=1,\ldots,M$, with, for example, 
\begin{equation*}
\Delta I_n^i = \frac{1}{2}\left( h(z_n^i) + \bar h_n^M\right)\Delta t - \Delta y_n.
\end{equation*}

While the above implementation of the FPF requires one to solve the elliptic PDE \R{eq:elliptic_FPF} twice per time-step
we now suggest a time-stepping approach in terms of an associated Schr\"odinger problem. Let us assume that
we have $M$ equally weighted particles $z_n^i$ representing the conditional filtering distribution
at time $t_n$. We first propagate these particles forward under the drift term alone, that is,
\begin{equation*}
\hat z^i_{n+1} := z_n ^i + \Delta t\,f_{t_n}(z_n^i),\qquad i=1,\ldots,M .
\end{equation*}
In the next step, we draw $L = KM$ with $K\ge 1$ samples $\widetilde{z}^l_{n+1}$ from the forecast PDF
\begin{equation*}
\widetilde{\pi}(z) := \frac{1}{M} \sum_{i=1}^M {\rm n}(z;\hat z^i_{n+1},\gamma \Delta t I)
\end{equation*}
and assign importance weights
\begin{equation*} 
w^l_{n+1} \propto \exp \left( -\frac{\Delta t}{2}(h(\widetilde{z}_{n+1}^l))^2 + \Delta y_n h(\widetilde{z}^l_{n+1})\right)
\end{equation*}
with normalisation \R{eq:IW_normalised}. Recall that we assumed that $y_t \in \mathbb{R}$ for simplicity 
and $\Delta y_n := y_{t_{n+1}}-y_{t_n}$. We then solve the Schr\"odinger problem 
\begin{equation}\label{eq:FPF4}
P^\ast = \arg \min_{P\in \Pi_{\rm M}} {\rm KL}(P||Q)
\end{equation}
with the entries of $Q\in \mathbb{R}^{L\times M}$ given by
\begin{equation*}
q_{li} = \exp \left(-\frac{1}{2\gamma \Delta t} \|\widetilde{z}^l_{n+1} - \hat z^i_{n+1}\|^2 \right)
\end{equation*}
and the set $\Pi_{\rm M}$ defined by \R{eq:SS_cond3}. The desired particles $z_{n+1}^i = Z_{n+1}^i(\omega)$ are finally given as realisations of
\begin{equation} \label{eq:FPF5}
Z_{n+1}^i = \sum_{l=1}^L \widetilde{z}^l_{n+1} p^\ast_{li} + (\gamma \Delta t)^{1/2} \Xi_n^i\,, \quad \Xi_n^i \sim {\rm N}(0,I) ,
\end{equation}
for $i=1,\ldots,M$. 

The update \R{eq:FPF5}, with $P^\ast$ defined by \R{eq:FPF4}, can be viewed as data-driven drift correction 
combined with a standard approximation to the Brownian diffusion part of the underlying SDE model.  
It remains to be investigated in what sense \R{eq:FPF5} can be viewed as an approximation to the
FPF formulation \R{eq:FPF1} as $M\to \infty$ and $\Delta t \to 0$.

\medskip

\begin{remark}
One can also use the matrix $P^\ast$ from \R{eq:FPF4} to implement a resampling scheme 
\begin{equation} \label{eq:FPF6}
\mathbb{P}[Z_{n+1}^i(\omega) = \widetilde{z}^l_{n+1}] = p^\ast_{li}
\end{equation}
for $i=1,\ldots,M$. Note that, contrary to classical resampling schemes based on weighted particles $(\widetilde z^l_{n+1},w^l_{n+1})$,
$l=1,\ldots,L$, the sampling probabilities $p^\ast_{li}$ take into account the underlying geometry of the forecasts $\hat z^i_{n+1}$ in state
space.
\end{remark}

\begin{figure} \label{fig:figureL63}
\begin{center}
\includegraphics[width = 0.8\textwidth,trim = 0 0 0 0 0,clip]{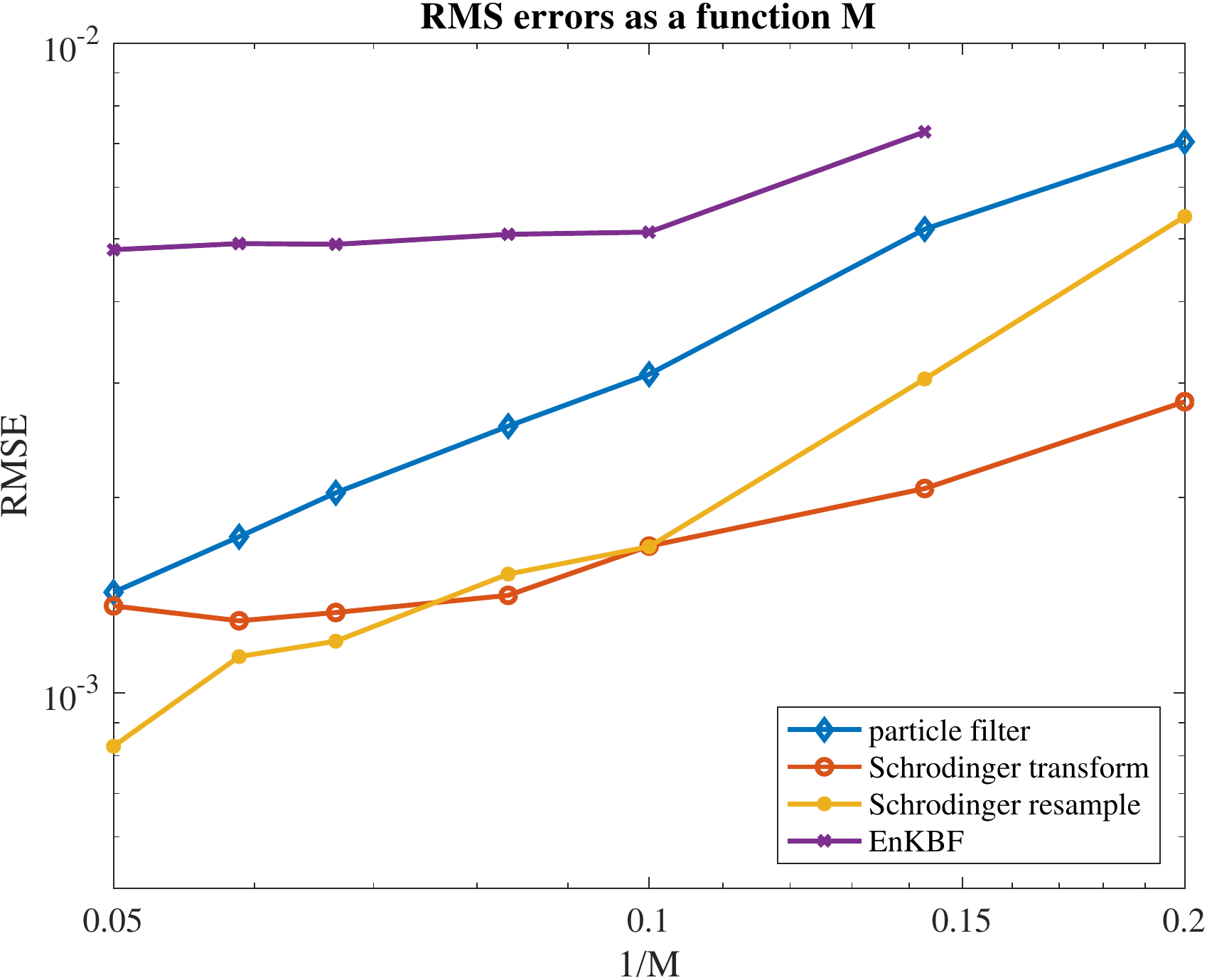}
\end{center}
\caption{RMS errors as a function of sample size, $M$, for a standard particle filter, the EnKBF, and implementations of
\R{eq:FPF5} (Schr\"odinger transform) and\R{eq:FPF6} (Schr\"odinger resample), respectively. Both Schr\"odinger-based
methods outperform the standard particle filter for small ensemble sizes. The EnKBF diverged for the smallest ensemble
size of $M=5$ and performed worse than all other methods for this highly nonlinear problem.}
\end{figure}

\begin{example}
We consider the SDE formulation 
\begin{equation*}
\dd Z_t = f(Z_t)\dd t + \gamma^{1/2} \dd W_t
\end{equation*}
of a stochastically perturbed Lorenz--63 model \cite{sr:lorenz63,sr:reichcotter15,sr:stuart15} with
diffusion constant $\gamma = 0.1$. The system is fully observed according to
\begin{equation*}
\dd Y_t = f(Z_t)\dd t + R^{1/2} \dd V_t
\end{equation*}
with measurement error variance $R=0.1$ and the system is simulated over a time interval $t\in [0,40\,000]$
with step--size $\Delta t = 0.01$. We implemented a standard particle filter with resampling performed after each time-step and
compared the resulting RMS errors with those arising from using \R{eq:FPF5} (Schr\"odinger transform) and
\R{eq:FPF6} (Schr\"odinger resample), respectively. See Figure \ref{fig:figureL63}. It can be seen that the Schr\"odinger-based 
methods outperform the standard particle filter in terms of RMS errors for small ensemble sizes. The Schr\"odinger transform
method is particularly robust for very small ensemble sizes while Schr\"odinger resample performs better at larger sample
sizes. We also implemented the EnKBF \R{eq:EnKBF2} and found that it diverged for the smallest ensemble size of $M=5$
and performed worse than the other methods for larger ensemble sizes.
\end{example}

%
\section{Conclusions}
%

We have summarised sequential data assimilation techniques suitable for state estimation of
discrete- and continuous-time stochastic processes. In addition to algorithmic approaches based on the standard filtering and 
smoothing framework of stochastic analysis, we have drawn a connection to a boundary value problem over joint 
probability measures first formulated by Erwin Schr\"odinger. We have argued that sequential data assimilation
essentially needs to approximate such a boundary value problem with the boundary conditions given by the
filtering distributions at consecutive observation times. 

Application of these techniques to high-dimensional problems arising, for example, from the spatial discretisation of
PDEs requires further approximations in the form of localisation and inflation, which we have not discussed in this
survey. See, for example, \citeasnoun{sr:evensen}, \citeasnoun{sr:reichcotter15}, and \citeasnoun{sr:ABN16} 
for further details. In particular, the localisation framework for particle filters as 
introduced by \citeasnoun{sr:reich15} and \citeasnoun{sr:reichcotter15} in the context of scenario (A) 
could be generalised to scenarios (B) and (C) from Definition \ref{ps:scenarios}.

Finally, the approaches and computational techniques discussed in this paper are also relevant to combined state and
parameter estimation.

\medskip

\noindent
{\bf Acknowledgement.} This research has been partially funded by 
Deutsche Forschungsgemeinschaft (DFG) through grant 
CRC 1294 \lq Data Assimilation\rq . Feedback on earlier versions of this paper
by Nikolas Kantas, Prashant Mehta, and Tim Sullivan have been invaluable.

%
%

\nocite*
\bibliography{bib_sreich}

%
\section{Appendices}
%
%
\subsection*{Appendix A. Mesh--free approximations to Fokker--Planck and backward Kolmogorov equations}
%

In this appendix, we discuss two closely related approximations, first to the Fokker--Planck equation
\R{eq:FPE1} with the (time-independent) operator \R{eq:FP_operator} taking the special form
\begin{equation*}
{\cal L}^\dagger \pi = - \nabla_z \cdot (\pi \,\nabla_z \log \pi^\ast) + \Delta_z \pi = \nabla_z \cdot \left(\pi^\ast \nabla_z \frac{\pi}{\pi^\ast}\right)
\end{equation*}
and, second, to its adjoint operator ${\cal L}$ given by \R{eq:diffusion_map1}.

The approximation to the Fokker--Planck equation \R{eq:FPE1} with drift term 
\begin{equation} \label{eq:sampling_drift}
f_t(z) = \nabla_z \log \pi^\ast (z)
\end{equation}
can be used to transform samples $x_0^i$, $i=1,\ldots,M$ from a (prior) PDF $\pi_0$ into
samples from a target (posterior) PDF $\pi^\ast$ using an evolution equation of the form
\begin{equation} \label{eqA:evolution1}
\frac{\dd}{\dd s} \breve{Z}_s = F_s (\breve{Z}_s) ,
\end{equation}
with $\breve{Z}_0 \sim \breve{\pi}_0 := \pi_0$ such that 
\begin{equation*}
\lim_{s \to \infty} \breve{Z}_s \sim \pi^\ast .
\end{equation*}
The evolution of the marginal PDFs $\breve{\pi}_s$ is given by Liouville's equation
\begin{equation} \label{eq:Liouville2}
\partial_s \breve{\pi}_s = -\nabla_z \cdot (\breve{\pi}_s F_s) .
\end{equation}
We now choose $F_s$ such that the Kullback--Leibler divergence
${\rm KL}\,(\breve{\pi}_s ||\pi^\ast)$ is non-increasing in time, that is,
\begin{equation} \label{eq:KL_decay}
\frac{\dd}{\dd s} {\rm KL}\,(\breve{\pi}_s ||\pi^\ast) = \int \breve{\pi}_s \left\{ F_s \cdot \nabla_z \log \frac{\breve{\pi}_s}{\pi^\ast} \right\}
\,\dd z \le 0 .
\end{equation}
A natural choice is
\begin{equation*}
F_s(z) := -\nabla_z \log \frac{\breve{\pi}_s}{\pi^\ast}(z) ,
\end{equation*}
which renders \R{eq:Liouville2} formally equivalent to the Fokker--Planck equation \R{eq:FPE1} with
drift term \R{eq:sampling_drift} \cite{sr:reichcotter15,sr:PC18}.

Let us now approximate the evolution equation \R{eqA:evolution1} over a reproducing kernel Hilbert space
(RKHS) ${\cal H}$ with kernel $k(z-z')$ and inner product $\langle f,g\rangle_{\cal H}$, which satisfies the reproducing
property
\begin{equation} \label{eq:rkp}
\langle k(\cdot - z'),f\rangle_{\cal H} = f(z') .
\end{equation}
Following \citeasnoun{sr:R90} and \citeasnoun{sr:DM90}, 
we first introduce the approximation
\begin{equation}\label{eqA:PDF}
\widetilde{\pi}_s(z) := \frac{1}{M} \sum_{i=1}^M k(z-z_s^i)
\end{equation}
to the marginal densities $\breve{\pi}_s$. Note that \R{eq:rkp} implies that
\begin{equation*}
\langle f,\widetilde{\pi}_s \rangle_{\cal H} = \frac{1}{M} \sum_{i=1}^M f(z_s^i) .
\end{equation*}
Given some evolution equations 
\begin{equation*}
\frac{\dd}{\dd s} z_s^i = u_s^i 
\end{equation*}
for the particles $z_s^i$, $i=1,\ldots,M$, we find that
\R{eqA:PDF} satisfies Liouville's equation, that is,
\begin{equation*}
\partial_s \widetilde{\pi}_s = -\nabla_z \cdot(\widetilde{\pi}_s \widetilde{F}_s)
\end{equation*}
with
\begin{equation*}
\widetilde{F}_s (z) = \frac{\sum_{i=1}^M k(z-z_s^i) \,u_s^i}{\sum_{i=1}^M k(z-z_s^i)} .
\end{equation*}
We finally introduce the functional
\begin{equation*}
{\cal V}(\{z_s^l\}) :=  \langle \widetilde{\pi}_s, \log \frac{\widetilde{\pi}_s}{\pi^\ast}\rangle_{\cal H}
= \frac{1}{M}\sum_{i=1}^M \log  \frac{\Frac{1}{M} \sum_{j=1}^M
k(z^i_s -z^j_s)}{\pi^\ast(z^i_s)}
\end{equation*}
as an approximation to the Kullback--Leibler divergence in the RKHS ${\cal H}$ and
set
\begin{equation} \label{eq:dFP}
u_s^i := -M \nabla_{z_s^i} {\cal V}(\{z_s^l\}) ,
\end{equation}
which constitutes the desired particle approximation to the Fokker--Planck equation 
\R{eq:FPE1} with drift term \R{eq:sampling_drift}. Time-stepping methods for such gradient flow systems 
have been discussed by \citeasnoun{sr:PR19}.

We also remark that an alternative interacting particle system, approximating the same asymptotic PDF $\pi^\ast$ in
the limit $s \to \infty$, has been proposed recently by \citeasnoun{sr:LW16} 
under the notion of Stein variational descent.  See \citeasnoun{sr:LLN18} for a theoretical analysis of Stein
variational descent, which implies in particular that Stein variational descent can be viewed as a Lagrangian particle 
approximation to the modified evolution equation
\begin{equation*}
\partial_s \breve{\pi}_s = \nabla_z \cdot (\breve{\pi}_s^2 \nabla_z \log \frac{\breve{\pi}_s}{\pi^\ast}(z)) =
\nabla_z \cdot \left(\breve{\pi}_s (\nabla_z \breve{\pi}_s - 
\breve{\pi}_s \nabla_z \log \pi^\ast)\right) 
\end{equation*}
in the marginal PDFs $\breve{\pi}_s$, that is, one uses
\begin{equation*}
F_s(z) := -\breve{\pi}_s \nabla_z \log \frac{\breve{\pi}_s}{\pi^\ast}(z)
\end{equation*}
in \R{eqA:evolution1}. The Kullback--Leibler divergence is still non-increasing since \R{eq:KL_decay} becomes 
\begin{equation*}
\frac{\dd}{\dd s} {\rm KL}\,(\breve{\pi}_s ||\pi^\ast) =   -\int \| F_s \|^2 \,\dd z \le 0 .
\end{equation*}
A numerical discretisation is obtained through the approximation
\begin{equation*}
F_s(z') \approx \int F_s(z)\,k(z-z')\,\dd z ,
\end{equation*}
\ie, one views the kernel $k(z-z')$ as a regularised Dirac delta function. This approximation leads to another vector field
\begin{align*}
\widehat{F}_s(z') &:= -\int \breve{\pi}_s (z) \left\{ \nabla_z \log \breve{\pi}_s(z) - \nabla_z \log \pi^\ast(z)\right\}  k(z-z')\,\dd z\\
&= \int \breve{\pi}_s(z) \left\{ \nabla_z k(z - z') + k(z - z') \,\nabla_z \log 
\pi^\ast (z)   \right\} \dd z .
\end{align*}
On extending the RKHS ${\cal H}$ and its reproducing property \R{eq:rkp} component-wise to vector-valued functions, it follows that
\begin{equation*}
\frac{\dd}{\dd s} {\rm KL}\,(\breve{\pi}_s ||\pi^\ast) =   -\int F_s \cdot \widehat{F}_s \,\dd z 
= -\langle \widehat{F}_s,\widehat{F}_s \rangle_{\cal H} \le 0 
\end{equation*}
along transformations induced by the vector field $\widehat{F}_s$. See
\citeasnoun{sr:LW16} for more details. One finally approximates the integral with respect to the PDF $\breve{\pi}_s$ 
by its empirical estimate using the
particles $z_s^i$, $i=1,\ldots,M$, which yields
\begin{equation*}
\widetilde{F}_s(z') := \frac{1}{M} \sum_{i=1}^M  \left\{ \nabla_z k(z^i_s - z') + k(z_s^i-z')\,\nabla_z \log 
\pi^\ast(z_s^i)  \right\} .
\end{equation*}

We now turn our attention to the dual operator ${\cal L}_t$, defined by \R{eq:diffusion_map1},
which also arises from \R{eq:elliptic-smooth-data2} and \R{eq:elliptic_FPF}, respectively. 
More specifically, let us rewrite \R{eq:elliptic_FPF}
in the form
\begin{equation} \label{eq:elliptic5}
{\cal A}_t \phi_t = -(h-\pi_t[h])
\end{equation}
with the operator ${\cal A}_t$ defined by
\begin{equation*}
{\cal A}_t g := \frac{1}{\pi_t} \nabla_z \cdot (\pi_t \nabla_z g) .
\end{equation*}
Then we find that ${\cal A}_t$ is of the form of ${\cal L}_t$ with $\pi_t$ taking the role of $\pi^\ast$.

We also recall that \R{eq:diffusion_map2} provides an approximation to ${\cal L}_t$ and hence to
${\cal A}_t$. This observation allows one to introduce a sample-based method for approximating the potential $\phi$ defined 
by the elliptic partial differential equation \R{eq:elliptic5}
for a given function $h(z)$. 

Here we instead follow the presentation of \citeasnoun{sr:TM16} and \citeasnoun{sr:TdWMR17} and
assume that we have $M$ samples $z^i$ from a PDF $\pi$. The method is based on
\begin{equation} \label{eq:prashant}
\frac{\phi - {\rm e}^{\epsilon {\cal A}} \phi}{\epsilon} \approx h - \pi[h]
\end{equation}
for $\epsilon>0$ sufficiently small and upon replacing ${\rm e}^{\epsilon {\cal A}}$ with a diffusion map 
approximation \cite{sr:H18} of the form
\begin{equation} \label{eq:approx}
{\rm e}^{\epsilon {\cal A}} \phi (z) \approx T_\epsilon \phi(z) := \sum_{i=1}^M k_\epsilon (z,z^i) \,\phi (z^i) .
\end{equation}
The required kernel functions $k_\epsilon (z,z^i)$ are defined as follows. Let 
\begin{equation*}
n_\epsilon (z) := {\rm n}(z;0,2\epsilon \,I)
\end{equation*}
and
\begin{equation*}
p_\epsilon (z) := \frac{1}{M} \sum_{j=1}^M n_\epsilon (z-z^j) = \frac{1}{M} \sum_{j=1}^M {\rm n}(z;z^j,2\epsilon \,I) .
\end{equation*}
Then
\begin{equation*}
k_\epsilon (z,z^i) := \frac{n_\epsilon (z-z^i)}{c_\epsilon (z)\, p_\epsilon (z^i)^{1/2}} 
\end{equation*}
with normalisation factor
\begin{equation*}
c_\epsilon (z) := \sum_{l=1}^M \frac{n_\epsilon (z-z^l)}{p_\epsilon (z^l)^{1/2}} .
\end{equation*}
In other words, the operator $T_\epsilon$ reproduces constant functions.

The approximations \R{eq:prashant} and \R{eq:approx} lead to the fixed-point problem\footnote{It would also be possible to 
employ the approximation \R{eq:diffusion_map2} in the fixed-point problem \R{eq:fixed_point_FPF}, that is, to replace
$k_\epsilon (z^j,z^i)$ by $(Q_+)_{ji}$ in \R{eq:diffusion_map2} with $\Delta t = \epsilon$ and $\pi^\ast = \pi_t$.}
\begin{equation} \label{eq:fixed_point_FPF}
\phi_j = \sum_{i=1}^M k_\epsilon (z^j,z^i) \, \phi_i  + \epsilon \Delta h_i  , \qquad 
j = 1,\ldots,M ,
\end{equation}
in the scalar coefficients $\phi_j$, $j=1,\ldots,M$, for given
\begin{equation*}
\Delta h_i := h(z^i)- \bar h, \qquad \bar h := \frac{1}{M} \sum_{l=1}^M h(z^l) .
\end{equation*}
Since $T_\epsilon$ reproduces constant functions, \R{eq:fixed_point_FPF} determines $\phi_i$ up to 
a constant contribution, which we fix by requiring
\begin{equation*}
\sum_{i=1}^M \phi_i = 0 .
\end{equation*}
The desired functional approximation $\widetilde{\phi}$ to the potential $\phi$ is now provided by
\begin{equation} \label{eq:hatpsi}
\widetilde{\phi}(z) = \sum_{i=1}^M k_\epsilon (z,z^i) \left\{ \phi_i + \epsilon \Delta h_i\right\} .
\end{equation}
Furthermore, since
\begin{align*}
\nabla_z k_\epsilon(z,z^i) &= \frac{-1}{2\epsilon} k_\epsilon (z,z^i) \left( (z-z^i) - \sum_{l=1}^M k_\epsilon (z,z^l) (z-z^l) \right)\\
&= \frac{1}{2\epsilon} k_\epsilon (z,z^i) \left( z^i - \sum_{l=1}^M k_\epsilon (z,z^l) z^l \right) ,
\end{align*}
we obtain
\begin{equation*}
\nabla_z \widetilde{\phi}(z^j) = \sum_{i=1}^M \nabla_z k_\epsilon (z^j,z^i)\,r_i = \sum_{i=1}^M  z^i \,a_{ij} ,
\end{equation*}
with
\begin{equation*}
r_i = \phi_i + \epsilon\,\Delta h_i
\end{equation*}
and
\begin{equation*}
a_{ij} := \frac{1}{2\epsilon} k_\epsilon(z^j,z^i) \left( r_i - \sum_{l=1}^M k_\epsilon (z^j,z^l)\,r_l \right) .
\end{equation*}
We note that
\begin{equation*}
\sum_{i=1}^M a_{ij} =0
\end{equation*}
and 
\begin{equation*}
\lim_{\epsilon\to \infty}  a_{ij} = \frac{1}{M} \Delta h_i 
\end{equation*}
since 
\begin{equation*}
\lim_{\epsilon \to \infty} k_\epsilon(z^j,z^i) = \frac{1}{M} .
\end{equation*}
In other words, 
\begin{equation*}
\lim_{\epsilon \to \infty} \nabla_z \widetilde{\phi}(z^j) = \frac{1}{M} \sum_{i=1}^M z^i \,(h(z^i)-\bar h) = K^M
\end{equation*}
independent of $z^j$, which is equal to an empirical estimator for the covariance between $z$ and $h(z)$ and which, 
in the context of the FPF, leads to the EnKBF formulations \R{eq:EnKBF2} of section \ref{sec:EnKBF}. 
See \citeasnoun{sr:TdWMR17} for more details and \citeasnoun{sr:TMM19} for a convergence analysis.

%
\subsection*{Appendix B. Regularized St\"ormer--Verlet for HMC}
%

One is often faced with the task of sampling from a high-dimensional PDF of the form
\begin{equation*}
\pi(x) \propto \exp (-V(x)), \qquad V(x) := \frac{1}{2} (x-\bar x)^\T B^{-1} (x-\bar x) + U(x) ,
\end{equation*}
for known $\bar x\in \mathbb{R}^{N_x}$, $B\in \mathbb{R}^{N_x\times N_x}$, and
$U: \mathbb{R}^{N_x} \to \mathbb{R}$. The hybrid Monte Carlo (HMC) method  \cite{sr:Neal,sr:Liu,sr:NSS18} has emerged as
a popular Markov chain Monte Carlo (MCMC) method for tackling this problem. HMC relies on a symplectic
discretization  of the Hamiltonian equations of motion
\begin{align*}
\frac{\dd}{\dd \tau} x &= M^{-1} p\,,\\
\frac{\dd}{\dd \tau} p &= -\nabla_x V(x) = -B^{-1} (x-\bar x) -\nabla_x U(x)
\end{align*}
in an artificial time $\tau$ \cite{sr:LeiRei04}.
The conserved energy (or Hamiltonian) is provided by
\begin{equation} \label{eq:energyHMC}
{\cal H}(x,p) = \frac{1}{2} p^\T M^{-1} p + V(x)\,.
\end{equation}
The symmetric positive definite mass matrix $M \in \mathbb{R}^{N_x\times N_x}$ 
can be chosen arbitrarily, and a natural choice in terms of sampling
efficiency is $M = B^{-1}$ \cite{sr:BPSSS11}. However, when also taking into account computational efficiency,
a St\"ormer--Verlet discretisation 
\begin{align} \label{eq:SV1a}
p_{n+1/2} &= p_n - \frac{\Delta \tau}{2} \nabla_x V(x_n),\\ \label{eq:SV1b}
q_{n+1} &= q_n + \Delta \tau \widetilde{M}^{-1} p_{n+1/2},\\
p_{n+1} &= p_{n+1/2} - \frac{\Delta \tau}{2} \nabla_x V(x_{n+1}), \label{eq:SV1c}
\end{align}
with step-size $\Delta \tau >0$, mass matrix $M=I$ in \R{eq:energyHMC} and modified mass matrix
\begin{equation} \label{eq:SV2}
\widetilde{M} = I + \frac{\Delta \tau^2}{4} B^{-1}
\end{equation}
in \R{eq:SV1b} emerges as an attractive alternative, since it implies 
\begin{equation*}
{\cal H}(x_n,p_n) = {\cal H}(x_{n+1},p_{n+1})
\end{equation*}
for all $\Delta \tau >0$ provided $U(x) \equiv 0$. The St\"ormer--Verlet formulation
\R{eq:SV1a}--\R{eq:SV1c} is based on a regularised formulation of Hamiltonian equations of motion
for highly oscillatory systems as discussed, for example, by \citeasnoun{sr:RH11}.

Energy-conserving time-stepping methods for linear
Hamiltonian systems have become an essential building block for applications of HMC to infinite-dimensional 
inference problems, where $B^{-1}$ corresponds to the discretisation of a positive, self-adjoint and 
trace-class operator ${\cal B}$. See, for example, \citeasnoun{sr:BGLFS17}. 

Note that the St\"ormer--Verlet discretization \R{eq:SV1a}--\R{eq:SV1c} together with \R{eq:SV2} can be
easily extended to inference problems with constraints $g(x) = 0$ \cite{sr:LeiRei04} 
and that \R{eq:SV1a}--\R{eq:SV1c} conserves equilibria\footnote{Note that equilibria of the Hamiltonian equations of motion correspond
to MAP estimators of the underlying Bayesian inference problem.}, that is, points $x_\ast$ with $\nabla V(x_\ast) = 0$, 
regardless of the step-size $\Delta \tau$.

HMC methods, based on \R{eq:SV1a}--\R{eq:SV1c} and \R{eq:SV2}, 
can be used to sample from the smoothing distribution of an SDE as considered in 
Sections \ref{sec:filtering_and_smoothing} and \ref{sec:num_smoothing}.


\subsection*{Appendix C. Ensemble Kalman filter}

We summarise the formulation of an ensemble Kalman filter in the form \R{eq:optimal_transport3}.
We start with the stochastic ensemble Kalman filter \cite{sr:evensen}, which is given by
\begin{equation}\label{eq:sEnKF}
\widehat{Z}_1^j = z_1^j - K(h(z_1^j) + \Theta^j -y_1 ), \qquad \Theta^j \sim {\rm N}(0,R),
\end{equation}
with Kalman gain matrix
\begin{equation*}
K = P^{zh}(P^{hh}+R)^{-1} = \frac{1}{M-1}\sum_{i=1}^M z_1^i\,(h(z_1^i)-\bar h)^\T \left(P^{hh}  +R\right)^{-1}
\end{equation*}
and
\begin{equation*}
P^{hh} := \frac{1}{M-1} \sum_{l=1}^M h(z_1^l)\,(h(z_1^l)-\bar h)^\T, \qquad
\bar h := \frac{1}{M}\sum_{l=1}^M h(z_1^l) .
\end{equation*}
Formulation \R{eq:sEnKF} can be rewritten in the form \R{eq:optimal_transport3}
with
\begin{equation} \label{eq:sEnKF2}
p_{ij}^\ast = \delta_{ij} - \frac{1}{M-1}(h(z_1^i)-\bar h)^\T \left(P^{hh}  +R\right)^{-1}(h(z_1^j)-y_1 + \Theta^j)\,,
\end{equation}
where $\delta_{ij}$ denotes the Kronecker delta, that is, $\delta_{ij} = 0$ if $i\not=j$ and $\delta_{ii}=1$.

More generally, one can think about ensemble Kalman filters and their generalisations
\cite{sr:anderson10} as first defining appropriate updates $\widehat{y}_1^i$
to the predicted $y_1^i = h(z_1^i)$ using the observed $y_1$, which is then extrapolated to the state variable $z$ 
via linear regression, that is
\begin{equation} \label{eq:enkf}
\widehat{z}_1^j = z_1^j + \frac{1}{M-1}\sum_{i=1}^M z_1^i(h(z_1^i)-\bar h)^\T (P^{hh})^{-1}\left(
\widehat{y}_1^j - y_1^j\right) ,
\end{equation}
which can be reformulated in the form \R{eq:optimal_transport3} \cite{sr:reichcotter15}. Note
that the consistency result
\begin{equation*}
H\widehat{z}_1^i = \widehat{y}_1^i
\end{equation*}
follows from \R{eq:enkf} for linear forward maps $h(z) = Hz$.

Within such a linear regression framework, one can easily derive ensemble transformations for the
particles $z_0^i$ at time $t=0$. We simply take the coefficients $p_{ij}^\ast$, as defined for example by an
ensemble Kalman filter \R{eq:sEnKF2}, and applies them to $z_0^i$, that is,
\begin{equation*}
\widehat{z}_0^j = \sum_{i=1}^M z_0^i \,p_{ij}^\ast .
\end{equation*}
These transformed particles can be used to approximate the smoothing distribution $\widehat{\pi}_0$.
See, for example, \citeasnoun{sr:evensen} and \citeasnoun{sr:KTAN17} for more details.

Finally, one can also interpret the ensemble Kalman filter as a continuous update in artificial time $s\ge 0$ of the form
\begin{equation} \label{eq:KB}
\dd z_s^i = -P^{zh}R^{-1}\dd I_s^i
\end{equation}
with the innovations $I_s^i$ given either by
\begin{equation} \label{eq:KB_innovation1}
\dd I_s^i = \frac{1}{2} \left( h(z_s^i) + \bar h_s\right)\dd s - y_1\dd s
\end{equation}
or, alternatively, by
\begin{equation*}
\dd I_s^i = h(z_s^i) \dd s + R^{1/2}\dd V_s^i  - y_1\dd s ,
\end{equation*}
where $V_s^i$ stands for standard Brownian motion \cite{sr:br10b,sr:reich10,sr:br11}. Equation \R{eq:KB}
with innovation \R{eq:KB_innovation1} can be given a gradient flow structure \cite{sr:br10b,sr:reichcotter15} of the
form
\begin{equation} \label{eq:KB_gradientflow}
\frac{1}{\dd s} \dd z_s^i = - P^{zz} \nabla_{z^i} {\cal V}(\{z_s^j\}),
\end{equation}
with potential
\begin{align*}
{\cal V}(\{z^j\}) &:= \frac{1-\alpha}{4} \sum_{j=1}^M (h(z^j)-y_1)^\T R^{-1} (h(z^j)-y_1) \,\,+ \\ & \qquad \qquad  \frac{(1+\alpha)M}{4}
(\bar h - y_1)^\T R^{-1} (\bar h-y_1)
\end{align*}
and $\alpha = 0$ for the standard ensemble Kalman filter, while $\alpha \in (0,1)$ can be seen as a form of variance inflation 
\cite{sr:reichcotter15}. 

A theoretical study of such dynamic formulations in the limit of $s\to \infty$ and $\alpha = -1$
has been initiated by \citeasnoun{sr:SS17}. There is an interesting link to stochastic gradient methods 
\cite{sr:BCN18}, which find application in situations where
the dimension of the data $y_1$ is very high and the computation of the complete gradient $\nabla_z h(z)$ becomes
prohibitive. More specifically, the basic concepts of stochastic gradient methods can be extended to \R{eq:KB_gradientflow}
if $R$ is diagonal, in which case one would pick at random paired components of $h$ and $y_1$ at the $k$th time-step 
of a discretisation of \R{eq:KB_gradientflow}, with the step-size $\Delta s_k$ chosen appropriately. 
Finally, we also point to a link between natural gradient methods and Kalman filtering \cite{sr:O17} which can be explored further in
the context of the continuous-time ensemble Kalman filter formulation \R{eq:KB_gradientflow}.


\subsection*{Appendix D: Numerical treatment of forward--backward SDEs}

We discuss a numerical approximation of the forward--backward SDE problem defined by the forward SDE
\R{eq:Forward-SDE} and the backward SDE \R{eq:BSDE} with initial condition $Z_0^+ \sim \pi_0$ at time $t=0$ and 
final condition $Y_1(Z_1^+) = l(Z_1^+)/\beta$ at time $t=1$. Discretisation of the forward SDE \R{eq:Forward-SDE} by
the Euler--Maruyama method \R{eq:EMM1} leads to $M$ numerical solution paths $z_{0:N}^i$, $i=1,\ldots,M$,
which, according to Definition \ref{def:sample_based2}, lead to $N$ discrete Markov transition matrices $Q^+_n \in
\mathbb{R}^{M\times M}$, $n = 1,\ldots,N$.

The  Euler--Maruyama method is now also applied to the backward SDE \R{eq:BSDE} and yields 
\begin{equation*}
Y_n = Y_{n+1}  - \Delta t^{1/2} \Xi_n^\T V_n .
\end{equation*}
Upon taking conditional expectation we obtain
\begin{equation} \label{eq:BEM1}
Y_n(Z_n^+) = \mathbb{E}\left[ Y_{n+1} | Z_n^+ \right] 
\end{equation}
and
\begin{equation*}
\Delta t^{1/2} \,\mathbb{E} \left[ \Xi_n \Xi_n^\T\right] V_n (Z_n^+) = \mathbb{E}\left[ (Y_{n+1}-Y_n) \Xi_n  | Z_n^+\right] ,
\end{equation*}
respectively. The last equation leads to
\begin{equation} \label{eq:BEM2}
V_n(Z_n^+) = \Delta t^{-1/2} \mathbb{E}\left[ (Y_{n+1}-Y_n) \Xi_n|Z_n^+\right] .
\end{equation} 
We also have $Y_N(Z_N^+) = l(Z_N^+)/\beta$ at final time $t = 1 = N\Delta t$. See page 45 in \citeasnoun{sr:C16} for more details.

We finally need to approximate the conditional expectation values in \R{eq:BEM1} and \R{eq:BEM2}, for which we employ the
discrete Markov transition matrix $Q_{n+1}^+$ and the discrete increments
\begin{equation*}
\zeta_{ij} := \frac{1}{(\gamma \Delta t)^{1/2}} \left( z_{n+1}^i - z_n^j - \Delta t f_{t_n}(z_n^j)\right) \in \mathbb{R}^M .
\end{equation*}
Given $y_{n+1}^j \approx Y(z_{n+1}^j)$ at time level $t_{n+1}$, we then approximate \R{eq:BEM1} by
\begin{equation} \label{eq:BEM1a}
y_n^j := \sum_{i=1}^M y_{n+1}^i  (Q^+_{n+1})_{ij}
\end{equation}
for $n = N-1,\ldots,0$. The backward iteration is initiated by setting $y_N^i = l(z_N^i)/\beta$, $i=1,\ldots,M$. 
Furthermore, a Monte Carlo approximation to \R{eq:BEM2} at $z_n^j$ is provided by
\begin{equation*}
\Delta t^{1/2} \,\sum_{i=1}^M \left\{ \xi_{ij} (\xi_{ij})^\T (Q^+_{n+1})_{ij}\right\} v_n^j = \sum_{i=1}^M  (y_{n+1}^i - y_n^j) \xi_{ij}  (Q^+_{n+1})_{ij} 
\end{equation*}
and, upon assuming invertibility, we obtain the explicit expression
\begin{equation} \label{eq:BEM2a}
 v_n^j :=  \Delta t^{-1/2}\,\left( \sum_{i=1}^M \xi_{ij} (\xi_{ij})^\T (Q^+_{n+1})_{ij} \right)^{-1} \sum_{i=1}^M  (y_{n+1}^i - y_n^j) \xi_{ij}  (Q^+_{n+1})_{ij} 
\end{equation}
for $n=N-1,\ldots,0$.

Recall from Remark \ref{rem:fbsde} that $y_n^j \in \mathbb{R}$ provides an approximation to $\psi_{t_n}(z_n^j)$ and $v_n^j \in \mathbb{R}^{N_z}$
an approximation to $\gamma^{1/2} \nabla_z \psi_{t_n}(z_n^j)$, respectively, where $\psi_t$ denotes the solution of the backward Kolmogorov equation 
\R{eq:BKE1} with final condition $\psi_1(z) = l(z)/\beta$. Hence, the forward solution paths $z_{0:N}^i$, $i=1,\ldots,M$, together with 
the backward approximations \R{eq:BEM1a} and \R{eq:BEM2a} provide a mesh-free approximation to the backward Kolmogorov equation
\R{eq:BKE1}. Furthermore, the associated control law \R{eq:optimal_control_SDE} can be approximated by
\begin{equation*}
u_{t_n}(z_n^i) \approx \frac{\gamma^{1/2}}{y_n^i} v_n^i .
\end{equation*}

The division by $y_n^i$ can be avoided by means of the following alternative formulation. We introduce the potential
\begin{equation} \label{eq:def_phi}
\phi_t := \log \psi_t\,,
\end{equation}
which satisfies the modified backward Kolmogorov equation
\begin{equation*}
0= \partial_t \phi_t + {\cal L}_t \phi_t + \frac{\gamma}{2} \|\nabla_z \phi_t\|^2
\end{equation*}
with final condition $\phi_1(z) = \log l(z)$, where we have
ignored the constant $\log \beta$. Hence It\^o's formula applied to $\phi_t(Z_t^+)$ leads to
\begin{equation} \label{eq:Lagrange2}
\dd \phi_t =  -\frac{\gamma}{2}  \|\nabla_z \phi_t \|^2 \dd t +
\gamma^{1/2} \nabla_z \phi_t \cdot \dd W_t^+
\end{equation}
along solutions $Z_t^+$ of the forward SDE \R{eq:Forward-SDE}. It follows from 
\begin{equation*} 
\frac{\dd \widehat{\mathbb{P}}}{\dd \mathbb{Q}^u}_{|z_{[0,1]}} =  \frac{l(z_1)}{\beta} \frac{\pi_0(z_0)}{q_0(z_0)} 
\exp \left(\frac{1}{2\gamma} \int_0^1 \left( \|u_t\|^2\,\dd t -
2 \gamma^{1/2} u_t \cdot \dd W_t^+ \right)\right) ,
\end{equation*}
with $u_t = \gamma \nabla_z \phi_t$, $q_0 = \widehat{\pi}_0$, and
\begin{equation*}
\log \frac{l(z_1)}{\beta}-\log \frac{\widehat{\pi}_0(z_0)}{\pi_0(z_0)} = \int_0^1 \dd \phi_t
\end{equation*}
that $\widehat{\mathbb{P}} = \mathbb{Q}^u$, as desired. 

The backward SDE associated with \R{eq:Lagrange2} becomes
 \begin{equation} \label{eq:BSDE3}
\dd Y_t = -\frac{1}{2} \|V_t\|^2 \dd t + V_t \cdot \dd  W_t^+
\end{equation}
and its Euler--Maruyama discretisation is
\begin{equation*}
Y_n = Y_{n+1} + \frac{\Delta t}{2} \|V_n\|^2 - \Delta t^{1/2} \Xi_n^\T V_n .
\end{equation*}
Numerical values $(y_n^i,v_n^i)$ can be obtained as before with \R{eq:BEM1a} replaced by
\begin{equation*} \
y_n^j := \sum_{i=1}^M \left( y_{n+1}^i + \frac{\Delta t}{2} \|v_n^j\|^2\right)  (Q^+_{n+1})_{ij}
\end{equation*}
and the control law \R{eq:optimal_control_SDE} is now approximated by 
\begin{equation*}
u_{t_n}(z_n^i) \approx \gamma^{1/2} v_n^i\,.
\end{equation*}

We re-emphasise that the backward SDE \R{eq:BSDE3} arises naturally from an optimal control perspective onto the smoothing problem. See 
\citeasnoun{sr:C16} for more details on the connection between optimal control and backward SDEs. In particular, this connection leads to
the following alternative approximation
\begin{equation*}
u_{t_n}(z_n^j) \approx   \sum_{i=1}^M z_{n+1}^i \left\{   (\widehat{Q}_{n+1}^+)_{ij}- (Q_{n+1}^+)_{ij} \right\}
\end{equation*}
of the control law \R{eq:optimal_control_SDE}. Here $\widehat{Q}_{n+1}^+$ denotes the twisted Markov transition matrix defined by
\begin{equation*}
\widehat{Q}_{n+1}^+ = D(y_{n+1})\,Q_{n+1}^+ \,D(y_n)^{-1}, \qquad y_n = (y_n^1,\ldots,y_n^M)^\T.
\end{equation*}

\medskip

\begin{remark}
The backward SDE \R{eq:BSDE3} can also be utilised to reformulate the Schr\"odinger system \R{eq:SS1a}--\R{eq:SS1d}. 
More specifically, one seeks an initial $\pi_0^\psi$ which evolves under the forward SDE  \R{eq:Forward-SDE} with $Z_0^+ \sim \pi_0^\psi$ 
such that the solution $Y_t$ of the associated backward SDE \R{eq:BSDE3} with final condition
\begin{equation*}
Y_1(z) = \log \widehat{\pi}_1(z) - \log \pi_1^\psi(z)
\end{equation*}
implies $\pi_0 \propto \pi_0^\psi \exp(Y_0)$. The desired control law in \R{eq:Forward-SDE2} is provided by
\begin{equation*}
u_t(z) = \gamma \nabla_z Y_t(z) = \gamma^{1/2} V_t(z) .
\end{equation*}
\end{remark}

\end{document}